\documentclass[11pt,final]{article}

\usepackage{amsmath, amscd, amsfonts, amssymb, amsthm, tikz, times, dsfont, enumerate, mathtools}
\usepackage[T1]{fontenc}
\usepackage{graphicx}
\usepackage[title]{appendix}
\usepackage{bbm}
\usepackage{kotex}
\usepackage{regexpatch}
\usepackage{subfig}
\usepackage{hyperref}
\usepackage{fullpage}

\makeatletter
\newcommand*{\rom}[1]{\expandafter\@slowromancap\romannumeral #1@}
\makeatother

\theoremstyle{plain}
\newtheorem{thm}{Theorem}[section]
\newtheorem{lem}[thm]{Lemma}
\newtheorem{cor}[thm]{Corollary}

\newtheorem{defi}[thm]{Definition}

\newtheorem*{remark*}{Remark}

\newcommand{\prob}{\mathbb{P}}
\newcommand{\PP}{\prob}

\newcommand{\R}{\mathbb{R}}
\newcommand{\C}{\mathbb{C}}
\newcommand{\E}{\mathbb{E}}
\newcommand{\T}{\mathbb{T}}
\newcommand{\eps}{\epsilon}
\newcommand{\indic}{\mathbbm{1}}
\newcommand{\IND}{\indic}
\newcommand{\nn}{\nonumber}
\newcommand{\ER}{Erd\H{o}s-R\'{e}nyi }

\newcommand{\hz}{\hat{z}}
\newcommand{\tz}{\tilde{z}}

\newcommand{\cE}{\mathcal{E}}
\newcommand{\cD}{\mathcal{D}}
\newcommand{\cY}{\mathcal{Y}}
\newcommand{\cX}{\mathcal{X}}
\newcommand{\cL}{\mathcal{L}}
\newcommand{\ac}{\mathring{a}}
\newcommand{\Ac}{\mathring{A}}

\newcommand{\ee}{\mathbf{e}}
\newcommand{\vv}{\mathbf{v}}
\newcommand{\uu}{\mathbf{u}}
\newcommand{\ww}{\mathbf{w}}
\newcommand{\xx}{\mathbf{x}}

\renewcommand{\Im }{\mathrm{Im}}

\newcommand{\var}{\mathrm{Var}}

\AtEndDocument{
\par
\medskip
\begin{tabular}{@{}l@{}}
	\textbf{Charles~Bordenave}\\
	{\small\textnormal{Institut de Math\'{e}matiques de Marseille; CNRS; Aix-Marseille Universit\'{e}, Marseille, 13288, France.}}\\
	{\small\textnormal{E-mail: \href{charles.bordenave@univ-amu.fr}{charles.bordenave@univ-amu.fr} }}
\end{tabular}

\par
\medskip
\begin{tabular}{@{}l@{}}
	\textbf{Jaehun~Lee}\\
	{\small\textnormal{Department of Mathematical Sciences, KAIST, Daejeon, 34141, Korea.}}\\
	{\small\textnormal{E-mail: \href{jaehun.lee@kaist.ac.kr}{jaehun.lee@kaist.ac.kr} }}
\end{tabular}
}

\date{}

\begin{document}

\title{Noise sensitivity for the top eigenvector of a sparse random matrix}
\author{Charles~Bordenave \qquad Jaehun~Lee}
\maketitle

\begin{abstract}
We investigate the noise sensitivity of the top eigenvector of a sparse random symmetric matrix. Let $v$ be the top eigenvector of an $N\times N$ sparse random symmetric matrix  with an average of $d$ non-zero centered entries per row. We resample $k$ randomly chosen entries of the matrix and obtain another realization of the random matrix with top eigenvector $v^{[k]}$. Building on recent results on sparse random matrices and a noise sensitivity analysis previously developed for Wigner matrices, we prove that, if $d\geq N^{2/9}$, with high probability, when $k \ll N^{5/3}$, the vectors $v$ and $v^{[k]}$ are almost collinear and, on the contrary, when $k\gg N^{5/3}$, the vectors $v$ and $v^{[k]}$ are almost orthogonal. A similar result  holds for the eigenvector associated to the second largest eigenvalue of the adjacency matrix of an Erd\H{o}s-R\'enyi random graph with average degree $d \geq N^{2/9}$.
\end{abstract}

\section{Introduction}

{\em Noise sensitivity} is an important phenomenon in probability theory that  describes a function of many independent random variables whose output asymptotically decorrelates when only a small proportion of the random variables are resampled. It has deep connections with threshold phenomena and it has been extensively studied since the pioneering work of Benjamini, Kalai, and Schramm \cite{BKS99}. It has found many applications in theoretical computer science and statistical mechanics where it commonly appears in large systems in a critical state such as critical percolation. We refer to the monographs \cite{MR3157205} and \cite{GS2015} for references and background.

Recently, the authors of \cite{BLZ20} have investigated the noise sensibility of Wigner random matrices, that is a $N \times N$ symmetric matrix with i.i.d. centered entries with unit variance above the diagonal. Calling an unit eigenvector corresponding to the largest eigenvalue {\em top eigenvector}, they studied how the direction of the top eigenvector varies when {\em resampling} a number $k$ of uniformly chosen entries of a Wigner matrix. Under an exponential tail assumption on the entries, they proved a threshold phenomenon as $N$ goes to infinity: if $k\ll N^{5/3}$, with high probability, the top eigenvectors remain nearly aligned while if $k \gg N^{5/3}$ their are almost orthogonal. Since $N^{5/3}$ is much smaller than  $N(N+1)/2$, the number of independent random variables in the matrix, the latter result can be interpreted as a noise sensitivity statement. On the random matrix side, the proofs in \cite{BLZ20} built on many outstanding  results which have been proved on the spacing and fluctuations of eigenvalues and on the delocalization of their eigenvectors, we refer to \cite{MR3792624, EY17} for lecture notes on this topic. 

In this paper, we extend the results of \cite{BLZ20} to a large class of {\em sparse symmetric random matrices} with an average of $d$ non-zero entries per row.  In the regime $d \geq N^\eps$ for some $\eps >0$, many remarkable results have recently been achieved for such sparse random matrices including eigenvector delocalization and Tracy-Widom or Gaussian fluctuation of the extremal eigenvalues, including \cite{BHY17,EKYY12,EKYY13,HKM18,HLY20,LS18,LL19,LV18,MR3579707}. One thus might expect to observe the same threshold phenomenon for the top eigenvector in sparse random matrix ensemble as it was shown for Wigner matrices. Indeed, we prove this phenomenon assuming a certain condition on the parameter $d$.  Our work notably builds upon \cite{EKYY13,HLY20} for local laws of the resolvent and \cite{HLY20,LL19} for eigenvalue spacings.\\

Sparse random matrices have many applications in computer sciences and statistics. One of canonical models for a such matrices is the sparse Erd\H{o}s-R\'{e}nyi graph, which is often used to describe random networks. In view of the graph, resampling an entry (of the adjacency matrix) can be regarded as creating or deleting an edge on the graph with some probability so that we can generate a random perturbation to some given networks through resampling. Since eigenvectors tend to be more informative than eigenvalue, it might be expected that the above-described phase transition of top eigenvector find some opportunities to be applied in other disciplines.

\subsection{Definition and main results}

We first introduce the main model of random matrices which we will consider. 
\begin{defi}[Sparse random matrices]\label{def: sparse RM}
Let $\vartheta >0$ be a fixed number and $q=q(N) \in (0,\sqrt N]$ be a sparsity parameter. Let $H=(h_{ij})$ be an $N\times N$ random matrix where all entries are real and independent up to the symmetry constraint $h_{ij}=h_{ji}$. 
We	assume that $h_{ij}$ is the product
	\begin{align*}
		h_{ij} = \frac{x_{ij}y_{ij}}{q},
	\end{align*}
	where $\{x_{ij}: i\le j\}$ and $\{y_{ij}: i\le j\}$ are independent and satisfy the following conditions: for all $i ,j$
	\begin{enumerate}[(i)]
		\item $\E x_{ij} = 0$, $\E x_{ij}^2 = 1$ and $\E \exp (\vartheta x_{ij}^2) \leq \vartheta^{-1}$.
		\item $\PP ( y_{ij} = 1) = 1 - \PP( y_{ij} = 0) =  q^{2}/N$.
	\end{enumerate}
\end{defi}

The condition $\E \exp (\vartheta x_{ij}^2) \leq \vartheta^{-1}$ asserts that the entries of the matrix are uniformly sub-Gaussian. Our condition ensures that 
$$
\E h_{ij} = 0 \quad \hbox{ and } \quad \E h_{ij} ^2 = \frac 1 N . 
$$ 
The order of magnitude of a non-zero entry is of order $1/q$. More precisely, for any integer $k \geq 1$, there exists a constant $C = C(k,\vartheta) \geq 1$ such that, 
\begin{equation}\label{eq:momenthij}
q ^{2 - 2k} N^{-1}  \leq \E h_{ij} ^{2k} \leq C q^{2 - 2k} N^{-1}. 
\end{equation}

In this paper, we will use the following notation in the asymptotic $N \to \infty$: The symbols $O(\cdot)$ and $o(\cdot)$ are used for the standard big-O and little-o notation. For nonnegative functions $f$ and $g$ of parameter $N$, we write $f\lesssim g$ if there exists a constant $C>0$ such that $f\le C g$, and $f\asymp g$ if $f\gtrsim g$ and $g\gtrsim f$. Finally, we use the less standard notation 
$
f \ll g
$
if there exists a constant $\eps >0$ such that $N^{\eps} f = O ( g)$. Beware that the underlying constants could depend implicitly on the parameter $\vartheta$ which is fixed throughout the paper.

We now describe the resampling procedure. Let $(i_k,j_k), 1 \leq k \leq N(N+1)/2,$ be a random uniformly chosen ordering of the set $S = \{ (i,j) : 1 \leq i \leq j \leq N\}$, independently of $H$. For a positive integer $k\le N(N+1)/2$, the set $S_{k}=\{(i_{1},j_{1}),\ldots, (i_{k},j_{k})\}$ is thus a random set of $k$ distinct pairs (with $i_{m}\le j_{m}$) which is chosen uniformly from the family of all sets of $k$ distinct elements in $S$. By convention $S_0$ is the empty set.
\begin{defi}[Resampling procedure]\label{def: resampling}
	 Let $H'=(h_{ij}')$ be an independent copy of $H$. For integer $0 \leq k \leq N(N+1)/2$, we define $H^{[k]}=(h_{ij}^{[k]})$ as the random symmetric matrix generated from the given random matrix $H$, by resampling entries in $S_k$: for $i\le j$, 
	 \begin{align*}
	h_{ij}^{[k]} = \begin{cases}
	h_{ij}' & (i,j)\in S_{k}, \\
	h_{ij}  & (i,j)\notin S_{k}.
	\end{cases}
	\end{align*}
	The remaining entries of $H^{[k]}$ below the diagonal are determined by symmetry. 
\end{defi}
Let $\lambda_{1}\ge\cdots\ge\lambda_{N}$ be the ordered eigenvalues of $H$. We consider an orthonormal basis of eigenvectors of $H$ by $\{\vv_{1},\cdots,\vv_{N}\}$, i.e., $H\vv_{i}=\lambda_{i}\vv_{i}$ and $\lVert \vv_{i} \rVert$ = 1 for each $i$. Note that Luh and Vu recently showed that sparse random matrices have simple spectrum \cite{LV18} with probability tending to one as $N$ goes to infinity.  
This implies that $\lambda_{1}>\cdots>\lambda_{N}$ and  the eigenvectors are uniquely determined up to a sign. We call $\vv_{1}$ the top eigenvector of $H$. Similarly, we use the notation $\lambda_{1}^{[k]}\ge\cdots\ge\lambda_{N}^{[k]}$ and $\vv_{1}^{[k]},\cdots,\vv_{N}^{[k]}$ to denote the ordered eigenvalues and the associated unit eigenvectors of $H^{[k]}$.

The usual scalar product in $\R^N$ is denoted by $\langle \cdot, \cdot \rangle$ and $\| v \|_{\infty} = \max_{i} |v_i|$ is the $\ell^\infty$-norm of a vector. Our main results are the following two complementary claims.  
\begin{thm}[Noise sensitivity]\label{thm: main1}
	If $q \gtrsim N^{1/9}$ and $k \gg N^{5/3}$ then
	\begin{align*}
	\E\left\lvert \left\langle \vv_{1},\vv_{1}^{[k]} \right\rangle \right\rvert = o(1).
	\end{align*}
\end{thm}

\begin{thm}[Noise stability]\label{thm: main2}
	If $q \gtrsim N^{1/9}$ and $k \ll N^{5/3}$
	then
	\begin{align*}
  \E  \min_{s\in\{\pm 1\}} \sqrt{N}\lVert \vv_{1}-s\vv_{1}^{[k]} \rVert_{\infty} = o(1).
	\end{align*}
As a result, $\E\left\lvert \left\langle \vv_{1},\vv_{1}^{[k]} \right\rangle \right\rvert = 1 - o(1).$
\end{thm}

If $q \gtrsim \sqrt N$, then Theorem \ref{thm: main1} and Theorem \ref{thm: main2} are contained in \cite{BLZ20}. To explain the threshold at $N^{5/3+o(1)}$ and the technical condition $q \gtrsim N^{1/9}$ on the sparsity parameter we may repeat the heuristic first explained in \cite{BLZ20}. First, from \cite{EKYY13}, the eigenvectors of $H$ are
delocalized in the sense that $\| \vv_m \|_\infty = N^{-1/2 + o(1)}$ with
high probability for any $1 \leq m \leq N$. Recall that $\lambda^{[k]}_1$ is the largest eigenvalue of $H^{[k]}$ with eigenvector $\vv^{[k]}_1$.  
We might guess from the derivative of a simple eigenvalue as the function of the matrix entries that
\begin{equation}\label{eq:heur1}
\lambda_1^{[k]} - \lambda^{[k-1]}_1 \simeq ( 1+ \IND (i_k \ne j_k ))  v_{i_k}( h'_{i_kj_k} -  h_{i_kj_k} )v_{j_k} \simeq  \frac{ h'_{i_kj_k} -  h_{i_kj_k} }{N^{1+o(1)}}~,
\end{equation}
where $v_i$ is the $i$-th coordinate of the top eigenvector $\vv^{[k]}$. 
Assuming that $v_i$ is nearly independent of the matrix  entries $h_{ij}$ and $h'_{ij}$, since $h_{ij}$ is centered with variance $1/N$, we would get from the central limit theorem that
$$
\lambda_1^{[k]} - \lambda_1 = \sum_{t =0}^{k-1}( \lambda_1^{[t+1]} - \lambda_1 ^{[t]} )  \simeq \frac{\sqrt k}{N^{3/2+o(1)}}~ .
$$
On the other hand, if $q \gtrsim  N^{1/9}$ then \cite[Theorem 1.6]{HLY20} implies that $\lambda_1 - \lambda_2$ is of order $N^{-2/3}$. Hence as long as $\sqrt k / N^{3/2 + o(1)}$ is much smaller than $N^{-2/3}$, it is believable that the approximation \eqref{eq:heur1} is valid and that $\vv_1^{[k]}$ is a small perturbation of $\vv_1$. This explains the threshold at $k = N^{5/3+o(1)}$. In some sense, the proof of Theorem \ref{thm: main2} makes rigorous the above heuristics. As it is usual in (non-integrable) random matrix theory, instead of working directly with eigenvalues, we will instead study the resolvent matrix of $H^{[k]}$ to shadow the behavior of $\vv_1^{[k]}$ and $\lambda_1^{[k]}$ and interpret it as a stochastic process where $k$ plays the role of time.

\begin{remark*}[Noise sensitivity for other eigenvectors]
	Following the above heuristic argument, for the eigenvector associated with the $j$-th largest eigenvalue, $\lambda_{j}$, we expect that the threshold is of order $$N^{5/3+o(1)}\min(j,N-j)^{-2/3},$$
	since the rigidity bound for $\lambda_{j}$ is given as $N^{-2/3}\min(j,N-j)^{-1/3}$. However we note that an important modification is required to show the noise sensitivity of the other eigenvectors: the argument surrounding \eqref{eq: compare lambda_1 mu_1}, Lemma \ref{lem: eigen vec small perturbation} and Lemma \ref{lem: lem14 in BLZ} are tailored to the case of the top eigenvector.
\end{remark*}

Theorem \ref{thm: main1} is proved by considering the variance of the largest eigenvalue $\lambda_1$ of $H$.
The main inequality we prove is that
\begin{equation}\label{eq:hmain1}
\E \left| \langle \vv_1,\vv_1^{[k]} \rangle \right|^2 \lesssim \frac{N^3\var(\lambda_1- \cX)}{k},
\end{equation}
where $\cX$ is defined as
\begin{align}\label{eq: random correction term}
	\cX =\frac{1}{N}\sum_{1\le i,j \le N}\left(h_{ij}^{2}-\frac{1}{N}\right) = \frac 1 N \mathrm{Tr} (H^2) -1.
\end{align}
It is a consequence of  \cite[Theorem 1.4]{HLY20} that $\var(\lambda_1 - \cX)$ is of order $N^{-4/3+o(1)}$ provided that $q \gtrsim N^{1/9}$. We then deduce Theorem \ref{thm: main1}.
As in \cite{BLZ20}, the proof of the inequality \eqref{eq:hmain1} is based on a variance formula for
general functions of independent random variables due to Chatterjee
\cite{Chatterjee05}. The inequality \eqref{eq:hmain1} shows that small variance implies noise sensitivity of the
top eigenvector. 

We note that \eqref{eq:hmain1} is also true with $\cX$ replaced by $0$ (as done in \cite{BLZ20}). It is immediate to check from \eqref{eq:momenthij} that $\var (\cX)\asymp 1/( N q^2 )$ which is larger than $N^{-4/3}$ for $q \leq N^{1/6}$. Moreover, it follows from \cite{YK20+,HLY20} that $\var(\lambda_1)$ is of the same order than $\var (\cX)$ for $1 \ll q \leq N^{1/6}$. Hence, the presence of $\cX$ in \eqref{eq:hmain1} was necessary to conclude in the regime $N^{1/9} \lesssim q \ll N^{1/6}$.

We conjecture that Theorem \ref{thm: main1} and Theorem \ref{thm: main2} remains true as long as $q \gg 1$.  With the current bounds available in \cite{YK20+,HLY20,LL19}  and the techniques of proof in the present paper, it is possible to obtain the following statements for $1 \ll q \leq N^{1/9}$: the conclusion of Theorem \ref{thm: main1} is true for $k \gg  \min(N^{7/3} q^{-6},N^{2} q^{-2})$ 
while the conclusion of Theorem \ref{thm: main2} is true for $k \ll N q^{2}$.
Since we expect that these bounds are not optimal, we shall only focus in this paper on the case $q \gtrsim N^{1/9}$.

\begin{remark*}[Higher order fluctuations of extremal eigenvalues]
	When $1 \ll q\ll N^{1/9}$, we can recover the edge rigidity by introducing higher order random correction terms introduced in the recent preprint \cite{Lee21+} posted after the first version of the present work. Thus it may be possible to show the conclusion of Theorem \ref{thm: main1} under the condition that $q\gg 1$ and $k\gg N^{5/3}$ if we replace the term $\cX$ with a new correction term $\widetilde{\mathcal{L}}$, in the main inequality \eqref{eq:hmain1}. (See \cite[Lemma 2.5]{Lee21+} and \cite[Theorem 2.10]{Lee21+} for the precise definition of $\widetilde{\mathcal{L}}$.) This $\widetilde{\mathcal{L}}$ captures higher (sub-leading) oder fluctuations of extremal eigenvalues (of sparse random matrices) whereas the term $\cX$ only governs the leading order fluctuation of those. We note that the argument associated with \eqref{eq: difference of correction term by single resampling} should be modified to establish this extension rigorously. If we denote by $\widetilde{\mathcal{L}}_{st}$ the correction term corresponding to the matrix $H_{(st)}$ obtained from $H$ by a single entry resampling at random. (See Section \ref{high level pf of main thm 1} for more detail.), we expect to have
	$$ \widetilde{\mathcal{L}}-\widetilde{\mathcal{L}}_{st} \prec N^{-1-\eps}, $$
	which will be beneficial to make some desired estimates. Similarly, for the extension of Theorem \ref{thm: main2}, the shift of the resolvent in Section \ref{sec: main argument for noise stability} must be justified with some proper modifications.  
\end{remark*}

Our definition of sparse random matrices was dictated by the use of \cite{LL19} in the proof of Theorem \ref{thm: main1}. The proof of Theorem \ref{thm: main2} does not use \cite{LL19}. Thus Theorem \ref{thm: main2} remains true for the more general sparse random matrix model considered in \cite{HLY20}.

\begin{remark*}
	From the definition of our sparse random matrices, after $k$ resampled entries there are only around $kq^2/N$ entries which have been actually modified (most of the resampled entries simply replace a null entry by a null entry). Hence the threshold at $N^{5/3 +o(1)}$ occurs after only $q^2 N^{2/3+o(1)}$  visible changes of the matrix entries. With this point of view, as $q$ gets smaller, we see that the top eigenvector gets more noise sensitive. 
\end{remark*}

The proofs of Theorem \ref{thm: main1} and Theorem \ref{thm: main2} follow the same general strategy as \cite{BLZ20}. It should be noted however that new technical challenges appear as the sparsity parameter $q$ gets smaller. The dependency of the spectrum on a single matrix entry is larger and concentration inequalities are much weaker. As a consequence some bounds used in \cite{BLZ20} where not good enough in the sparse regime. We had to modify substantially some technical arguments and also to improve some resolvent estimates on sparse random matrices from the current literature, they are gathered in the Section \ref{sec:resolvent}.

\subsection{Extension to edge resampling in \ER random graphs}

There is a natural extension of our main results to adjacency matrices of \ER random graphs. This adjacency matrix is the $N\times N$ random symmetric matrix whose diagonal entries are zero and whose entries above diagonal are independent Bernoulli random variables with mean $q^{2}/N$. We define the resampling procedure as in Definition \ref{def: resampling} with the random sets $S_k$ and an independent copy of \ER random graph. The resampling procedure describes a process where some randomly chosen edges of the graphs are added and other are removed (of order $k q^2/N$ after $k$ steps).

Let us denote an orthonormal basis of eigenvectors of \ER random graph by $\{\ww_{1},\cdots,\ww_{N}\}$ where each $\ww_{i}$ corresponds with the $i$-th largest eigenvalue. Similarly we denote by $\{ \ww_{1}^{[k]},\cdots,\ww_{N}^{[k]} \}$ the orthonormal eigenvector basis of \ER random graph after the resampling procedure.

In the regime $q \gg 1$, it is standard that the largest eigenvalue is close to $q^2$ and the top eigenvector is aligned the unit vector $\ee $  with constant coordinates: $\ee_i = 1/ \sqrt N$ for all $i$, see \cite[Theorem 2.16, Theorem 6.2]{EKYY13} for precise statements. These results imply that $\E | \langle \ww_1 , \ww_1^{[k]} \rangle | = 1 + o(1)$  for all $k$. There is thus a strong noise stability in this case. As one might expect, for the second largest eigenvalue and its corresponding eigenvector the situation is different and is parallel to sparse random matrices with mean zero entries.
\begin{thm}[Noise sensitivity]\label{thm: main1A}
	Fix $\ell \in \{2,N\}$.	If $q \gtrsim N^{1/9}$ and $k \gg N^{5/3}$ then
	\begin{align*}
		\E\left\lvert \left\langle \ww_{\ell},\ww_{\ell}^{[k]} \right\rangle \right\rvert = o(1).
	\end{align*}
\end{thm}
\begin{thm}[Noise stability]\label{thm: main2A}
	Fix $\ell \in \{2,N\}$.	If $q\gtrsim N^{1/9}$ and $k \ll N^{5/3}$
	then
	\begin{align*}
		\E  \min_{s\in\{\pm 1\}} \sqrt{N}\lVert \ww_{\ell}-s\ww_{\ell}^{[k]} \rVert_{\infty} = o(1).
	\end{align*}
\end{thm}
The proofs of these results will follow from an adaptation of the proofs of Theorem \ref{thm: main1} and Theorem \ref{thm: main2}. With a different perspective, the noise sensibility of the spectrum under edge resampling  has already been considered in \cite{MR3660521}. Our results suggest that in real-world networks, a heuristic to discriminate eigenvectors containing an information on the structure of the network from less relevant eigenvectors, could be through random uniform resampling of the edges: noise sensitive eigenvectors should not contain meaningful information.

\paragraph{Organization of the paper.}
In the next section, we shall cover some necessary tools used in the proof of the main results. In Section \ref{sec: proof strategy}, we describe the high-level proofs of Theorem \ref{thm: main1} and Theorem \ref{thm: main2}. The remaining sections, Section \ref{sec: thm1} and Section \ref{sec: thm2}, are devoted to the details, for Theorem \ref{thm: main1} and Theorem \ref{thm: main2}, respectively. The proofs of Theorem \ref{thm: main1A} and Theorem \ref{thm: main2A} are explained in Section \ref{sec:ER}. Finally, Section \ref{sec:resolvent} contains some new resolvent estimates on sparse random matrices.

\paragraph{Acknowledgments.} We thank the referees for their careful reading of the manuscript and many helpful suggestions. CB was supported by the research grant ANR-16-CE40-0024-01. JL was supported by the National Research Foundation of Korea (NRF-2017R1A2B2001952; NRF-2019R1A5A1028324).

\section{Preliminaries} 
In this section, we collect some necessary tools for the proof of main results. 

\subsection{Variance and noise sensitivity}

For any positive integer $i$, denote $[i]=\{1,\cdots,i\}$. Let $Y_{1},\cdots,Y_{n}$ be i.i.d. random variables taking values in a set $\cY$ equipped with a $\sigma$-algebra. Consider the random vector $Y=(Y_{1},\cdots,Y_{n})$ and let $Y'=(Y_{1}',\cdots,Y_{n}')$ be an independent copy of $Y$. We shall use the following notation,
$$
Y^{(i)}=(Y_{1},\cdots,Y_{i-1},Y_{i}',Y_{i+1},\cdots,Y_{n})
$$
For $\mathcal{I}\subset[n]$, we define $Y^{\mathcal{I}}= (Y^{\mathcal{I}}_{1}, \cdots, Y^{\mathcal{I}}_{n})$ by setting
$$
 Y^{\mathcal{I}}_{i} = \begin{cases}
  Y_{i} & \text{if }i\notin\mathcal{I}, \\
  Y_{i}'& \text{if }i\in\mathcal{I}.
 \end{cases}
$$
Let $\sigma=(\sigma(1),\cdots,\sigma(n))$ be a permutation in the symmetric group $\mathcal{S}_{n}$. For $i \in [n]$, we set $\sigma[i] = \{\sigma(1),\cdots,\sigma(i)\}$ and $\sigma[0] = \emptyset$.  Let $Y''$ and $Y'''$ be independent copies of $Y$. We assume $Y$, $Y'$, $Y''$ and $Y'''$ are independent. For $j \in [n]$, let $Y^{(j)\circ\sigma[i-1]}$ be the vector obtained from $Y^{\sigma[i-1]}$ by replacing $j$-th component of $Y^{\sigma[i-1]}$ as follows:
\begin{align*}
	Y^{(j)\circ\sigma[i-1]}_{j} = \begin{cases}
		Y_{j}'' & j\in\sigma[i-1], \\
		Y_{j}''' & j\notin\sigma[i-1].
	\end{cases}
\end{align*} 
For example, if $n=5$, $i=j=3$ and $\sigma=(2,3,1,5,4)$, we have $\sigma([i-1])=\{2,3\}$,
$$
Y^{\sigma[i-1]}=(Y_{1}, Y_{2}', Y_{3}', Y_{4}, Y_{5}) \quad \text{and} \quad
Y^{(j)\circ\sigma[i-1]}=(Y_{1}, Y_{2}', Y_{3}'', Y_{4}, Y_{5}).
$$
On the other hand, if $j=1$, we have
$$
Y^{(j)\circ\sigma[i-1]}=(Y_{1}''', Y_{2}', Y_{3}', Y_{4}, Y_{5}).
$$
\begin{lem}[Variance and noise sensitivity]\label{lem: superconcentration}
	Assume $f:\cY^{n}\to\R$ is a measurable function. Let $j$ be a random variable uniformly distributed on $[n]$ independently of $(Y,Y',Y'')$ and $\sigma$ be uniformly distributed in $\mathcal S_n$ independently of $(Y,Y',Y'',j)$. For any $k \in[n]$, define $I_{k}$ by
$$
	I_{k}=\E\left[ \left(f(Y)-f(Y^{(j)})\right) \left(f(Y^{\sigma([k-1])})-f(Y^{(j)\circ\sigma([k-1])})\right) \right].
$$
	Then, we have for any $k \in[n]$,
$$
	I_{k}\le\left(\frac{n+1}{n}\right)\left(\frac{2\var \big(f(Y)\big)}{k}\right).
$$
\end{lem}
This is a small modification of \cite[Lemma 3]{BLZ20}. We refer to \cite{BLZ20} for a proof and other similar statements.

\subsection{Local laws and universality in sparse random theory}

We start by introducing two handy probabilistic notions.
\begin{defi}[Overwhelming probability]
	Let $\{E_{N}\} $ be a sequence of events. We say $E_{N}$ holds with {\em overwhelming probability} if for any $D>0$, there exists $N_{0}(D)$ such that we have for $N\ge N_{0}(D)$
	\begin{align}
	\prob\big(E_{N}^{c}\big)\le N^{-D}.
	\end{align}
	If $\{ F_N\}$ is another sequence of events, we say that, on $F_N$, $E_N$ holds with overwhelming probability, if $E_N \cup F_N^c$ has overwhelming probability. 
\end{defi}

\begin{defi}[Stochastic domination]
	Let $(U_{N})$ and $(V_{N})$ be two sequences of nonnegative random variables. $U$ is said to be {\em stochastically dominated} by $V$ if for all $\eps>0$ and $D>0$ there exists $N_{0}(\eps,D)$ such that we have for $N\ge N_{0}(\eps,D)$
	\begin{align*}
		\prob[U_{N}>N^{\eps}V_{N}]\le N^{-D}.
	\end{align*}
    If $U$ is stochastically dominated by $V$, we use the notation $U\prec V$. If $(E_N)$ is a sequence of events, we say that on $(E_N)$, $U \prec V$, if $U' \prec V$ with $U'_N = \IND_{E_N}  U_N$. Finally, if $(U_N(t))$ and $(V_N(t))$ are two families of sequences of non-negative random variables indexed by $t \in T$, we say that $U_t \prec V_t$ uniformly in $t \in T$ if the above integer $N_0(\eps,D)$ can be taken independent of $t \in T$.
\end{defi}

Note that if $U_N$ and $V_N$ are deterministic then $U \prec V$ means $U_N \leq N^{o(1)} V_N$.

Let $H=(h_{ij})$ be as in Definition \ref{def: sparse RM}. Recall that $\lambda_1 \geq \ldots \geq \lambda_N$ are the eigenvalues of $H$ and $(\vv_1,\ldots,\vv_N)$ is an orthonormal basis of eigenvectors. A first key ingredient is the proof is the delocalization of eigenvectors. 

\begin{lem}[Delocalization of eigenvectors {\cite[Theorem 2.16, Remark 2.18]{EKYY13}}]\label{lem: delocalization}
	Assume $q\gg 1$. We have
$$	\max_{1\le i \le N} \lVert \vv_{i} \rVert_{\infty} \prec \frac{1}{\sqrt N}.
$$
\end{lem}

A second ingredient is a non-asymptotic bound on the eigenvalue spacings of $H$.
\begin{lem}[Tail bounds for the gaps between eigenvalues {\cite[Theorem 2.2]{LL19}}]\label{lem: tail bound for gaps}
 	Assume $q\gg 1$. There exists a constant $c>0$ such that the following holds for any $\delta \geq N^{-c}$,
$$
	\sup_{1\le i\le N-1} \prob\left( \lambda_{i}-\lambda_{i+1} \le \frac{\delta}{N}  \right)=O(\delta \log{N}).
$$
\end{lem}
To be precise, in the above statement, the constant $c$ depends on the sub-Gaussian tail parameter $\vartheta$. Also, the $O(\cdot)$ on the right-hand side depends on $\vartheta$ and on a uniform lower bound on $\log q / \log N$ (which is positive by the assumption $q \gg 1$). There exists a non-quantitative result which is optimal on the scaling which is contained in \cite[Theorem 1.6]{LL19}. 

 \begin{lem}[Tracy-Widom scaling for the gap {\cite[Theorem 1.6]{HLY20}}]\label{lem:TW}
 	Assume $q\gtrsim N^{1/9}$. For any $\eps >0$, there exists a constant $c>0$ such that 
$$
	\PP (  \lambda_{1}-\lambda_{2} \ge c N^{-2/3} ) \geq 1 - \eps. $$
\end{lem}

We now describe the location of the eigenvalues. Recall that if $\mu$ is a finite measure on $\R$, its {\em Cauchy-Stieltjes transform} is defined as the holomorphic function on $\C_+ = \{ z \in \C : \Im(z) > 0 \}$ by
$$
z \mapsto \int \frac{d\mu(\lambda)}{\lambda -z }.
$$
A measure is characterized by its Cauchy-Stieltjes transform and tools like Helffer-Sj{\"{o}}strand formula allow to infer precise information on the measure through its Cauchy-Stieltjes transform, see e.g. \cite{MR3792624} for its use in random matrix theory. We denote by $m(z)$ the Cauchy-Stieltjes transform of the empirical measure of eigenvalues of $H$:
$$
m(z) = \frac 1 N \sum_{i=1}^N \frac{1}{\lambda_i - z}. 
$$
For an arbitrarily small constant $\eps >0$, we define the shifted spectral domain
\begin{align*}
	\mathcal{D}(\eps) =\left\{ w =\kappa+\mathfrak{i}\eta\in\C_{+} : |\kappa|\le 3, 0\le\eta\le 1, |\kappa|+\eta \ge N^{\eps}\left(\frac{1}{q^3 N^{1/2}} +\frac1 {q^{3}N\eta}+\frac{1}{(N\eta)^2}\right) \right\}.
\end{align*}
\begin{lem}[Local law, Theorem 2.1 of \cite{HLY20}]\label{lem: local law}
Assume $q\gg 1$ and let $\eps >0$. There exists an explicit random symmetric measure $ \rho_\star$ with random  support $[- \cL, \cL]$ whose Stieltjes transform $ m_{\star}$ satisfies the following. Uniformly for any $z  = \mathcal{L}+w$, with $w = \kappa + \mathfrak{i}\eta\in\mathcal{D}(\eps)$, we have, 
	\begin{enumerate}[-]
		\item If $\kappa\ge 0$, 
		\begin{align*}
			|m(z)-m_\star(z)| \prec \frac{1}{\sqrt{|\kappa| + \eta}}
			\left(\frac 1 {N\eta^{1/2}}+\frac{1}{q^{3/2} N^{1/2} } +\frac{1}{q^{3}N\eta}+\frac 1{(N\eta)^2}\right),
		\end{align*}
		\item If $\kappa\le 0$, 
		\begin{align*}
			|m(z)-m_\star(z)| \prec \frac1 {N\eta} + \frac 1 {q^{3/2} N^{1/2}\eta^{1/2}}.
		\end{align*}
	\end{enumerate}
\end{lem}

We note that in \cite{HLY20}, the set $\mathcal{D}( \eps)$ is restricted to $|\kappa| \leq 1$ but their local law holds for any $\kappa$ taking value in any fixed interval (the focus in \cite{HLY20} is on the edge behavior). In Section \ref{sec:resolvent}, we will prove an improvement on Lemma \ref{lem: local law} when $\kappa$ and $\eta$ are sufficiently small.

The random measure $\rho_\star$ has positive density on $(-\cL,\cL)$ and it is a small deformation of the semi-circular law. From \cite[Proposition 2.6]{HLY20}, the measure $\rho_\star$ satisfies a polynomial equation whose coefficients depend on the moments of the entries of $H$ and on the random variable $\cX$ defined by \eqref{eq: random correction term}, see Section \ref{sec:resolvent} for details. Let us mention that there exists a deterministic real $L = 2 + O(1/q^2)$ such that 
\begin{equation}\label{eq:proxyL}
|\cL  -  L - \cX | \prec N^{-1/2}q^{-3},
\end{equation}
see \cite[Proposition 2.6]{HLY20} for details. We have for all $z = E + \mathfrak{i}\eta \in \C_+$,
	\begin{align}\label{eq:mstarasym}
			\Im [m_\star(E+\mathfrak{i}\eta)]\asymp\begin{cases}
				\sqrt{\kappa+\eta}, & \text{if}\;\; E\in[-\mathcal{L},\mathcal{L}]\\
				\frac{\eta}{\sqrt{\kappa+\eta}}, & \text{if}\;\; E\notin[- \mathcal{L},\mathcal{L}],
			\end{cases}
		\end{align}
where $\kappa$ is the distance of $E$ to $\{- \mathcal{L}, \mathcal{L}\}$, the boundary of the support of $\rho_\star$.

Lemma \ref{lem: local law} can be used to establish a rigidity estimate of the eigenvalues of $H$. For integer $1 \leq i \leq N$, we define the typical location of $\lambda_i$ as the number $\gamma_i$ such that 
$$
\rho_\star ([\gamma_i,\cL]) = \frac{i-1}{N},
$$
that is, $\gamma_i$ is associated with the $(i-1)$-th $1/N$-quantile of $\rho_\star$. See \cite[Lemma 2.12]{YK20+} for the asymptotic value of $\gamma_i$ in terms of $\cX$ and the corresponding $1/N$-quantile of the semi-circular law.

If $q \gg 1$, there exists $\eps > 0$ such that $q \gg N^{\eps}$. Then, the proof of Theorem 1.4 in \cite{HLY20} with $\eta = N^{-2/3}$ and $\kappa = N^{\mathfrak a} ( N^{-1/3}q^{-3} + N^{-2/3} )$ (instead of $\kappa = N^{\mathfrak a} (q^{-6} + N^{-2/3} )$ in \cite{HLY20}) gives 
\begin{equation}\label{eq:riglambda1}
|\lambda_1 - \cL| \prec N^{-1/3}q ^{-3} + N^{-2/3}.
\end{equation} More generally, armed with Lemma \ref{lem: local law}, we can obtain the next lemma by following the standard argument using Helffer-Sj{\"{o}}strand formula such as \cite[Section 1.8]{MR3792624} with cut on the imaginary axis at $\eta = N^{-2/3}$:

\begin{lem}[Eigenvalue rigidity]\label{lem: new rigidity estimate}
	Assume $q\gg 1$. For all $1\le i\le N$, we have
	\begin{align*}
		|\lambda_{i}- \gamma_{i}|\prec N^{-1/3}q ^{-3} + N^{-2/3}.
	\end{align*}
\end{lem}

From \eqref{eq:mstarasym}, we find easily $\gamma_i \gtrsim i^{2/3} N^{-2/3}$ uniformly in $1 \leq i \leq N$. Moreover by Lemma \ref{lem: new rigidity estimate}, if $q \gtrsim N^{1/9}$, we have $|\lambda_i - \gamma_i | \prec N^{-2/3}$. Hence, by Lemma \ref{lem: tail bound for gaps}, the next corollary follows.

\begin{cor}\label{cor: lower bound for spectral gap}
Let $\eps >0$ and assume $q \gtrsim N^{1/9}$. There exist $c>0$ such the following holds for any $\delta \geq N^{-c}$, for all $N$ large enough, with probability at least $1- \delta \log{N}$:
	\begin{align*}
		\lambda_{1}-\lambda_{i} \ge\begin{cases}
			c\delta N^{-1}  & \text{if}\;\; 2 \le i \le N^{\eps} \\
			c i^{2/3}N^{-2/3} & \text{if}\;\; N^{\eps} < i \le N.
		\end{cases}
	\end{align*}
    Moreover, on the event $\{\lambda_{1}-\lambda_{2}\ge c\delta N^{-1}\}$,
    the above inequalities holds with overwhelming probability.
\end{cor}

\section{High-level proof of the main results}\label{sec: proof strategy}
We adapt the method of proof in \cite{BLZ20} by applying recent results for the sparse \ER graph model, in order to establish Theorem \ref{thm: main1} and Theorem \ref{thm: main2}. 

\subsection{High-level proof of Theorem \ref{thm: main1}}\label{high level pf of main thm 1}
For any $1\le i\le j\le N$, denote by $H_{(ij)}$ the symmetric matrix obtained from $H$ by replacing the entries $h_{ij}$ and $h_{ji}$ with $h_{ij}''$, where $h_{ij}''$ is an independent copy of $h_{ij}$. Similarly, we write $H^{[k]}_{(ij)}$ for the symmetric matrix obtained from $H^{[k]}$ by replacing $h^{[k]}_{ij}$ and $h^{[k]}_{ji}$ as follows:
\begin{itemize}
	\item If $(i,j)\in S_{k}$, then $h^{[k]}_{ij}$ and $h^{[k]}_{ji}$ are replaced with $h_{ij}''$. 
	\item If $(i,j)\notin S_{k}$, then $h^{[k]}_{ij}$ and $h^{[k]}_{ji}$ are replaced with $h_{ij}'''$, where $h_{ij}'''$ is another independent copy of $h_{ij}$.
\end{itemize}
Denote by $(st)$ a random pair of indices chosen uniformly from $\{(i,j): 1\le i\le j\le N
\}$. Note that 
\begin{align*}
|\{(i,j): 1\le i\le j\le N\}|=N(N+1)/2
\end{align*}
Let $\mu_{1}\ge\cdots\ge\mu_{N}$ be the ordered eigenvalues of $H_{(st)}$ and, let $\uu_{1},\cdots,\uu_{N}$ be the associated unit eigenvectors of $H_{(st)}$. Similarly, we define $\mu_{1}^{[k]}\ge\cdots\ge\mu_{N}^{[k]}$ and $\uu_{1}^{[k]},\cdots,\uu_{N}^{[k]}$ for $H^{[k]}_{(st)}$. We apply Lemma \ref{lem: superconcentration} with $Y=H$ and $f(H)=\lambda_{1} - L - \cX$: 
\begin{align}\label{eq: application of superconcentration lem}
\E\left[ \big(\lambda_{1}-\mu_{1}-Q_{st}\big)\big(\lambda_{1}^{[k]}-\mu_{1}^{[k]}-Q_{st}^{[k]}\big) \right]\le\frac{2\text{Var}(\lambda_{1}- L - \cX)}{k}\cdot\frac{N(N+1)+2}{N(N+1)},
\end{align}
where
\begin{align}\label{eq: difference of correction term by single resampling}
Q_{st}&:=\frac{1}{N}(h_{st}^{2}-(h_{st}'')^{2})(1+\indic(s\neq t)) , \nn\\
Q_{st}^{[k]}&:=\begin{cases}
\frac{1}{N}((h_{st}')^{2}-(h_{st}'')^{2})(1+\indic(s\neq t)) &\text{if } (st)\in S_{k},\\
\frac{1}{N}(h_{st}^{2}-(h_{st}''')^{2})(1+\indic(s\neq t)) &\text{if } (st)\notin S_{k}.
\end{cases}
\end{align}
By the spectral theorem, we have
\begin{align} \label{eq: compare lambda_1 mu_1}
	\langle \uu_{1},H\uu_{1}\rangle = \sum_{i=1}^{N}\lambda_{i}|\langle \uu_{1},\vv_{i}\rangle|^{2}\le \lambda_{1}\sum_{i=1}^{N}|\langle \uu_{1},\vv_{i}\rangle|^{2} = \lambda_{1} = \langle \vv_{1},H\vv_{1}\rangle.
\end{align}
Similarly, it follows that
\begin{align*}
	\langle \vv_{1},H_{(st)}\vv_{1}\rangle \le \langle \uu_{1},H_{(st)}\uu_{1}\rangle.
\end{align*}
Combining the two above inequalities, we obtain
\begin{align*}
	\langle \uu_{1},(H-H_{(st)})\uu_{1}\rangle \le \lambda_{1} - \mu_{1} \le 	\langle \vv_{1},(H-H_{(st)})\vv_{1}\rangle.
\end{align*}
Also, by the same argument, we have
\begin{align*}
	\langle \uu^{[k]}_{1},(H^{[k]}-H^{[k]}_{(st)})\uu^{[k]}_{1}\rangle \le \lambda_{1}^{[k]} - \mu^{[k]}_{1} \le 	\langle \vv^{[k]}_{1},(H^{[k]}-H^{[k]}_{(st)})\vv^{[k]}_{1}\rangle.
\end{align*}
Let us write $\vv_{1}=(v_{1},\cdots,v_{N})$, $\uu_{1}=(u_{1},\cdots,u_{N})$, $\vv_{1}^{[k]}=(v_{1}^{[k]},\cdots,v_{N}^{[k]})$, and $\uu_{1}^{[k]}=(u_{1}^{[k]},\cdots,u_{N}^{[k]})$. Then, we find
\begin{align*}
	Z_{st}u_{s}u_{t} \le \lambda_{1}-\mu_{1} \le Z_{st}v_{s}v_{t},
\end{align*}
where 
\begin{align*}
	Z_{st}:=(h_{st}-h_{st}'')(1+\indic(s\neq t)).
\end{align*}
Similarly,
\begin{align*}
	Z_{st}^{[k]}u_{s}^{[k]}u_{t}^{[k]} \le \lambda_{1}^{[k]}-\mu_{1}^{[k]} \le Z_{st}^{[k]}v_{s}^{[k]}v_{t}^{[k]},
\end{align*}
where
\begin{align*}
	Z_{st}^{[k]}:=\begin{cases}
		(h_{st}'-h_{st}'')(1+\indic(s\neq t)) &\text{if } (st)\in S_{k},\\
		(h_{st}-h_{st}''')(1+\indic(s\neq t)) &\text{if } (st)\notin S_{k}.
	\end{cases}
\end{align*}
We set $T_{1}=(Z_{st}v_{s}v_{t}-Q_{st})(Z_{st}^{[k]}v_{s}^{[k]}v_{t}^{[k]}-Q_{st}^{[k]})$, $T_{2}=(Z_{st}v_{s}v_{t}-Q_{st})(Z_{st}^{[k]}u_{s}^{[k]}u_{t}^{[k]}-Q_{st}^{[k]})$, $T_3 = (Z_{st}u_{s}u_{t}-Q_{st})(Z_{st}^{[k]}v_{s}^{[k]}v_{t}^{[k]}-Q_{st}^{[k]})$, $T_4 = (Z_{st}u_{s}u_{t}-Q_{st})(Z_{st}^{[k]}u_{s}^{[k]}u_{t}^{[k]}-Q_{st}^{[k]})$. We have

\begin{align}\label{eq: top eig val difference under sigle resample}
\min(T_{1},T_{2},T_{3},T_{4})\le  \big(\lambda_{1}-\mu_{1}-Q_{st}\big)\big(\lambda_{1}^{[k]}-\mu_{1}^{[k]}-Q_{st}^{[k]}\big) \le \max(T_{1},T_{2},T_{3},T_{4}).
\end{align}

The next key lemma asserts that after one resample the top eigenvectors are close in $\ell^{\infty}$-norm.  Its proof will use the delocalization of eigenvectors and the rigidity of the eigenvalues.
\begin{lem}\label{lem: eigen vec small perturbation}
Assume $q \gtrsim N^{1/9}$ and let $c,\delta >0$ be such that $N^{c + \delta} \ll q $. For $1 \leq i \le j \leq N$, let $\uu_{1}^{(ij)}$ be the top eigenvector of $H_{(ij)}$. Then, on the event $\left\{  \lambda_{1}-\lambda_{2} \ge N^{-1-c}\right\} $,
the event
\begin{align*}
 \bigcap_{1 \leq i\leq j \leq N} \left\{   \inf_{s\in\{\pm 1\}}\lVert s\vv_{1}-\uu_{1}^{(ij)} \rVert_{\infty}\le N^{-1/2-\delta}\right\}
\end{align*}
holds with overwhelming probability. The analogous result for $H^{[k]}_{(ij)}$ also holds.
\end{lem}
Next, let $0 < \delta < 1/9$  and  $0 < \eps < \delta /3$ to be defined later, we define the events 
\begin{align}\label{eq: event 1 conditions}
	\mathcal{E}_{1} :=	\left\{ \max \left(\lVert \vv_1   \rVert_{\infty},\lVert \uu_1 \rVert_{\infty},\lVert  \vv_{1}^{[k]}\rVert_{\infty},\lVert \uu_1^{[k]} \rVert_{\infty}\right) \le  N^{\eps - 1/2} \right\},
\end{align}
\begin{align}\label{eq: event 2 conditions}
	\mathcal{E}_{2} :=	  \left\{   \max \left(  \lVert \vv_{1}-\uu_{1} \rVert_{\infty} , \lVert \vv^{[k]}_{1}-\uu^{[k]}_{1} \rVert_{\infty} \right) \le N^{-1/2-\delta}\right\}.
\end{align}
Set the event $\mathcal{E}:=\mathcal{E}_{1}\cap\mathcal{E}_{2}$. Let $c >0$ such that $c + \delta < 1/9$. According to Lemma \ref{lem: delocalization}, Lemma \ref{lem: tail bound for gaps} and Lemma \ref{lem: eigen vec small perturbation}, we have
$\prob(\mathcal{E}^{c})= O(N^{-c}\log{N})$ by choosing the $\pm$-phase properly for $\uu_{(ij)}$ and $\uu_{(ij)}^{[k]}$. On the event $\mathcal{E}$, we observe that $v_{s}v_{t}u_{s}^{[k]}u_{t}^{[k]}$, $u_{s}u_{t}v_{s}^{[k]}v_{t}^{[k]}$ and $u_{s}u_{t}u_{s}^{[k]}u_{t}^{[k]}$ can be replaced with
\begin{align*}
v_{s}v_{t}v_{s}^{[k]}v_{t}^{[k]}+O\left(N^{3\eps - 2 - \delta}\right).
\end{align*}
Thus, on the event $\mathcal{E}$, it follows from \eqref{eq: top eig val difference under sigle resample}
\begin{multline}\label{eq:lmQk}
\big(\lambda_{1}-\mu_{1}-Q_{st}\big)\big(\lambda_{1}^{[k]}-\mu_{1}^{[k]}-Q_{st}^{[k]}\big)\ge Z_{st}Z_{st}^{[k]}v_{s}v_{t}v_{s}^{[k]}v_{t}^{[k]}-O\left( | Z_{st}Z_{st}^{[k]}| N^{3\eps - 2 - \delta} \right)\\
 -  |Q_{st}Z_{st}^{[k]}|N^{2\eps - 1}  -|Q_{st}^{[k]}Z_{st}|  N^{2\eps - 1} - |Q_{st}Q_{st}^{[k]}|.
\end{multline}

We shall check the following decorrelation lemma between the event $\mathcal E$ and our random variables of interest.
\begin{lem}\label{lem: expectation estimate} If $4 \eps + \delta < 1/9$, we have
	\begin{align*}
	\E\left[ Z_{st}Z_{st}^{[k]}v_{s}v_{t}v_{s}^{[k]}v_{t}^{[k]}\indic_{\mathcal{E}^c} \right] = o\left(\frac 1 {N^{3}}\right),
	\end{align*}
and
	\begin{align*}
	\E\left[ (\lambda_{2}-\mu_{2}-Q_{st})(\lambda_{2}^{[k]}-\mu_{2}^{[k]}-Q_{st}^{[k]})\indic_{\mathcal{E}^{c}} \right]=o\left(\frac 1 {N^{3}}\right).
	\end{align*}
\end{lem}

Since $\E |Z_{st}Z_{st}^{[k]}|=O(N^{-1})$, $\E |Q_{st}Z_{st}^{[k]}| =O(N^{-2}q^{-1})$ and $\E |Q_{st}Q_{st}^{[k]}| =O(N^{-3}q^{-2})$,
we deduce that the inequality
\begin{align}\label{eq: resulting ineq}
\E\left[ (\lambda_{2}-\mu_{2}-Q_{st})(\lambda_{2}^{[k]}-\mu_{2}^{[k]}-Q_{st}^{[k]}) \right]\geq  \E\left[ Z_{st}Z_{st}^{[k]}  v_{s}v_{t}v_{s}^{[k]}v_{t}^{[k]}\right]+o\left(\frac 1 {N^{3}}\right)
\end{align}
follows from \eqref{eq:lmQk} and Lemma \ref{lem: expectation estimate}.

Using that $(s,t)$ is uniformly distributed on $\{(i,j) : 1 \leq i \leq j \leq N\}$ and that $\E [Z_{ij}Z_{ij}^{[k]} | S_k]= 4/N$ if $i < j$, we will prove the following lemma. 
\begin{lem}\label{lem: expectation estimate2} We have
	\begin{align*}
	 \E\left[ Z_{st}Z_{st}^{[k]} v_{s}v_{t}v_{s}^{[k]}v_{t}^{[k]}\right] =\frac{2}{N^3} \E \left[\langle \vv_{1}, \vv_{1}^{[k]} \rangle^2 \right]+o\left(\frac 1 {N^{3}}\right).
	\end{align*}
\end{lem}
 
Now we are ready to prove the main statement. From \eqref{eq: application of superconcentration lem} and \eqref{eq: resulting ineq}, we find
\begin{align*}
\E \left[\langle \vv_{1}, \vv_{1}^{[k]} \rangle^2 \right] \leq  \frac{N^3\text{Var}(\lambda_{1}- L - \cX)}{k}\left( 1 + o(1) \right)  + o(1).
\end{align*}
Using  \eqref{eq:riglambda1}, we have for any $\eps >0$,
$$
\text{Var}(\lambda_{1}- L - \cX) = O ( N^{\eps-4/3} ).  
$$
It remains to use Jensen's inequality: $ \left(\E|\langle \vv_{1}, \vv_{1}^{[k]} \rangle| \right)^2 \leq  \E \left[\langle \vv_{1}, \vv_{1}^{[k]} \rangle^2 \right]$ and the assumption $k \gg N^{5/3}$ to conclude that $\E \left[|\langle \vv_{1}, \vv_{1}^{[k]} \rangle| \right] = o(1)$.

This concludes the proof of Theorem \ref{thm: main1} with  Lemma \ref{lem: eigen vec small perturbation}, Lemma \ref{lem: expectation estimate} and Lemma \ref{lem: expectation estimate2} granted. These lemmas are proved in Section \ref{sec: thm1}. \qed

\subsection{High-level proof of Theorem \ref{thm: main2}}\label{high level pf of main thm 2}

For $z=E+\mathfrak{i}\eta$ with $\eta>0$ and $E\in\R$, we introduce the resolvent matrix
\begin{align*}
R(z)=(H-zI)^{-1},
\end{align*}
where $I$ denotes the identity matrix. We denote by $R^{[k]}(z)$ the resolvent of $H^{[k]}$. As already advertised in the introduction, the proof of Theorem \ref{thm: main2} relies on a fine study of the functional process $R^{[k]}$ where $k$ plays the role of time.  The domain of the parameter $z$ will be tuned to follow the evolution of the largest eigenvalue $\lambda_1^{[k]}$ and the top eigenvector $\vv^{[k]}$ as $k$ evolves. 

The main technical result is the following.

\begin{lem}\label{lem: lem13 in BLZ}
	Assume $q \gtrsim N^{1/9}$ and $k\ll N^{5/3}$. Then, there exists $\delta_0 >0$ such that for all $0  < \delta < \delta_0$, there exists $c>0$ such that, with overwhelming probability,
	\begin{align*}
	\sup_{z} \max_{1\le i,j\le N} N\eta|\Im R_{ij}^{[k]}(z)-\Im R_{ij}(z)|\leq  N^{-c},
	\end{align*}
	where the supremum is over all $z=E+\mathfrak{i}\eta$ with $|E-\cL |\le N^{-2/3+\delta}$ and $\eta = N^{-2/3-\delta}$.
\end{lem}

We write $\vv_{1}=(v_{1},\cdots,v_{N})$ and $\vv^{[k]}_{1}=(v^{[k]}_{1},\cdots, v^{[k]}_{N})$. The next lemma connects the entries of the resolvent with the coordinates of the top eigenvector for $z$ close enough to the largest eigenvalue.

\begin{lem}\label{lem: lem14 in BLZ}
	Assume $q \gtrsim N^{1/9}$ and $k\ll N^{5/3}$. Let $0 < \delta <\delta_0$ be as in Lemma \ref{lem: lem13 in BLZ}. There exists $c'>0$ such that with probability  $1-o(1)$ it holds that 
	\begin{align*}
	\max_{1\le i,j \le N}N|\eta\Im R_{ij}(z)-v_{i}v_{j}|\le N^{-c'}\quad\text{and}\quad
	\max_{1\le i,j \le N}N|\eta\Im R^{[k]}_{ij}(z)-v_{i}^{[k]}v_{j}^{[k]}|\le N^{-c'},
	\end{align*}
	with $z=\lambda_{1}+\mathfrak{i}\eta$ and $\eta = N^{-2/3-\delta}$.
\end{lem}

In the proof of Lemma \ref{lem: lem14 in BLZ}, we will also show that $\lambda_1$ and $\lambda_1^{[k]}$ are close as long as $k \ll N^{5/3}$, see Lemma \ref{lem: lem12 in BLZ} below. The proofs of Lemma \ref{lem: lem13 in BLZ} and Lemma \ref{lem: lem14 in BLZ} are postponed to Section \ref{sec: thm2}.

We now explain the proof of Theorem \ref{thm: main2} with  Lemma \ref{lem: lem13 in BLZ} and Lemma \ref{lem: lem14 in BLZ} granted. According to \eqref{eq:riglambda1}, we have
\begin{align*}
|\lambda_{1}-\cL |\prec N^{-2/3}.
\end{align*}
Thus, with overwhelming probability $z = \lambda_1 + \mathfrak{i}\eta$ with $\eta=N^{-2/3-\delta}$ is in the domain considered in Lemma \ref{lem: lem13 in BLZ}. In particular, since
\begin{align*}
\lvert v_{i}v_{j} - v_{i}^{[k]}v_{j}^{[k]} \rvert 
&\le |v_{i}v_{j}-\eta\Im R_{ij}(z)| + \eta|\Im R_{ij}(z)-\Im R_{ij}^{[k]}(z)| + |\eta\Im R_{ij}^{[k]}(z)-v_{i}^{[k]}v_{j}^{[k]}|,
\end{align*}
the combination of Lemma \ref{lem: lem13 in BLZ} and Lemma \ref{lem: lem14 in BLZ} implies the following claim:

\begin{lem}\label{lem: main lem for thm2}
	Assume $q \gtrsim N^{1/9}$ and $k\ll N^{5/3}$. There exists $c >0$ such that with probability $1-o(1)$, it holds that: 
	\begin{align}\label{eq:vivjivjk}
	\max_{1\le i,j\le N} N\lvert v_{i}v_{j} - v_{i}^{[k]}v_{j}^{[k]} \rvert \leq N^{-c}.
	\end{align}
\end{lem}

Let $0 < c' < c$ with $c$ as Lemma \ref{lem: main lem for thm2}. To prove Theorem \ref{thm: main2}, we prove that with probability $1- o(1)$, 
\begin{equation}\label{eq:vvkcp}
\sqrt{N} \| \vv_{1} - \vv_{1}^{[k]} \|_\infty \leq N^{-c'},
\end{equation}
for a proper choice of the $\pm$-phase for the top eigenvectors. Let $\eps >0$ such that $c' + \eps < c$. Let us call $\mathcal E_0$ the event that \eqref{eq:vivjivjk} holds and $\|\vv_{1}^{[k]} \|_{\infty} \leq N^{\eps-1/2}$. By Lemma   \ref{lem: main lem for thm2} and Lemma \ref{lem: delocalization}, it is sufficient to prove that for $N$ large enough \eqref{eq:vvkcp} holds on the event $\mathcal E_0$.

Let $i$ be such that $|v_i| \geq 1/\sqrt N$. We choose the phases of $\vv$ and $\vv^{[k]}$ such that $v_i , v_i^{[k]}$ are non-negative. Then, we get on $\mathcal E_0$,
$$
| v_i - v_i^{[k]} | = \frac{| v^2_i - (v_i^{[k]})^2 |}{v_i + v_i^{[k]} } \leq \frac{| v^2_i - (v_i^{[k]})^2 |}{v_i}  \leq  N^{-c - 1/2}.
$$
For any integer $1 \leq j \leq N$, we write: 
\begin{align*}
| v_j - v_j^{[k]} | = &\frac{1}{v_i} | v_i v_j - v_i v_j^{[k]} |  \leq \frac{1}{v_i} | v_i v_j - v^{[k]}_i v_j^{[k]} | + \frac{|v^{[k]}_j|}{v_i} | v_i - v_i^{[k]} |.
\end{align*}
Hence, on the event $\mathcal E_0$, we find
$$
| v_j - v_j^{[k]} | \leq  N^{-c - 1/2} +  N^{\eps - c - 1/2}.
$$
For our choice of $\eps$, we deduce that \eqref{eq:vvkcp} holds for all $N$ large enough.  Theorem \ref{thm: main2} is proved.\qed

\section{Noise sensitivity of the top-eigenvector}\label{sec: thm1}

\subsection{Proof of Lemma \ref{lem: eigen vec small perturbation}}\label{pf: eigen vec small perturbation}
Let $\mu^{(ij)}_{1}\ge\cdots\ge\mu^{(ij)}_{N}$ be the ordered eigenvalues of $H_{(ij)}$ and, let $\uu_{1}^{(ij)},\cdots,\uu_{N}^{(ij)}$ be the associated unit eigenvectors of $H_{(ij)}$. Using \eqref{eq: compare lambda_1 mu_1}, we find
\begin{align*}
\lambda_{1}  \ge \langle \uu_{1}^{(ij)},H\uu_{1}^{(ij)}\rangle  = \mu^{(ij)}_{1}+\langle\uu_{1}^{(ij)},(H-H_{(ij)})\uu_{1}^{(ij)}\rangle \ge \mu^{(ij)}_{1}-2(|h_{ij}|+|h_{ij}''|)\lVert \uu_{1}^{(ij)} \rVert_{\infty}^{2} .
\end{align*}
Similarly, reversing the role of $H$ and $H_{(ij)}$, we get
$$
\mu^{(ij)}_{1}  \ge \lambda_1 -2(|h_{ij}|+|h_{ij}''|)\lVert \vv_{1} \rVert_{\infty}^{2}.  
$$
From \eqref{eq:momenthij} we have $|h_{ij}| \prec 1/q$. Hence, by Lemma \ref{lem: delocalization} and $q \gg N^{c+\delta}$, we obtain 
\begin{align}\label{eq: eigvenvalue perturbation}
\max_{1\le i\le j\le N} |\lambda_{1}-\mu^{(ij)}_{1}| \prec \frac{1}{qN} \ll \frac{1}{N^{1+c}}.
\end{align}
We decompose $\uu_{1}^{(ij)}$ in the eigenvector basis of $H$:
\begin{align*}
\uu_{1}^{(ij)} = \sum_{\ell=1}^{N}\alpha_{\ell}\vv_{\ell}.
\end{align*}
We write two expressions for $H\uu_{1}^{(ij)}$:
\begin{align*}
H\uu_{1}^{(ij)} = \sum_{\ell=1}^{N}\lambda_{\ell}\alpha_{\ell}\vv_{\ell} = (H-H_{(ij)})\uu_{1}^{(ij)}+(\mu_{1}^{(ij)} - \lambda_{1}) \uu_{1}^{(ij)} +\lambda_{1} \uu_{1}^{(ij)}.
\end{align*}
We deduce that
\begin{align*}
\lambda_{1} \uu_{1}^{(ij)} = \sum_{\ell=1}^{N}\lambda_{\ell}\alpha_{\ell}\vv_{\ell} + (H_{(ij)}-H)\uu_{1}^{(ij)}+(\lambda_{1} - \mu_{1}^{(ij)}) \uu_{1}^{(ij)}.
\end{align*}
Next, by taking an inner product with $\vv_{\ell}$ for $\ell\neq 1$, we obtain
\begin{align*}
\lambda_{1}\alpha_{\ell} = \lambda_{1} \langle \vv_{\ell},\uu_{1}^{(ij)} \rangle = \langle \vv_{\ell},\lambda_{1}\uu_{1}^{(ij)} \rangle = \lambda_{\ell}\alpha_{\ell} + \langle \vv_{\ell},(H_{(ij)}-H)\uu_{1}^{(ij)} \rangle + (\lambda_{1} - \mu_{1}^{(ij)})\alpha_{\ell}.
\end{align*}
In other words,
\begin{align}\label{eq:l1ll1}
\bigg( (\lambda_{1}-\lambda_{\ell})+ (\mu_{1}^{(ij)}-\lambda_{1}) \bigg) \alpha_{\ell} = \langle \vv_{\ell},(H_{(ij)}-H)\uu_{1}^{(ij)} \rangle.
\end{align}
Let $\eps >0$. According to Corollary \ref{cor: lower bound for spectral gap}, there exists $c'>0$  such that on the event $\{\lambda_1 - \lambda_2 \geq N^{-1-c} \}$, with overwhelming probability:
\begin{align}\label{eq:lambda1lambdal2}
\lambda_{1}-\lambda_{\ell} \ge\begin{cases}
c'  N^{-1-c} & 2 \le \ell \le N^{\eps}, \\
c' \ell^{2/3}N^{-2/3} & N^{\eps} < \ell \le N.
\end{cases}
\end{align}
Since Lemma \ref{lem: delocalization} implies
\begin{align*}
\left| \langle \vv_{\ell},(H_{(ij)}-H)\uu_{1}^{(ij)} \rangle \right| \le 4 |h_{ij}''-h_{ij}| \lVert \vv_{\ell}\rVert_{\infty} \lVert \uu_{1}^{(ij)}\rVert_{\infty} \prec \frac{1}{qN},
\end{align*}
we deduce from \eqref{eq:l1ll1} and \eqref{eq: eigvenvalue perturbation} that, on the event $\{\lambda_1 - \lambda_2 \geq N^{-1-c} \}$,
\begin{align*}
(\lambda_{1}-\lambda_{\ell}) \cdot|\alpha_{\ell}| \prec \frac{1}{qN}.
\end{align*}
Thus, combining this last inequality with \eqref{eq:lambda1lambdal2}, we obtain, on the event $\{\lambda_1 - \lambda_2 \geq N^{-1-c} \}$, 
\begin{align}\label{eq : bounds for coeffi}
|\alpha_{\ell}|\prec\begin{cases}
q^{-1}N^{c}& 2\le \ell \le N^{\eps} \\
q^{-1}\ell^{-2/3}N^{-1/3}& N^{\eps} < \ell \le N.
\end{cases}
\end{align}
On the other hand, by setting $s=\alpha_{1}/|\alpha_{1}|$, we have
\begin{align*}
\lVert s\vv_{1}-\uu_{1}^{(ij)}\rVert_{\infty} &= \lVert (s-\alpha_{1})\vv_{1}+\textstyle{\sum_{\ell\neq 1}\alpha_{\ell}\vv_{\ell}} \rVert_{\infty} \nn\\
&\le (1-|\alpha_{1}|)\lVert\vv_{1}\rVert_{\infty} + \textstyle{\sum_{\ell\neq 1}}|\alpha_{\ell}|\lVert\vv_{\ell}\rVert_{\infty} \nn\\
&\prec N^{-1/2} \textstyle{\sum_{\ell\neq 1}}|\alpha_{\ell}| ,
\end{align*}
where on the last line, we have used that $1 - |\alpha_1|  = 1  -  \sqrt{ 1 - \sum_{\ell \ne 1} \alpha_\ell^2 } \leq \sum_{\ell \ne 1} |\alpha_\ell|$.
Using \eqref{eq : bounds for coeffi}, we finally obtain
\begin{align*}
\lVert s\vv_{1}-\uu_{1}^{(ij)}\rVert_{\infty} &\prec N^{-1/2}q^{-1}N^{c+\eps}+N^{-1/2}q^{-1}N^{-1/3}\textstyle{\sum_{N^{\eps}<\ell \le N}\ell^{-2/3}}\nn\\
&\prec q^{-1}N^{-1/2+c+\eps} + q^{-1}N^{-1/2}.
\end{align*}
Therefore we complete the proof 
by choosing $\eps>0$ small enough so that $q \gg N^{c + \delta + \eps}$. We can handle the case $H^{[k]}_{(ij)}$ similarly since $H^{[k]}$ and $H$ have the same law. \qed

\subsection{Proof of Lemma \ref{lem: expectation estimate}}
Recall that $\cE = \cE_1 \cap \cE_2$ where the events $\mathcal{E}_{1}$ and $\mathcal{E}_{2}$ are defined in \eqref{eq: event 1 conditions} and \eqref{eq: event 2 conditions}. We start by proving the first statement of Lemma \ref{lem: expectation estimate}. We split the expectation into two parts.
\begin{align}\label{eq:ZZvvE}
\E\left[ Z_{st}Z_{st}^{[k]}v_{s}v_{t}v_{s}^{[k]}v_{t}^{[k]}\indic_{\mathcal{E}^{c}} \right]
=\E\left[ Z_{st}Z_{st}^{[k]}v_{s}v_{t}v_{s}^{[k]}v_{t}^{[k]}\indic_{\mathcal{E}_{1}\cap\mathcal{E}_{2}^{c}} \right]+\E\left[ Z_{st}Z_{st}^{[k]}v_{s}v_{t}v_{s}^{[k]}v_{t}^{[k]}\indic_{\mathcal{E}_{1}^{c}} \right].
\end{align}
Using Cauchy-Schwarz inequality, we find
$$
\left| \E\left[ Z_{st}Z_{st}^{[k]}v_{s}v_{t}v_{s}^{[k]}v_{t}^{[k]}\indic_{\mathcal{E}_{1}^{c}} \right] \right| \leq  \E\left[ |Z_{st}Z_{st}^{[k]}| \indic_{\mathcal{E}_{1}^{c}} \right] \leq  \sqrt{ \E\left[ |Z_{st}Z_{st}^{[k]}|^2  \right] \PP ( \mathcal{E}_{1}^{c} )}.
$$
By Lemma \ref{lem: delocalization} the event $\cE_1$ holds with overwhelming probability and by Equation \eqref{eq:momenthij}, we have  $\E |Z_{st}Z_{st}^{[k]}|^2 \lesssim 1 / (N q^2)$. It follows that for any $C >0$, 
$$
\left| \E\left[ Z_{st}Z_{st}^{[k]}v_{s}v_{t}v_{s}^{[k]}v_{t}^{[k]}\indic_{\mathcal{E}_{1}^{c}} \right] \right| = O( N^{-C} ).
$$
We now turn to the first term on the right-hand side of \eqref{eq:ZZvvE}. With $\eps >0$ as in the definition of $\cE_1$, we want to show that
\begin{align}\label{eq: kapppa estimate 1 pf}
	\E\left[\big|Z_{st}Z_{st}^{[k]}\big|\indic_{\mathcal{E}_{2}^{c}}\right] \ll  N^{-1 - 4 \eps} ,
\end{align}
which implies
\begin{align}\label{eq: kapppa estimate 1}
\E\left[ Z_{st}Z_{st}^{[k]}v_{s}v_{t}v_{s}^{[k]}v_{t}^{[k]}\indic_{\mathcal{E}_{1}\cap\mathcal{E}_{2}^{c}} \right]\le N^{4\eps -2}  \E\left[\big|Z_{st}Z_{st}^{[k]}\big|\indic_{\mathcal{E}_{2}^{c}}\right]\ll N^{-3}.
\end{align}
There is a dependence between $Z_{st}Z_{st}^{[k]}$ and the event $\cE_{2}$. To circumvent this difficulty, we introduce some new events. Let $c >4\eps $ such that $c + \delta < 1/9$ (and $\delta$ as in the definition of $\cE_2$). we consider the event $\cE_{3} = \cE_{3,0} \cup \cE_{3,1}$ where 
\begin{equation}
\label{eq:defE301}\cE_{3,0}  =  \left\{ \min \left( \lambda_1 - \lambda_2 , \lambda_1^{[k]} - \lambda_2^{[k]} \right) \geq  N^{-1-c} \right\} \; \hbox{ and } \;  \cE_{3,1} =   \left\{ \min \left( \mu_1 - \mu_2 , \mu_1^{[k]} - \mu_2^{[k]} \right) \geq N^{-1-c} \right\}.
\end{equation}

By Lemma \ref{lem: eigen vec small perturbation}, for any $C >0$, we have $\PP ( \cE^c_2 \cap \cE_3) = O (N^{-C})$. Therefore, arguing as above, it is sufficient to prove that 
\begin{equation*}
\E\left[\big|Z_{st}Z_{st}^{[k]}\big|\indic_{\cE_{3}^{c}}\right] \ll  N^{-1 - 4 \eps} .
\end{equation*}
We note that
\begin{align}
	\E\left[\big|Z_{st}Z_{st}^{[k]}\big|\indic_{\mathcal{E}_{3}^{c}}\right] 
	&\le \frac 1 2 \E\left[ \big( Z_{st}^{2}+(Z_{st}^{[k]})^{2}\big)\indic_{\mathcal{E}_{3}^{c}}\right] \nn\\
	&\lesssim\E\left[ \big(h_{st}^{2}+(h'_{st})^{2}+(h_{st}'')^{2}+(h_{st}''')^{2}\big)\indic_{\mathcal{E}_{3}^{c}}\right] \nn\\
	&\lesssim \E\left[ (h_{st}^{2}+(h'_{st})^{2})\indic_{\cE_{3,1}^c} \right] + \E\left[((h_{st}'')^{2}+(h_{st}''')^{2})\indic_{\cE_{3,0}^c}\right] .\label{eq:decouplingJaehun}
\end{align}
We now use by construction the variables $(h_{st},h'_{st})$ are independent of the event $\cE_{3,1}$. We get by Lemma \ref{lem: tail bound for gaps} that
$$
\E\left[ (h_{st}^{2}+ (h'_{st})^{2})\indic_{\cE_{3,1}^c} \right] = O ( N^ {-1 - c } \log N).
$$
Similarly, $h_{st}''$ and $h_{st}'''$ are independent of the event $\cE_{3,0}$ and 
$$
\E\left[ ((h_{st}'')^{2}+(h_{st}''')^{2})\indic_{\cE_{3,0}^c} \right] = O ( N^ {-1 - c } \log N).
$$ 
Since we choose $c > 4\eps$, it concludes the proof of \eqref{eq: kapppa estimate 1 pf} and of the first claim of Lemma \ref{lem: expectation estimate}.

We now prove the second statement of Lemma \ref{lem: expectation estimate}. As above we decompose $\cE^c$ as the disjoint union of $\cE^c_1$ and $\cE_1 \cap \cE_2^c$. We have $|Q_{ij}| \prec 1/(Nq^2)$ and $|Z_{ij}| \prec 1/q$. Since $\cE_1$ holds with overwhelming probability, we find from \eqref{eq: top eig val difference under sigle resample}, that for any $C >0$,
$$
\E\left[ \big| (\lambda_{2}-\mu_{2}-Q_{st})(\lambda_{2}^{[k]}-\mu_{2}^{[k]}-Q_{st}^{[k]})\big| \indic_{\mathcal{E}^c_{1}} \right] = O( N^{-C}).
$$

We now deal with the event  $\cE_1 \cap \cE_2^c$. From \eqref{eq: top eig val difference under sigle resample} and arguing as in \eqref{eq: kapppa estimate 1}, we find
$$
\E\left[ \big| (\lambda_{2}-\mu_{2}-Q_{st})(\lambda_{2}^{[k]}-\mu_{2}^{[k]}-Q_{st}^{[k]})\big| \indic_{\mathcal{E}_{1} \cap  \cE^c_2 } \right] \leq N^{4\eps - 2}  \E\left[\max(|Z_{st}|,N |Q_{st}|) \cdot \max (| Z_{st}^{[k]}| , N |Q^{[k]}_{st}| ) \indic_{\mathcal{E}_{2}^{c}}\right].
$$
We observe that 
$$
\max(|Z_{st}|,N |Q_{st}|)  \leq |h_{st}| + |h''_{st}| + |h_{st}|^2 + |h''_{st}| ^2 ,
$$
the right-hand side is at most  $2(|h_{st}| + |h''_{st}| )$ with overwhelming probability since $|h_{ij}| \prec 1/q$. The same comment applies to $\max (| Z_{st}^{[k]}| , N |Q^{[k]}_{st}| )$. It follows that 
$$
 \E\left[\max(|Z_{st}|,N |Q_{st}|) \cdot \max (| Z_{st}^{[k]}| , N |Q^{[k]}_{st}| ) \indic_{\mathcal{E}_{2}^{c}}\right] \lesssim \E\left[ \big(h_{st}^{2}+(h'_{st})^{2}+(h_{st}'')^{2}+(h_{st}''')^{2}\big)\indic_{\mathcal{E}_{2}^{c}} \right].
$$
Finally, by Lemma \ref{lem: eigen vec small perturbation} on the event $\cE_3$, $\cE_2$ has overwhelming probability. We thus may substitute in the above inequality $\indic_{\mathcal{E}_{2}^{c}}$ by $\indic_{\mathcal{E}_{3}^{c}} $. We are then back to the upper bound in \eqref{eq:decouplingJaehun}. The conclusion follows. \qed

\subsection{Proof of Lemma \ref{lem: expectation estimate2}}\label{sec: expectation computation}
Integrating over the random pair $(st)$, we have
\begin{align*}
\E\left[ Z_{st}Z_{st}^{[k]}v_{s}v_{t}v_{s}^{[k]}v_{t}^{[k]} \right] = 
\frac{2}{N(N+1)}\E\left[\sum_{1\le i\le j\le N}Z_{ij}Z_{ij}^{[k]}v_{i}v_{j}v_{i}^{[k]}v_{j}^{[k]}\right].
\end{align*}
For brevity, we set  $V_{ij} = v_{i}v_{j}v_{i}^{[k]}v_{j}^{[k]}$. We split the above sum on the right-hand side into two parts,
\begin{align*}
\sum_{(ij)\in S_{k}}Z_{ij}Z'_{ij}V_{ij} +
\sum_{(ij)\notin S_{k}}Z_{ij}Z''_{ij}V_{ij},
\end{align*}
where $1\le i\le j\le N$ in both sums, $Z_{ij}=(h_{ij}-h_{ij}'')(1+\indic(i\neq j))$, $Z'_{ij}=(h_{ij}'-h_{ij}'')(1+\indic(i\neq j))$ and $Z''_{ij}=(h_{ij}-h_{ij}''')(1+\indic(i\neq j))$. Note that
\begin{align*}
\E\left[ Z_{ij}Z'_{ij} \right] = \E\left[ Z_{ij}Z''_{ij} \right] = \begin{cases}
\frac{4}{N} &\text{if } i< j,\\
\frac{1}{N} &\text{if } i=j.
\end{cases}
\end{align*}
Due to the dependence, it is tricky to compute $\E[Z_{ij}Z_{ij}^{[k]}V_{ij}]$ directly. Thus, we introduce the conditional expectation $\E [ \cdot |S _k]$ for given $S_k$ to avoid this issue. We shall first estimate 
$$\E \left[ \sum_{1\le i\le j\le N}\E\left[Z_{ij}Z_{ij}^{[k]}|S_k\right]V_{ij} \right],$$
and then show the contribution of
$$\E\left[ \sum_{1\le i\le j\le N}\left(Z_{ij}Z_{ij}^{[k]}-\E\left[Z_{ij}Z_{ij}^{[k]}|S_k\right]\right)V_{ij} \right]$$
is negligible. We start by computing 
\begin{align*}
\sum_{1\le i\le j\le N}\E\left[Z_{ij}Z_{ij}^{[k]}|S_k\right]V_{ij} = \sum_{(ij)\in S_{k}} \E\left[Z_{ij}Z_{ij}'\right] V_{ij} + \sum_{(ij)\notin S_{k}}\E\left[Z_{ij}Z_{ij}''\right]V_{ij}.
\end{align*}
Using the explicit expression for the expectations, we obtain 
\begin{align*}
\sum_{i\le j}\E\left[Z_{ij}Z_{ij}^{[k]}|S_k\right]V_{ij}&=
\frac{4}{N}\sum_{i<j}V_{ij}+\frac{2}{N}\sum_{i=j}V_{ij}
-\frac{1}{N}\sum_{i=j}V_{ij}\\
& = \frac{2}{N} \langle \vv , \vv^{[k]} \rangle^2 + O \left( \frac{1}{N} \sum_{i} |V_{ii}| \right).
\end{align*}
The last sum of the above equation is negligible. Indeed, using the delocalization of eigenvectors (Lemma \ref{lem: delocalization}), we have 
$$
 \sum_{i} |V_{ii}| \prec N^{-1}.
$$
Since $|V_{ij}| \leq 1$, we deduce in particular that 
\begin{align*}
\E \left[ \sum_{1\le i\le j\le N}\E\left[Z_{ij}Z_{ij}^{[k]}|S_k\right]V_{ij} \right]= \frac{2}{N} \E \langle \vv , \vv^{[k]} \rangle^2 + o \left(\frac 1 N  \right).
\end{align*}

To conclude the proof of the lemma, what remains to show is
\begin{align*}
\E\left[\sum_{1\le i\le j\le N}W_{ij} V_{ij}\right]=
 o \left(\frac 1 N  \right),
\end{align*}
where we have set $W_{ij}:=Z_{ij}Z_{ij}^{[k]}-\E\left[Z_{ij}Z_{ij}^{[k]} |S_k \right]$. For the remainder of this proof, we fix a pair $(i,j)$, $1 \leq i \leq j \leq N$. It is sufficient to check that $\E [ W_{ij} V_{ij} ] = o ( 1/ N^3)$ where the $o(\cdot)$ is uniform over the choice of the pair $(i,j)$.

Let $h_{ij}''''$ be a independent copy of $h_{ij}$ which is also independent of $(H,H',H'',H''')$. Similarly to $H_{(ij)}$ and $H_{(ij)}^{[k]}$, we can define analogously $\tilde{H}_{(ij)}$ and $\tilde{H}_{(ij)}^{[k]}$ by replacing $(i,j)$-element with $h_{ij}''''$. Denote by $\tilde{\uu}_{1}= (\tilde{u}_1, \ldots, \tilde{u}_N)$ and $\tilde{\uu}_{1}^{[k]} = (\tilde{u}_{1}^{[k]},\ldots, \tilde{u}_N^{[k]})$ the top eigenvectors of $\tilde{H}_{(ij)}$ and $\tilde{H}_{(ij)}^{[k]}$ respectively. To ease the notation, we define
\begin{align*}
U_{ij}:= \tilde{u}_{i} \tilde{u}_{j} \tilde{u}_{i}^{[k]} \tilde{u}^{[k]}_{j}.
\end{align*}
By construction, we have
\begin{align*}
\E\left[ W_{ij}U_{ij}   \right] = \E \left[ \E\left[W_{ij}|S_k\right]\cdot \E\left[ U_{ij}|S_k  \right] \right] = 0,
\end{align*}
because, given $S_k$, the pair $(Z_{ij},Z_{ij}^{[k]})$ only depends on $(h_{ij},h_{ij}',h_{ij}'',h_{ij}''')$ while $(\tilde{\uu}_1,\tilde{\uu}_{1}^{[k]})$ is independent of $(h_{ij}, h_{ij}',h_{ij}'',h_{ij}''')$. Thus, it is enough to show
\begin{align}\label{eq:WVUij}
\E\left[ W_{ij}\left( V_{ij} - U_{ij} \right)\right]=o\left(\frac 1 {N^3}\right).
\end{align}
The proof of \eqref{eq:WVUij} is performed as the proof of Lemma \ref{lem: expectation estimate}. Let $\delta >0$, $0 < \eps < \delta/3$,  $c >4\eps $ be such that $c + \delta < 1/9$. We consider the event $\tilde{\cE}_1$ defined as $\cE_1$ but with $\tilde{H}_{(st)}$ and $\tilde{H}_{(st)}^{[k]}$ replacing $H_{(st)}$ and $H^{[k]}_{(st)}$ :
$$
	\tilde{\mathcal{E}}_{1} =	\left\{ \max \left(\lVert \vv_1   \rVert_{\infty},\lVert \tilde{\uu}_1 \rVert_{\infty},\lVert  \vv_{1}^{[k]}\rVert_{\infty},\lVert \tilde{\uu}_1^{[k]} \rVert_{\infty}\right) \le  N^{\eps - 1/2} \right\}.
$$ 
Similarly, if $\{\tilde{\mu}_i\}_{i=1}^{N}$ and $\{\tilde{\mu}_i^{[k]}\}_{i=1}^{N}$ are the eigenvalues of  $\tilde{H}$ and $\tilde{H}^{[k]}$  respectively, we consider the event $\tilde{\cE}_{3} = \cE_{3,0} \cup \tilde{\cE}_{3,1}$ with $\cE_{3,0}$ defined by \eqref{eq:defE301} and 
$$
\tilde{\cE}_{3,1} =  \left\{ \min \left( \tilde{\mu}_1 - \tilde{\mu}_2 , \tilde{\mu}_1^{[k]} - \tilde{\mu}_2^{[k]} \right) \geq N^{-1-c} \right\},
$$

As in the proof of Lemma \ref{lem: expectation estimate}, we use
Lemma \ref{lem: eigen vec small perturbation} to deduce that on the event $\tilde{\mathcal{E}}_{3}$, we have
\begin{align*}
\lVert \vv_{1} - \tilde{\uu}_{1} \rVert_{\infty}\le N^{-\frac{1}{2}-\delta} \quad \text{and}\quad 
\lVert \vv_{1}^{[k]} - \tilde{\uu}_{1}^{[k]} \rVert_{\infty}\le N^{-\frac{1}{2}-\delta},
\end{align*}
with overwhelming probability after choosing the phases of $\tilde{\uu}_{(ij)}$ and $\tilde{\uu}_{(ij)}^{[k]}$ properly. Hence, on the event $\tilde{\cE}=\tilde{\cE}_{1} \cap \tilde{\cE}_{3}$, we find that the bound 
\begin{align*}
|V_{ij} - U_{ij}| \leq C N^{3\eps - 2 - \delta}
\end{align*}
holds with overwhelming probability for some $C >0$. Also, by Lemma \ref{lem: delocalization}, the event $\tilde{\cE_1}$ holds with overwhelming probability. Using that $|V_{ij}|, |U_{ij}|\leq 1$, we deduce  that for any $C >0$, 
$$
|\E\left[ W_{ij}\left( V_{ij} - U_{ij} \right)\right]|  \lesssim N^{3\eps - 2 - \delta}\E  |W_{ij}|    + N^{4\eps - 2} \E  [ |W_{ij}|\indic_{  \tilde{\cE}_{3}^c} ] + N^{-C}. 
$$
Since $\E[|Z_{ij}Z_{ij}^{[k]}|]=O(1/N)$, we have $\E  |W_{ij}|  =O(1/N)$ and thus the first term on the right-hand side of the above equation is $o(1/N^3)$. For the second term, we simply write that 
$$
\E [ |W_{ij}|\indic_{  \tilde{\cE}_{3}^c} ]   \lesssim \E \left[ \left( h_{ij}^2  + (h'_{ij})^2 + (h''_{ij})^2 + (h'''_{ij})^2 + \frac 1{N} \right) \indic_{  \tilde{\cE}_{3,1}^c}\right] = O \left( N^{-1-c} \log N\right),
$$
where we have used the independence of $\tilde{\cE}_{3,1}$ and $(h_{ij},h'_{ij},h''_{ij}, h'''_{ij})$ and invoked Lemma \ref{lem: tail bound for gaps}. Since $4 \eps < c$, this concludes the proof of \eqref{eq:WVUij}. \qed

\section{Noise stability of the top-eigenvector}\label{sec: thm2}

\subsection{Preliminaries on the resolvent matrix}

Our first lemma is used to detect the largest eigenvalues from the diagonal entries of the resolvent $R(z) = (H - zI)^{-1}$ for $z$ close enough to $\cL$.
\begin{lem}\label{lem: lem9 in BLZ}
Assume $q \gtrsim N^{1/9}$. For any integer $1\leq j \leq N$, there exists a random integer $1\le i\le N$ such that for all $E$ and $\eta>0$
	\begin{align*}
\big(\max(\eta,|\lambda_{j}-E|)\big)^{-2}\le 2 N\eta^{-1}\Im R(E+\mathfrak{i}\eta)_{ii}.
	\end{align*}
The other way around, let $\eps >0$. With overwhelming probability, for all integers $1\le i\le N$ and all $E$ such that $|E-\mathcal{L}| \leq N^{-2/3+\eps}$, we have
	\begin{align*}
	N\eta^{-1}\Im R(E+\mathfrak{i}\eta)_{ii} \leq N^{4\eps} \left(\min_{1\le j\le N}|\lambda_{j}-E|\right)^{-2}.
	\end{align*}
\end{lem}
	\begin{proof} From the spectral theorem,
		\begin{align*}
		N\eta^{-1}\Im R(E+\mathfrak{i}\eta)_{ii}=\sum_{p=1}^{N}\frac{N(\vv_{p}(i))^{2}}{(\lambda_{p}-E)^{2}+\eta^{2}} \geq  \frac{N(\vv_{j}(i))^{2}}{(\lambda_{j}-E)^{2}+\eta^{2}}\ge \frac{N(\vv_{j}(i))^{2}}{2\big(\max(\eta,|\lambda_{j}-E|)\big)^{2}},
		\end{align*}
The first statement follows since there exists $i$ such that $|\vv_{j}(i)|\ge N^{-1/2}$.

		Next, we prove the second statement. Fix $\eps>0$ and consider $E$ satisfying $|E-\mathcal{L}|\le N^{-2/3+\eps}$. From \eqref{eq:riglambda1}, with overwhelming probability $|\lambda_1 - E | \leq 2N^{-2/3 +\eps}$. Thus, from Lemma \ref{lem: delocalization} and Lemma \ref{lem: new rigidity estimate}, for some $c>0$, with overwhelming probability the following event holds: (i) $\max_{1\le p \le N}\lVert\vv_{p} \rVert_{\infty}^{2}\le N^{-1+\eps}$, (ii) $|\lambda_1 - E | \leq 2N^{-2/3 +\eps}$ and (iii) for all such $E$ with $|E-\mathcal{L}|\le N^{-2/3+\eps}$, we have $E-\lambda_{p}\ge cp^{2/3}N^{-2/3}$ for all integer $p> N':=\lfloor N^{2\eps}\rfloor$. On this event, from (i) and (iii), we have for some $C>0$
		\begin{align*}
		\sum_{p=N'+1}^{N}\frac{N(\vv_{p}(i))^{2}}{(\lambda_{p}-E)^{2}+\eta^{2}}\le
		\sum_{p=N'+1}^{N}\frac{N^{\eps}}{(\lambda_{p}-E)^{2}}\le CN^{\eps}(N')^{-1/3}N^{4/3},
		\end{align*}
		and
		\begin{align*}
		\sum_{p=1}^{N'}\frac{N(\vv_{p}(i))^{2}}{(\lambda_{p}-E)^{2}+\eta^{2}} \le
		\frac{N^{\eps}N'}{\left(\min_{1\le j\le N}|\lambda_{j}-E|\right)^{2}}.
		\end{align*}
Finally from (ii), for all $N$ large enough, we have $CN^{\eps}(N')^{-1/3}N^{4/3} \leq N^{\eps}N'\left(\min_{1\le j\le N}|\lambda_{j}-E|\right)^{-2}$. This proves the second statement.
	\end{proof}

The following lemma on the resolvent of sparse random matrices will be crucial to study the resolvent process indexed by the successive resampled entries. Below, we use the Kronecker delta symbol: $\delta_{ij} =\indic_{i = j}$.
\begin{lem}\label{lem: lem10 in BLZ}
Assume $q \gtrsim N^{1/9}$ and let $0 < \delta < 1/3$. We have
	\begin{align*}
		\sup_z \max_{1 \leq i,j \leq N} \Big|\big|R(z)_{ij}\big| - \delta_{ij} \Big| \prec \frac 1 q + \frac{1}{N\eta}, 	\end{align*}
	and 
	\begin{align*}
	\sup_z  \max_{1 \leq i,j \leq N} \big|\Im R(z)_{ij}\big|\prec \frac{1}{N\eta} ,
	\end{align*}
	where the two suprema are over all $z = E + \mathfrak{i}\eta$ with $|E-\cL| \leq N^{-2/3 + \delta}$ and $\eta = N^{-2/3 -\delta}$.
\end{lem}

The first statement of the lemma is a consequence of \cite[Theorem 2.8]{EKYY13} and the norm estimate of $m_{\text{sc}}(z)$.
The second statement is new, it uses notably the improved local law for the Cauchy-Stieltjes transform $m(z)$ given in Lemma \ref{lem:locallaw2}.
 We postpone its proof to Subsection \ref{sec: estimate on the imaginary part of local law}. 
We are now ready to prove Lemma \ref{lem: lem13 in BLZ}.

\subsection{Proof of Lemma \ref{lem: lem13 in BLZ}}\label{sec: main argument for noise stability}

\paragraph{Step 1: net argument.} We have $|R_{ij}(z) - R_{ij}(z')| \leq |z - z'|/ \eta^2$ where $\eta = \min(\Im(z),\Im(z'))$. Hence, by a standard net argument where we partition the interval $[-N^{-2/3 + \delta} , N^{-2/3 + \delta}]$ into  $N^2$ sub-intervals, it suffices to prove the conclusion of Lemma \ref{lem: lem13 in BLZ} for any fixed $\kappa$ real with $|\kappa| \leq N^{-2/3 + \delta}$,   $z= E  +\mathfrak{i}\eta$ where $E = \cL + \kappa$ and $\eta = N^{-2/3-\delta}$. Moreover, from \eqref{eq:proxyL} and Lemma \ref{lem: lem10 in BLZ}, it is sufficient to prove that for any deterministic real $\kappa$ with $|\kappa| \leq  2 N^{-2/3 + \delta}$, 
\begin{equation}\label{eq:tbpimRRk}  N\eta|\Im R_{ij}^{[k]}(\tz)-\Im R_{ij}(\tz)|\prec N^{-c} ,
 \end{equation}
 uniformly in $1 \leq i , j \leq N$, 
with 
\begin{equation}\label{eq:deftz}
\tz = \kappa + L + \cX + \mathfrak{i}\eta,
\end{equation} and  $L$ deterministic as in \eqref{eq:proxyL}.
In the remainder of the proof, we fix such $\kappa$ and corresponding random $\tz$.

\paragraph{Step 2: shifted resolvent matrix.} The random variable $\tz$ depends on the entries $h_{i_t,j_t}$, $1 \le t \le k$. To avoid this, we set 
$$
\hz =   \kappa + L + \hat\cX + \mathfrak{i}\eta,
$$
where $$
\hat \cX  = \cX - \frac{1}{N}\sum_{t=1}^k (1 + \indic(i_t \ne j_t) ) \left(h_{i_tj_t}^{2}-\frac{1}{N}\right).
$$
By construction, from \eqref{eq:momenthij}, we have 
$$
|\hz  - \tz | \prec \max\left( \frac{1}{Nq^{2}} , \frac{\sqrt{k}}{N^{3/2} q} \right).  
$$
Recall the resolvent identity.
\begin{align*}\label{eq: resolvent identity}
	(X-zI)^{-1}=(Y-zI)^{-1}+(Y-zI)^{-1}(Y-X)(X-zI)^{-1}
\end{align*}
and the Ward identity for the resolvent:  for any integers $i,j$, 
\begin{equation}\label{eq:ward}
\sum_{l=1}^N R_{il}(z) \bar R_{jl}(z) = (R(z)R^*(z))_{ij} = \frac{ \Im R_{ij}(z)}{\Im(z)}.
\end{equation}
It implies that 
\begin{align*}
| R_{ij}(\tz) - R_{ij}(\hz) | \leq & |\hz  - \tz |  \sum_{l} | R(\tz) _{il} | |R(\hz)_{lj} | \\
\le &  |\hz  - \tz |  \sqrt{ \sum_{l} | R(\tz) _{il} |^2} \sqrt{ \sum_{l} | R(\tz) _{lj} |^2} \\
\prec & \max\left( \frac{1}{Nq^{2}} , \frac{\sqrt{k}}{N^{3/2} q} \right) \frac{1}{N\eta^2},
\end{align*}
where we have used Cauchy-Schwarz inequality, \eqref{eq:ward} and  Lemma \ref{lem: lem10 in BLZ}. Since $k \ll N^{5/3}$, we have 
\begin{equation*}\label{eq:jefjrf}
\frac{\sqrt{k}} { N^{3/2}q\eta}  \ll \frac{N^{\delta}}{q} \ll 1,
\end{equation*}
provided that $N^{\delta}\ll q $.
We deduce that $N\eta | R_{ij}(\tz) - R_{ij}(\hz) | \prec N^{-c}$ for some $c >0$. The same conclusion holds for $R^{[k]}_{ij}(\tz) - R^{[k]}_{ij}(\hz)$. It follows that to prove \eqref{eq:tbpimRRk}, it is sufficient to prove that
\begin{equation}\label{eq:tbpimRRk2}
  N\eta|\Im R_{ij}^{[k]}(\hz)-\Im R_{ij}(\hz)|\prec N^{-c},
\end{equation}
uniformly in $1 \leq i , j \leq N$.

\paragraph{Step 3: fluctuation of the resolvent process.}
Now, for $0 \leq t \leq N(N+1)/2$, we define $R^{[t]} (z) = (H^{[t]} - z)^{-1}$ as the resolvent of $H^{[t]}$. Since no other value of the resolvent will be considered, for ease of notation, we omit the parameter $\hz$ and simply write  $R^{[t]}$ in place of $R^{[t]}(\hz)$.    From the resolvent identity, we get 
$$
 R^{[k]}_{ij} -  R_{ij}  = \sum_{t=1}^k  \left( R^{[t]}_{ij} -  R^{[t-1]}_{ij}\right)  = \sum_{t=1}^k  (h_{i_tj_t} - h'_{i_tj_t}) (R^{[t]} E_{i_t j_t} R^{[t-1]} ) _{ij}, 
$$
where $E_{ij}=\ee_{i}\ee_{j}^{T}+\ee_{j}\ee_{i}^{T}\indic(i\neq j)$ where $\ee_{i}$ denotes the canonical basis of $\R^{n}$ such that the $i$-th entry is equal to $1$ and the other entries are equal to $0$. We set 
$$
h_{t} = h_{i_tj_t}, \quad h'_t =  h'_{i_tj_t} , \quad E_t = E_{i_t j_t} \quad \hbox{ and } \quad G_t = N\eta \Im \left( (R^{[t]} E_{t} R^{[t-1]} ) _{ij} + (R^{[t]} E_{t} R^{[t-1]} ) _{ji}  \right) ,
$$
($G_t$ depends implicitly on $\{i,j\}$). Since $R_{ij} = R_{ji}$, we get that, by construction, 
 $$
N\eta  (   \Im R^{[k]}_{ij} -  \Im R_{ij} ) = \frac 1 2 \sum_{t=1}^k (h_t - h'_t) G_t.
 $$
 The main technical ingredient in the proof  of Lemma \ref{lem: lem13 in BLZ} is the following statement (note that for deterministic sequences of non-negative numbers, $U\prec V$ is equivalent to $U \leq N^{o(1)} V$). 

\begin{lem}\label{prop: large deviation}
Assume $q \gtrsim N^{1/9}$ and $k \ll N^2$. With the above notation, for any integer $r \geq 1$, uniformly in $i,j$,
	\begin{align}
	\E 	\left( \sum_{t=1}^{k} (h_t - h'_t) G_t \right)^{2r} \prec \left(\frac{k}{N^{3} \eta^2 } \right)^r + \left(\frac{k}{N^{3} \eta^2 } \right)   q^{2 - 2r}.
    \end{align}
\end{lem}

Before proving Lemma \ref{prop: large deviation} in the next subsection, let us conclude the proof of Lemma \ref{lem: lem13 in BLZ}. Since $\eta = N^{-2/3 - \delta}$, we have $k / (N^3 \eta^2 ) = N^{2 \delta} k / N^{5/3}$. Moreover, if $k \ll N^{5/3}$, we may find a small $ \delta_0 > 0$ such that $k \ll N^{5/3 - 3\delta_0}$.  We set $c = \delta_0$ and assume that $0 < \delta < \delta_0$. From Markov inequality, this concludes the proof of \eqref{eq:tbpimRRk2} and Lemma \ref{lem: lem13 in BLZ}.

\subsection{Proof of Lemma \ref{prop: large deviation}}

\paragraph{Step 1: moment expansion and symmetry.}
We set $y_t = h_t - h'_t$ and write
\begin{align*}
	\E \left( \sum_{t=1}^{k} y_{t}G_{t} \right)^{2r} = \sum_{t_{1},\ldots,t_{2r}} \E \big[ y_{t_{1}}G_{t_{1}}\cdots y_{t_{2r}} G_{t_{2r}} \big].
\end{align*}
Combining the terms with equal indices, we get
\begin{align*}
	\E \left( \sum_{t=1}^{k} y_{t}G_{t} \right)^{2r} =	 \sum_{m = 1}^{2r} \sum_{\rho} \sum_{T} \frac{(2r)!}{\prod_{l}  \rho_{l} !}
	\E \left[ \prod_{l =1}^m y^{\rho_l}_{t_l} G_{t_l}^{\rho_l} \right],
\end{align*}
where the second sum is over vectors $\rho = (\rho_l)$, with $\rho_1 + \ldots +  \rho_m = 2r$, $\rho_l \ge 1$  and the last sum is over all sequences $ T= (t_1,\ldots,t_m)$ pairwise disjoint $t_l$  in $\{1,\ldots, k\}$. Since $r$ is fixed, it is enough to fix in the remainder of the proof an integer $m$ in the above sum.

Despite the fact that $(G_{t})_{t \in T}$ is not independent of $(y_{t})_{t \in T}$, we start by checking that the contribution of vectors $\rho = (\rho_l)$ such that $\min_l \rho_l = 1$ is zero. More precisely, assume without loss of generality that $\rho_m = 1$. We set
\begin{align*}
K(m,\rho) = \sum_{T}  
	\E \left[ \prod_{l =1}^m y^{\rho_l}_{t_l} G_{t_l}^{\rho_l} \right] & = 	\E \left[ y_{t_m} G_{t_m}  \prod_{l =1}^{m-1} y^{\rho_{l}}_{t_{l}} G_{t_{l} }^{\rho_{l}} \right].
\end{align*}
We claim that $K(m,\rho) = 0$ if $\rho_m = 1$. Indeed, we can realize our random variables by considering the  $m$-tuple $((h''_1,h'''_1),\ldots,(h''_m,h'''_m))$ of iid copies of $h_{ij}$, independent of a uniform $m$-tuple $((i'_1,j'_1),\ldots,(i'_m,j'_m))$ of distinct elements in $\{ (i,j) : i \leq j \}$. Then for a given $T$ as in the above sum and $1 \leq l \leq m$, we set $(h_{t_l},h'_{t_l}) = (h''_l,h'''_l)$ and $(i_{t_l},j_{t_l})= (i'_l,j'_l)$. As a function of $(h''_{m},h'''_{m})$, $G_{t_m}$ is symmetric (because switching the values of $h_{m}$ and $h'_{m}$ maps $R^{[t_m]}$ to $R^{[t_{m}-1]}$ and maps  $R^{[t_{m}-1]}$ to  $R^{[t_{m}]}$). Moreover, as a function of $(h''_{m},h'''_{m})$, for $l \leq m-1$, $G_{t_l}$ is a function of $h''_{m}\indic(t_l < t_m) + h'''_{m} \indic(t_l > t_m)$. Summing over $T$ , it follows that 
$$
 \sum_{T} 
 \prod_{l =1}^{m-1} y^{\rho_{l}}_{t_{l}} G_{t_{l} }^{\rho_{l}} 
$$
is a symmetric function of $(h''_{m},h'''_{m})$. Indeed, consider the map $(t_1,\ldots,t_m) \mapsto (k+1-t_1,\ldots,k+1-t_m)$. This maps defines an involution on the set of $T$ in the above sum and its image on $\prod_{l =1}^{m-1} y_{t_l}^{\rho_l} G_{t_l}^{\rho_l}$ is symmetric in $(h''_{m},h'''_{m})$. Therefore recalling that $h_{t}$ and $h'_t$ have the same distribution, we get
$$
 \sum_{T} 
	\E \left[ h_{t_m} G_{t_m}  \prod_{l =1}^{m-1} y^{\rho_{l}}_{t_{l}} G_{t_{l} }^{\rho_{l}} \right] = \sum_{T} 
	\E \left[ h'_{t_m} G_{t_m}  \prod_{l =1}^{m-1} y^{\rho_{l}}_{t_{l}} G_{t_{l} }^{\rho_{l}} \right].
$$
Since $y_t = h_t - y'_t$, we get that $K(m,\rho) = 0$. 

We thus restrict ourselves to vectors $\rho = (\rho_l)$ such that 
$$\rho_l \geq 2, \quad \hbox{for all $1 \leq l \leq m$}.$$
Our goal is then to prove that, uniformly over all $T$ and such vectors $\rho$, we have
\begin{equation}\label{eq:tbpLD}
\E \left[ \prod_{l =1}^m y^{\rho_l}_{t_l} G_{t_l}^{\rho_l}\right] \prec   \left(\frac{1}{N^{3} \eta^2 } \right)^m q^{2(m-r)}.
\end{equation}
This immediately implies the statement of the lemma since $(i)$ $\min_l \rho_l \geq 2$ implies that $m \leq r$ and $(ii)$ there are at most $k^m$ choices for the elements of $T$.

\paragraph{Step 2: resolvent bound.}

In order to extract the moments of $y_{t}$ in \eqref{eq:tbpLD}, we shall use a decoupling argument using the resolvent expansion.  For $0 \leq s \leq k$, we define $\hat H^{[s]}$ as the symmetric matrix obtained from $H^{[s]}$ by setting the entries $(i_tj_t)_{t \in T}$ and $(j_ti_t)_{t \in T}$ to $0$. The resolvent of $\hat H^{[s]}$ at $\hz$ is denoted by $\hat R^{[s]}  = (\hat H^{[s]} - \hz)^{-1}$. We note that given $(i_tj_t)_{t \in T}$, the matrix $\hat H^{[s]}$ is independent of $(y_t)_{t \in T}$.
For ease of notation, we also set 
$$
\alpha = \frac 1 {N\eta} \quad \hbox{ and } \quad \beta = \frac 1 q + \frac{1}{N\eta}.
$$
Iterating the resolvent identity, we get
\begin{align}\label{eq:residit}
	\hat R^{[s]} - R^{[s]} = \sum_{p=1}^{8}\left(R^{[s]}(H^{[s]}-\hat H^{[s]})\right)^{p} R^{[t]}
	+ \left(R^{[s]}(H^{[s]}-\hat H^{[s]})\right)^{9}\hat R^{[s]}.
\end{align}
We have
\begin{align*}
	H^{[s]}-\hat H^{[s]} = \sum_{t \in T} ( h_{i_{t}j_{t}} \indic( t > s )  + h'_{i_{t}j_{t}} \indic( t \leq s ))  E_{i_{t}j_{t}}
	 \end{align*}
Recall the fact that $|h_{ij}|\prec q^{-1}$, $|R^{[s]}_{ij}|  \prec 1$ (by Lemma \ref{lem: lem10 in BLZ}) and $\lVert \hat R^{[s]}\rVert\le\eta^{-1}$. Since $q^9 \gtrsim N$, we deduce that
$$
|\hat R^{[s]}_{ij}  - R^{[s]}_{ij} | \prec \sum_{p=1}^8 \frac 1 {q^p} + \frac{1}{q^9 \eta} \prec \beta   .
$$
Similarly, using $|\Im(ab)| \leq |\Im(a)| |b| + |a| |\Im(b)|$, we find, by Lemma \ref{lem: lem10 in BLZ},
\begin{align*}
	\left| \Im \hat R^{[s]}_{ij} - \Im R^{[s]}_{ij}  \right| &\prec  \sum_{p=1}^8 \frac \alpha {q^p} + \frac{1}{q^9 \eta} \prec \alpha.
\end{align*}
Therefore, using again Lemma \ref{lem: lem10 in BLZ} and , we obtain 
	\begin{align}\label{eq:bdhatR}
\max_{1 \leq i,j \leq N} \Big| \big|\hat R^{[s]}_{ij}\big| - \delta_{ij} \Big|\prec \beta \quad 	\hbox{ and }	\quad \max_{1 \leq i,j \leq N} \Big|\Im \hat R^{[s]}_{ij}\Big|\prec \alpha.
	\end{align}
We are ready for the decoupling argument.

\paragraph{Step 3: decoupled resolvent.} 

The following lemma on stochastic domination is elementary.
\begin{lem}\label{lem:stdoexp}
Let $(U_N)$, $(V_N)$ be two sequences of non-negative random variables and $(u_N)$ be a non-negative sequence such that $U_N \prec u_N$. If there exist $C >0$ and $p ,q >0$ such that $1/p + 1 / q < 1$ and $(\E U_N^p )^{1/p} \prec N^C u_N$ and $(\E V_N^q )^{1/q} \prec N^C \E V_N$ then $\E [U_N V_N] \prec u_N\E[V_N]$.  
\end{lem}
\begin{proof}
Set $r$ such that $1/p + 1/q + 1/r =1$. From H\"{o}lder inequality, for any event $\cE$
$$
\E [U_N V_N] - \E [U_N V_N \indic_\cE] = \E [U_N V_N \indic_{\cE^{c}}] \leq (\E U_N^p )^{1/p}(\E V_N^q )^{1/q} \PP(\cE^{c})^{1/r}.
$$
For a fixed $\eps >0$, we consider the event $\cE = \{ U_N \leq N^{\eps} u_N\}$. Since $\cE$ has overwhelming probability, we deduce from the assumptions that  $\E [U_N V_N] \leq N^{\eps} u_N \E [ V_N \indic_\cE]  + o (u_N\E V_N)$. The conclusion follows.
\end{proof}

We set 
$$
\hat G_t =  N\eta \Im \left( (\hat R^{[t]} E_{t} \hat R^{[t-1]} ) _{ij} +   (\hat R^{[t]} E_{t} \hat R^{[t-1]} ) _{ji}\right).
$$
In this paragraph, we prove that \eqref{eq:tbpLD} holds when $G_t$ is replaced by $\hat G_t$. In the next and final step, we will prove that $G_t$ and $\hat G_t$ are close. From \eqref{eq:bdhatR}, we observe that for $t \in T$,  
$$
|\hat G_{t}| \prec  1, 
$$
Given $(i_tj_t)_{t \in T}$, $y_t$ is independent of $(\hat R^{[s]})_{0 \leq s \leq k}$. We deduce from \eqref{eq:momenthij}, Lemma \ref{lem:stdoexp} and the assumption $\min_l \rho_l \geq 2$ that
\begin{align}
\left|\E \left[ \prod_{t \in T} y^{\rho_t}_{t} \hat G_{t}^{\rho_t}\right]\right| & \lesssim \frac{q^{2(m-r)}}{ N^m} \E \left[ \prod_{t \in T} \hat G_{t} ^{2} \right] .\label{eq:ytgtrt}
\end{align}
Note that in the above expression, we have 
set $\rho_t := \rho_l$ if $t= t_l \in T$. 

Next, we estimate $\E \left[ \prod_{t \in T} \hat G_{t} ^{2} \right]$ in \eqref{eq:ytgtrt}. We first observe that $ \prod_{t \in T} (\alpha^2 \hat G_t^2)$ is a sum of products of the form 
$$
	\prod_{t \in T}  \Im\left( \hat R^{[s_{1t}]}_{a_{1t} i_t} \hat R^{[s_{2t}]}_{a_{2t} j_t} \right) \Im \left( \hat R^{[s_{3t}]}_{a_{3t} i_t}  \hat R^{[s_{4t}]}_{a_{4t} j_t}  \right) ,
$$
with $(a_{1t},a_{2t}), (a_{3t},a_{4t}) \in \{(i,j),(j,i)\}$ and $(s_{1t},s_{2t}), (s_{3t},s_{4t}) \in \{(t,t-1), (t-1,t)\}$. Using $|\Im(ab)| \leq |\Im(a)| |b| + |a| |\Im(b)|$, we deduce from \eqref{eq:bdhatR} that  
\begin{align}\label{eq: trick 1}
\prod_{t \in T}  \hat G_t^2 \prec \sum \prod_{t \in T}   \left| \hat R^{[s_{t}]}_{a_{t} b_t} \right| \left| \hat R^{[s'_{t}]}_{a'_{t} b'_t} \right|,
\end{align}
where the sum is over possible choices of  $a_t,a'_t$ in $\{i,j\}$, $b_t,b'_t$ in $\{i_t,j_t\}$ and $s_t,s'_t$ in $\{t-1,t\}$.

We now bound the right-hand side of \eqref{eq: trick 1}. Since $2 |ab| \leq |a|^2 +  |b|^2$, it suffices to treat the case $(a_t,b_t,s_t)  = (a'_t,b'_t,s'_t)$.  We denote by $\E_T$ the conditional expectation with respect to $\mathcal F_T$, the $\sigma$- algebra generated by $H$, $H'$ and $(i_sj_s)_{s \notin T}$ (in other words, we integrate only on $(i_tj_t)_{t\in T}$ given the rest of the variables). Since $r$ is fixed, we have 
\begin{align*}
\E_T \prod_{t \in T}   \left| \hat R^{[s_{t}]}_{a_{t} b_t} \right|^2  \lesssim \frac{1}{N^{2m}}  \sum_{w} \prod_{t \in T}  |\hat R^{[s_{t}]}_{a_{t} u_t} [w]|^2,
\end{align*}
where the sum is over all $w = ( (u_t,v_t) )_{t \in T}$ with $1\leq u_t,v_t \leq N$ and $\hat R^{[s]}[w]$ is the resolvent of the symmetric matrix $H^{[s]}[w]$ obtained from $H^{[s]}$ by setting the entries $(u_tv_t)_{t \in T}$ and $(v_tu_t)_{t \in T}$ to $0$ (that is $\hat R^{[s]} = \hat R^{[s]} [(i_t,j_t)_{t \in T}]$). We would like to apply Ward identity of the resolvent \eqref{eq:ward} in the above expression but the matrix $\hat R^{[s]}[w]$ depends on the summation index. 

To overcome this difficulty, we approximate $\hat R^{[s]}[w]$ by the resolvent of another carefully chosen matrix. For $T_0 \subset T$, let $W_{T_0}$  be the set of $w= ( (u_t,v_t) )_{t \in T}$ as above such that $\{ u_t ,v_t \} \cap \{i,j\} \ne \emptyset$ if and only if $t \in T_0$. If $w \in W_{T_0}$, we set $w_0 =  ( (u_t,v_t) )_{t \in T_0}$ and $w_1  = ( (u_t,v_t) )_{t \notin T_0}$. We write
\begin{align}\label{eq: trick 2}
\E_T \prod_{t \in T}   \left| \hat R^{[s_{t}]}_{a_{t} b_t} \right|^2  \lesssim \frac{1}{N^{2m}} \sum_{T_0 \subset T}  \sum_{w_0} \sum_{w_1}  \prod_{t \in T}  |\hat R^{[s_{t}]}_{a_{t} u_t} [w]|^2.
\end{align}

We next define $\hat R^{[s]}[w_0]$ as the resolvent of the symmetric matrix $H^{[s]}[w_0]$ obtained from $H^{[s]}$ by setting to the entries $(u_tv_t)_{t \in T_0}$ and $(v_tu_t)_{t \in T_0}$ to $0$ and, for $t \in T \backslash T_0$, the entries $(u_tv_t)$ and $(v_tu_t)$ are set to $h_{u_tv_t}$ (irrespectively of the value of $s$). The computation leading to \eqref{eq:bdhatR} gives
	\begin{align}\label{eq:bdhatRw0}
\max_{1 \leq i,j \leq N} \Big| \big|\hat R^{[s]}_{ij}[w_0]\big| - \delta_{ij} \Big|\prec \beta \quad 	\hbox{ and }	\quad \max_{1 \leq i,j \leq N} \Big|\Im \hat R^{[s]}_{ij}[w_0]\Big|\prec \alpha,
	\end{align}
	uniformly over all choices of $w_0$. Moreover, the resolvent identity implies
	$$
	\hat R^{[s]} [w] = 	\hat R^{[s]} [w_0] +  \hat R^{[s]} [w_0] (H^{[s]}[w_0] -H^{[s]}[w] ) \hat R^{[s]} [w]. 
	$$
In particular, since $u_t,v_t$ is different from $i,j$ for all $t \notin T_0$, we deduce from \eqref{eq:bdhatR}-\eqref{eq:bdhatRw0} that for $t \notin T_0$ and $a \in \{i,j\}$, $|\hat R^{[s]}_{a u_t}[w]|, |\hat R^{[s]}_{av_t}[w]| \prec \beta$ and similarly for $\hat R^{[s]}[w_0]$. Using \eqref{eq:momenthij}, we find  
$$
|\hat R^{[s]} [w] _{au_t} - 	\hat R^{[s]} [w_0]_{au_t} | \prec \beta^2 q^{-1} \lesssim \alpha,
$$
where the last inequality comes from $q \gtrsim N^{1/9}$. We note that the bound $|\hat R^{[s]} [w] _{au_t} -  R^{[s]}_{au_t} | \prec q^{-1}$ would have been too large for our purposes for $q \lesssim N^{1/3}$.

In \eqref{eq: trick 2}, we use for $t \in T_0$, $|\hat R^{[s_t]}_{a_{t} u_t} [w]| \prec 1$ and for $t \notin T_0$, 
$|\hat R^{[s_t]}_{a_{t} u_t} [w]| \prec |\hat R^{[s_t]}_{a_t u_t} [w_0] | + \alpha$.  We obtain
\begin{align*}
\E_T \prod_{t \in T}   \left| \hat R^{[s_{t}]}_{a_{t} b_t} \right|^2  & \prec \frac{1}{N^{2m}} \sum_{T_0 \subset T}  \sum_{w_0} \sum_{w_1}  \prod_{t \notin T_0}  \left( |\hat R^{[s_t]}_{a_{t} u_t} [w_0]|^2 + \alpha^2 \right).
\end{align*}
We next observe that for $t \notin T_0$, the matrix $\hat R^{[s_t]} [w_0] $ does not depend on $w_1 = ((u_s,v_s))_{s \notin T_0}$. We get
\begin{align*}
\E_T \prod_{t \in T}   \left| \hat R^{[s_{t}]}_{a_{t} b_t} \right|^2  
& \prec \frac{1}{N^{2m}} \sum_{T_0 \subset T}  \sum_{w_0}   \prod_{t \notin T_0} \left(  \sum_{u,v} \left(  |\hat R^{[s_t]}_{a_{t} u} [w_0]|^2  + \alpha^2 \right)\right)\\
& \prec \frac{1}{N^{2m}} \sum_{T_0 \subset T}  \sum_{w_0}   \prod_{t \notin T_0} \left(  N  \alpha / \eta +  N^2 \alpha^2 \right) \\
& =  \sum_{T_0 \subset T}  \frac{1}{N^{2|T_0|}} \sum_{w_0}   \left( 2 \alpha^2 \right)^{m - |T_0|},
\end{align*}
where we have used Ward identity \eqref{eq:ward}, $\alpha = 1/(N\eta)$ and \eqref{eq:bdhatRw0}. The number of possibilities for $w_0$ is at most $(4N)^{|T_0|}$. Hence, since $N\alpha^{2} = N^{2\delta + 1/3} \gg 1$, the above expression is maximized for $T_0 = \emptyset$ for all $N$ large enough. Therefore, we finally obtain in \eqref{eq:ytgtrt} the bound, 
\begin{equation}\label{eq:tbphatG}
\E \left[ \prod_{t \in T} y^{\rho_t}_{t} \hat G_{t}^{\rho_t}\right]  \prec    \frac{q^{2(m-r)}}{ N^m} \alpha^{2m}, 
\end{equation}
which is precisely the aimed bound in  \eqref{eq:tbpLD}.

\paragraph{Step 4: resolvent expansion.} In this final step, we  prove \eqref{eq:tbpLD}. The proof is a slightly more complicated version of the argument leading to \eqref{eq:tbphatG}. 
To do this, in view of \eqref{eq:tbphatG}, it is sufficient to compare $G_t$ and $\hat G_t$. From the resolvent identity, we have 
$$
R^{[s]} = \hat R^{[s]} + \sum_{p=1}^{n-1} \left(\hat R^{[s]}(\hat H^{[s]}- H^{[s]})\right)^{p} \hat R^{[s]} + \left(\hat R^{[s]}(\hat H^{[s]}- H^{[s]})\right)^{n} R^{[s]}.
$$
By Lemma \ref{lem: lem10 in BLZ} and \eqref{eq:bdhatR}, we have
\begin{equation*}
\alpha   \left|\left(\left(\hat R^{[s]}(\hat H^{[s]}- H^{[s]})\right)^{n} R^{[s]} \right)_{ij} \right|\prec q^{-n}.
\end{equation*}
Hence, if $n$ is large enough, this term can be made smaller than the right-hand side of \eqref{eq:tbphatG}.
Recall
$
G_{t} = N\eta \Im  ((R^{[t]}E_{t}R^{[t-1]})_{ij} + (R^{[t]}E_{t}R^{[t-1]})_{ji} ).
$
By the resolvent expansion,
\begin{multline*}
	R^{[t]}E_{t}R^{[t-1]}
	= \left( \sum_{p=0}^{n-1} \left(\hat R^{[t]}(\hat H^{[t]}- H^{[t]})\right)^{p} \hat R^{[t]} + \left(\hat R^{[t]}(\hat H^{[t]}- H^{[t]})\right)^{n} R^{[t]} \right)  \\
	\times E_{t} 
	\left( \sum_{p'=0}^{n-1} \left(\hat R^{[t-1]}(\hat H^{[t-1]}- H^{[t-1]})\right)^{p'} \hat R^{[t-1]} + \left(\hat R^{[t-1]}(\hat H^{[t-1]}- H^{[t-1]})\right)^{n} R^{[t-1]} \right).
\end{multline*}
Thus, we find that $G_t - \hat G_t$ can be written, up to  negligible terms of order smaller than $q^{-n}$, as a finite sum of terms of the form
\begin{align}\label{eq:deoidhe}
J	&  = (N\eta) h_{x_{1}y_{1}}^{\eps_{1}}h_{x_{2}y_{2}}^{\eps_{2}}\cdots h_{x_{p+p'}y_{p+p'}}^{\eps_{p+p'}} 
	\Im\left(
	\hat{R}^{[t]}_{ix_{1}}\hat{R}^{[t]}_{y_{1}x_{2}} \cdots \hat{R}^{[t]}_{y_{p}i_{t}^{\eps}}
	\hat{R}^{[t-1]}_{j_{t}^{\eps}x_{p+1}} \hat{R}^{[t-1]}_{y_{p+1}x_{p+2}} \cdots \hat{R}^{[t-1]}_{y_{p+p'}j}
	\right),
\end{align}
where $p+p' \geq 1$, $(x_{l}y_{l}) \in \{(i_{s}j_{s}),(j_{s}i_{s})\}_{s \in T}$, $h^{\eps_{s}}$ is either $h$ or $h'$ and  $(i^\eps_{t}j^\eps_{t}) \in \{(i_{t}j_{t}),(j_{t}i_{t})\}$. We call $\tau = p +p' \geq 1$ the length of the expansion. We define $T_0$ as the set of $t \in T$ such that $\{i_t,j_t\} \cap \{i,j\} \ne \emptyset$. We claim that
\begin{align}\label{eq: general form of error terms}
| J | \prec q^{-\tau} \left(\sum_{a,b,s} |\hat R^{[s]}_{a b}| +  \beta^{2}+ \delta_t\right),
\end{align}
where the sum is over $s \in \{t,t-1\}$, $a \in \{i,j,i_s,j_s, s \in T_0\}$, $b \in \{ i_t,j_t\}$ and $\delta_t \in \{0,1\}$ is the indicator that $\{i_t,j_t\}$ has a non-empty intersection with $\{i,j\} \cup \{ i_s,j_s : s \in T\backslash t \}$. Indeed the factor $q^{-\tau}$ comes from \eqref{eq:momenthij}. Next, we use \eqref{eq:bdhatR} and $|\Im(ab)| \leq |\Im(a)| |b| + |a| |\Im(b)|$. If $\delta_t = 1$, we use $|\hat R_{kl}^{[s]}|\prec 1$ and the claimed bound follows. Assume otherwise that $\delta_t = 0$. Then, in \eqref{eq:deoidhe}, by assumption we have $\{ x_l,y_l \} = \{ x_{t_l},y_{t_l}\}$ for some $t_l \in T$. If there is at least one $l$ such that $t_l \notin T_0$ then, since $\delta_t = 0$, there are at least $3$ resolvent terms in \eqref{eq:deoidhe} of the form $\hat R^{[s]}_{kl}$ with $k\ne l$. From \eqref{eq:bdhatR}, we then obtain the bound $J \prec q^{-\tau} \beta^2$. In the final case, we have $\delta_t = 0$ and for all $l$, $t_l \in T_0$ and the claimed bound \eqref{eq: general form of error terms} follows.

We deduce that 
$$
\prod_{t \in T} G_{t}^{\rho_t} - \prod_{t \in T}  \hat G_{t}^{\rho_t} = \sum_* \prod_{t \in T} \hat G_t^{\sigma_t} \prod_{l=1}^{\rho_t - \sigma_t} J_{tl} + R,
$$
where $R$ is a remainder term with $|R| \prec q^{-n}$ and the sum is a finite sum over $0 \leq \sigma_t \leq \rho_t$ and terms $J_{tl}$ as above of length $1 \leq \tau_{tl} \leq n$ such that 
$$
\tau = \sum_{t,l} \tau_{tl} \geq \sum_{t} (\rho_t - \sigma_t)  \geq 1.
$$
As in \eqref{eq: trick 1}, we observe that for $t \in T$,  
$$
|\hat G_{t} |\prec  \sum_{a,b,s} |\hat R^{[s]}_{a b}| \prec 1,
$$
with $a \in \{i,j\}$, $b \in \{i_t,j_t\}$ and $s \in \{t-1,t\}$. 

Using $2|ab| \leq |a|^2 + |b|^2$ and the conditional independence of $(y_t)$ and $(\hat G_t)$ given $(i_sj_s)_{s \in T}$, we deduce 
\begin{align*}
\E\left|\prod_{t \in T} y^{\rho_t} G_{t}^{\rho_t} - \prod_{t \in T} y^{\rho_t} \hat G_{t}^{\rho_t} \right|& \prec \sum_{*}  \frac{q^{2(m-r)}}{ N^m}  \E  \prod_{t \in T}   |\hat G_t|^{\sigma_t}  \prod_{l=1}^{\rho_t -\sigma_t}   J_{tl} + R' \\
& \prec \sum_{**}  \frac{q^{2(m-r)}}{ N^m} \E \prod_{t \in T}   \left(|\hat R^{[s_t]}_{a_t b_t}|^2 +  q^{-2\tau_t}\beta^{4}+ \delta_t\right) + R',
\end{align*}
where $R'$ is negligible and the last sum is over the finitely many possibilities for $\tau_t \geq 1$, $a_t \in \{i,j,i_s,j_s, s \in T_0\}$ and  $b_t \in \{i_t,j_t\}$.

We may now essentially repeat the argument in the previous step to argue that 
$$
 \E_T \prod_{t \in T} \left(|\hat R^{[s_t]}_{a_t b_t}|^2 +  q^{-2}\beta^{4}+ \delta_t\right) \prec \alpha^{2m},
$$
where as above, $\E_T$ is the conditional expectation with respect to $\mathcal F_T$. This will conclude the proof of \eqref{eq:tbpLD}.

With the notation of \eqref{eq: trick 2} and the computation which follows, we write 
\begin{align*} 
\E_T \prod_{t \in T}  \left(|\hat R^{[s_t]}_{a_t b_t}|^2 +  q^{-2}\beta^{4}+ \delta_t\right)  
& \prec \frac{1}{N^{2m}} \sum_{T_0 \subset T}  \sum_{w_0} \sum_{w_1}  \prod_{t \notin T_0}  \left(   |\hat R^{[s_t]}_{a_{t} u_t} [w_0]|^2 + \alpha^2 + \delta_t[w]\right),
\end{align*}
where $\delta_t[w]$ is the indicator that $\{u_t,v_t\}$ has a non-empty intersection with $\{i,j,u_s,v_s, s \in T \backslash t\}$. Note that we have used that $q^{-2}\beta^4 \leq \alpha^2$ for our choice of $q$. We further decompose $T\backslash T_0$ as: 
$$
\prod_{t \notin T_0}   \left( |\hat R^{[s_t]}_{a_{t} u_t} [w_0]|^2 + \alpha^2 + \delta_t[w]\right) = \sum_{T_1 \subset T\backslash T_0}  \prod_{t \in T_1}   \left( |\hat R^{[s_t]}_{a_{t} u_t} [w_0]|^2 + \alpha^2 \right) \prod_{t \notin T_1 \cup T_0} \delta_t [w].
$$
We observe that once all indices $(u_s,v_s)_{s \ne t}$ are chosen there at most $4m N$ choices of $(u_t,v_t)$ such that $\delta_t[w] = 1$. It follows that 
\begin{align*}
\sum_{w_1}  \prod_{t \notin T_0}  \left( |\hat R^{[s_t]}_{a_{t} u_t} [w_0]|^2 + \alpha^2 + \delta_t[w]\right) & \lesssim \sum_{T_1 \subset T\backslash T_0}  N^{m - |T_0| - |T_1|}  \prod_{t \in T_1}  \left(  \sum_{u,v} \left( |\hat R^{[s_t]}_{a_{t} u} [w_0]|^2 + \alpha^2 \right)\right)\\
& \prec  \sum_{T_1 \subset T\backslash T_0}   N^{m - |T_0| - |T_1|}  \prod_{t \in T_1} \left(  N  \alpha / \eta +  N^2 \alpha^2 \right) \\
& \lesssim \sum_{T_1 \subset T\backslash T_0}   N^{m - |T_0| + |T_1|}  \alpha^{2|T_1|},
\end{align*}
where we have used Ward identity \eqref{eq:ward}, $\alpha = 1/(N\eta)$ and \eqref{eq:bdhatRw0}. We note that $N\alpha^2 \gg 1$. Thus for all $N$ large enough, the above sum is maximized for $T_1 = T\backslash T_0$ and $ N^{m - |T_0| + |T_1|}   \alpha^{2|T_1|} = (N\alpha)^{2 m - 2|T_0|}$. Since  the number of possibilities for $w_0$ is at most $(4N)^{|T_0|}$, we deduce that 
$$
 \E_T \prod_{t \in T} \left(|\hat R^{[s_t]}_{a_t b_t}|^2 +  q^{-2}\beta^{4}+ \delta_t\right)  \prec \frac{1}{N^{2m}} \sum_{T_0 \subset T} N^{2m-|T_0|} \alpha^{2m - 2|T_0|} \prec \alpha^{2m},
$$
where we have again used that $N\alpha^2 \gg 1$. This concludes the proof of \eqref{eq:tbpLD} and the proof of Lemma \ref{prop: large deviation}. \qed 

\subsection{Proof of Lemma \ref{lem: lem14 in BLZ}}

In order to show Lemma \ref{lem: lem14 in BLZ}, we need to estimate the effect of the resampling to $\lambda_{1}$. The following proposition provide us the upper bound of the difference between $\lambda_{1}$ and $\lambda_{1}^{[k]}$.

\begin{lem}\label{lem: lem12 in BLZ}
	Assume $q\gtrsim N^{1/9}$  and $k \ll N^{5/3}$. Then, if $0 < \delta < \delta_0$ with $\delta_0$ as in Lemma \ref{lem: lem13 in BLZ}, we have 
	\begin{align*}
	     |\lambda_{1}-\lambda_{1}^{[k]}|\prec N^{-2/3-\delta}.
	\end{align*}
\end{lem}

\begin{proof}
	If $\lambda_{1}=\lambda_{1}^{[k]}$, we are done. Thus, suppose $\lambda_{1}^{[k]}<\lambda_{1}$. We set $\eta=N^{-2/3-\delta}$. According to Lemma \ref{lem: lem9 in BLZ}, we can find $1\le i\le N$ such that
	\begin{align*}
	\frac{1}{2\eta^{2}} \le N\eta^{-1}\Im R(\lambda_{1}+\mathfrak{i}\eta)_{ii}.
	\end{align*}
	Since we have $|\lambda_{1}- \mathcal{L}|\prec N^{-2/3}$, it follows from Lemma \ref{lem: lem9 in BLZ} that
	\begin{align*}
	N\eta^{-1}\Im R^{[k]}(\lambda_{1}+\mathfrak{i}\eta)_{ii}\prec \left(\min_{1\le j\le N}\left|\lambda_{1}-\lambda_{j}^{[k]}\right|\right)^{-2}.
	\end{align*}
	Since $\lambda_{1}>\lambda_{1}^{[k]}\ge\lambda_{2}^{[k]}\ge\cdots\ge\lambda_{N}^{[k]}$, we observe
	\begin{align*}
	\min_{1\le j\le N}\left|\lambda_{1}-\lambda_{j}^{[k]}\right| = \left|\lambda_{1}-\lambda_{1}^{[k]}\right|.
	\end{align*}
Moreover, we can apply Lemma \ref{lem: lem13 in BLZ}. For $c >0$ as in Lemma \ref{lem: lem13 in BLZ}, we obtain
	\begin{align*}
	N\eta^{-1}\Im R^{[k]}(\lambda_{1}+\mathfrak{i}\eta)_{ii} &\ge N\eta^{-1}\left(\Im R(\lambda_{1}+\mathfrak{i}\eta)_{ii} 
	- \left|\Im R^{[k]}(\lambda_{1}+\mathfrak{i}\eta)_{ii} -  \Im R(\lambda_{1}+\mathfrak{i}\eta)_{ii}\right|\right) \\
	&\ge  \frac{1}{2\eta^{2}} - \frac{1}{N^{c}\eta^{2}} \gtrsim \frac{1}{\eta^{2}}.
	\end{align*}
	As a result, we obtain
	\begin{align*}
	\frac{1}{\eta^{2}} \prec \left|\lambda_{1}-\lambda_{1}^{[k]}\right|^{-2}.
	\end{align*}
	In other words, $
	\left|\lambda_{1}-\lambda_{1}^{[k]}\right| \prec \eta$. 
    We have the same conclusion in the other case $\lambda_{1}^{[k]}>\lambda_{1}$ by reversing the role $H$ and $H^{[k]}$.
\end{proof}

\begin{proof}[Proof of Lemma \ref{lem: lem14 in BLZ}] We fix $0 < \delta < \delta_0$ and set $\eta = N^{-2/3-\delta}$. 
We write $\vv_{m}=(\vv_{m}(1),\ldots,\vv_{m}(N))$ and $\vv^{[k]}_{m}=(\vv^{[k]}_{m}(1),\ldots, \vv^{[k]}_{m}(N))$ for $m=2,\ldots,N$. By the spectral theorem, we have
\begin{align*}
	N\eta\Im R(z)_{ij} = \frac{N\eta^{2}v_{i}v_{j}}{(\lambda_{1}-E)^{2}+\eta^{2}} + \sum_{m=2}^N\frac{N\eta^{2}\vv_{m}(i)\vv_{m}(j)}{(\lambda_{m}-E)^{2}+\eta^{2}}.
\end{align*}
Let $\eps >0$ and let $N':=\lfloor N^{2\eps} \rfloor$. In the proof of Lemma \ref{lem: lem9 in BLZ}, we have checked that with overwhelming probability: for all $E$ satisfying $|E- \mathcal{L}|\le N^{-2/3+\eps}$, we have, for some $C >0$,
\begin{align}\label{eq:djeijd}
	\left|\sum_{m=N'+1}^{N}\frac{N\vv_{m}(i)\vv_{m}(j)}{(\lambda_{m}-E)^{2}+\eta^{2}}\right| \leq C N^{\eps} (N')^{-1/3}N^{4/3}.
\end{align}
 By Lemma \ref{lem:TW} and Lemma \ref{lem: delocalization}, we can find $c_0>0$ such that 
\begin{align*}
	\prob\left(\cE \right)\ge 1-\eps/2,
\end{align*}
where $\cE$ is the event that \eqref{eq:djeijd} holds, $\{\lambda_{1}-\lambda_{2} > c_0 N^{-2/3}\}$ and  $\max_m \|\vv_m\|_\infty^2  \leq N^{\eps-1}$. On the event $\cE$, we find for all $E$ with $|\lambda_{1}-E|\le (c/2)N^{-2/3}$ that for some $C >0$,
\begin{align*}
	\left|\sum_{m=2}^{N'}\frac{Nv_{m,i}v_{m,j}}{(\lambda_{m}-E)^{2}+\eta^{2}}\right| 
	\leq C N^{\eps} N'N^{4/3}.
\end{align*}

We fix $\delta' >0$ such that $\delta + \delta' < \delta_0$. On the event $\cE$, for any $E$ such that $|\lambda_{1}-E|\le \eta N^{-\delta'}$, we have 
\begin{align*}
	\left| \frac{N\eta^{2}v_{i}v_{j}}{(\lambda_{1}-E)^{2}+\eta^{2}} - Nv_{i}v_{j}\right| \leq N^{\eps} \left| \frac{\eta^{2}}{(\lambda_{1}-E)^{2}+\eta^{2}} - 1 \right| \leq N^{\eps -2\delta'}.
\end{align*}

Recall $\eta = N^{-2/3 - \delta}$. Combining all of the above estimates and choosing $0 < \eps \leq  \min(\delta', \delta/3)$, we conclude that for all $E$ satisfying $|\lambda_{1}-E|\le \eta N^{-\delta'}$, for some $C>0$,
\begin{align*}
	\max_{1\le i\le j\le N} |N\eta\Im R(E+\mathfrak{i}\eta)_{ij}-Nv_{1,i}v_{1,j}|\leq C N^{-\min(\delta,\delta')},
\end{align*}
on the event $\cE$. Now we repeat the above  argument for $R^{[k]}$. We define the event $\cE^{[k]}$ similarly for $H^{[k]}$. It provides us an event $\cE^{[k]}$ of probability at least $1 - \eps /2$ such that for all $E$ satisfying $|\lambda_{1}^{[k]}-E|\le \eta N^{-\delta'}$,
\begin{align*}
	\max_{1\le i\le j\le N} |N\eta\Im R^{[k]}(E+\mathfrak{i}\eta)_{ij}-Nv^{[k]}_{1,i}v^{[k]}_{1,j}|\leq C N^{-\min(\delta,\delta')}.
\end{align*}
According to Lemma \ref{lem: lem12 in BLZ}, we have $ |\lambda_{1}-\lambda_{1}^{[k]}|\le \eta N^{-\delta} = N^{-2/3-\delta-\delta'}$  with overwhelming probability (since $\delta + \delta' < \delta_0$). Since $\PP (\cE \cap \cE^{[k]}) \geq 1 - \eps$ and $\eps$ can be made arbitrarily small, this concludes the proof of the lemma by picking any $0 < c' < \min(\delta,\delta')$.\end{proof}

\section{Resolvent of sparse random matrices}\label{sec:resolvent}

In this section, we have gathered the proofs of some estimates on the resolvent of $H$ which have been used. 

\subsection{Cauchy-Stieltjes transform near the edge}

Recall that for $z \in \C_+$, we have set $R(z) = (H-z)^{-1}$ and 
$$
m(z) = \frac 1 N \mathrm{Tr} R(z).
$$
The following local law improves on Lemma \ref{lem: local law} when $\kappa$ and $\eta$ are both small. We fix $\eps_0 >0$, for example $\eps_0 = 1/4$ is sufficient for our purposes. We define the spectral domains:
\begin{align*}
	\mathcal{D}_{0}:=\left\{ w =\kappa+\mathfrak{i}\eta\in\C_{+} : |\kappa|\le 3,   N^{\eps_0 -1} \le \eta\le 1 \right\}.
\end{align*}
and we let $\bar \cD_0 = \{ w = \kappa + \mathfrak{i}\eta : |\kappa| \leq 3, \eta \geq N^{\eps_0 - 1}\}$ be the infinite half-strip containing $\cD_0$. 
\begin{lem}\label{lem:locallaw2}
Assume $q \gg 1$.  Let $ m_{\star}$ be as in Lemma \ref{lem: local law}. Uniformly on $w = \kappa + \mathfrak{i}\eta\in\mathcal{D}_0$, we have, with $z  = \mathcal{L}+w$, 
		\begin{align*}
			|m(z)-m_\star(z)| \prec \frac{1}{N\eta} + \frac 1 {q^3} + (\kappa + \eta)^{1/4} \left(\frac{1}{N\eta} + \frac 1 {q^3} \right)^{1/2}˚.
		\end{align*}
\end{lem}

Before proving Lemma \ref{lem:locallaw2}, we first prove a bound between $m_\star$ and $m_{\text{sc}}$, the Cauchy-Stieltjes transform of the semi-circular law.

\begin{lem}\label{lem:locallaw02}
Assume $q \gg 1$. Let $\eps >0$ and set $\bar q = \min(q,N^{1/2-\eps})$.  There exists $C >0$ such that with overwhelming probability:
		\begin{align*}
			\sup_{z \in \C_+} |m_{\text{sc}} (z)-m_\star(z)| \leq \frac C { \bar q}.
		\end{align*}
\end{lem}

\begin{proof}By \cite[Proposition 2.6]{HLY20}, there exists a deterministic even polynomial $$Q(y) = \frac{a_2}{q^2} y^4 + \frac{a_3}{q^4} y^6 + \cdots$$ whose coefficients $a_i$ depend on the moments of $h_{ij}$ and are uniformly bounded such that the random multivariate polynomial
\begin{equation}\label{eq:defPm}
P(z,y) := 1 + z y + y^2 + Q(y) + \cX y^2
\end{equation}
satisfies 
$$
P(z,m_\star(z)) = 0.
$$

We set $P_0(z,y) = 1 + zy + y^2$. We have $P_0(z,m_{\text{sc}}(z)) = 0$. We set $f(z) = P_0(z,m_\star(z)) $ and $g(z) = m_\star(z) - m_{\text{sc}}(z)$.

We have $ |\cX|\prec 1/(q\sqrt N)$. Hence the event $\cE = \{ |\cX| \leq 1/\bar q^2\}$ has overwhelming probability. On the event $\cE$, uniformly in $z \in \C_+$, we have $|m_\star(z)| \leq C$. Since, 
$f(z) = -Q(m_\star(z)) - m_\star(z)^2 \cX$, we deduce that  if $\cE$ holds, for all $z \in \C_+$,
\begin{equation}\label{eq:fbdunif}
|f(z) | \leq \frac{C}{\bar q^2}.
\end{equation}
By Taylor expansion, we have
$$
f(z) = g(z) (z + 2 m_{\text{sc}}(z)) + g(z)^2.  
$$
With $\sqrt{\cdot}$ is the principal branch of the square root function, we have $z + 2m_{\text{sc}}(z) =\sqrt{z^2 - 4}$. Hence 
$$
2 g(z) = - \sqrt{z^2 - 4} \pm \sqrt{ z^2 - 4 + 4 f(z)}.
$$
Since $g(z)$ is the difference of two Cauchy-Stieltjes transforms of probability measures, as $\Im(z)$ goes to infinity, $|g(z)|$ must vanish.  From \eqref{eq:fbdunif}, this forces the choice of the above $\pm$-sign to be $+$ for all large $z$ and thus for all $z \in \C_+$ since $g(z)$ is analytic on $\C_+$.

The remainder of the proof is obvious by decomposing in two possibilities: if $|f(z)| \geq |z^2 - 4 | $ then $|g(z)| \leq \sqrt{|f(z)|} +   \sqrt{ 5|f(z)|}$. If $ |f(z)| \leq |z^2 - 4 |$, then, by Taylor expansion,
$$
|2 g(z)| = |- \sqrt{z^2 - 4}  + \sqrt{ z^2 - 4 + 4 f(z)}| = \left|- \sqrt{z^2 - 4} \left( - 1 +  1 + O \left( \frac{| f(z)|}{|z^2 - 4|} \right)\right)\right| \lesssim\frac{| f(z)|}{\sqrt{|z^2 - 4|}}.
$$
It concludes the proof since $1/ \sqrt{|z^2 - 4|} \leq 1/\sqrt{|f(z)|}$.
\end{proof}

By \cite[Theorem 2.8]{EKYY13}, Lemma \ref{lem:locallaw02} implies the following weak local law.
\begin{cor}\label{cor:locallaw02}
Assume $q \gg 1$.  For any $\eps >0$, with overwhelming probability, 
		\begin{align*}
			\sup_{w \in \bar \cD_0} \max_{1 \leq i , j \leq N} | R_{ij} (z) - \delta_{ij} m_\star(z ) | \leq N^{\eps}\left( \frac 1 q + \frac 1 {\sqrt{N\eta}}\right),
		\end{align*}
		where $z = w + \cL$ and $w = E + \mathfrak{i}\eta$.
\end{cor}
\begin{proof}
From \cite[Theorem 2.8]{EKYY13}, the result holds with $m_{\text{sc}}(z)$ in place of $m_\star(z)$ (\cite[Theorem 2.8]{EKYY13} is stated in $\cD_0$ but the case $\eta \geq 1$ extends obviously). It remains to use Lemma \ref{lem:locallaw02} to bound the difference $m_{\text{sc}}(z) - m_\star(z)$. 
\end{proof}

All ingredients are gathered to prove Lemma \ref{lem:locallaw2}.
\begin{proof}[Proof of Lemma \ref{lem:locallaw2}] We fix $w \in \cD_0$ and let $z = \cL+w$. We set 
$$g(z) = m(z) - {m}_{\star}(z) \quad \hbox{ and } \quad \Lambda (z) = | g(z) |.$$ 
Let $P(z,y)$ be as in \eqref{eq:defPm}. 
Applying Taylor expansion, we have, from \eqref{eq:defPm} and $ P(z, m_{\star}(z)) = 0$,
\begin{align*}
	P(z, m(z)) = \partial_{2}P(z, m_{\star}(z))(g(z)) + \frac{1}{2} \partial^2_{2}P(z, m_{\star}(z)) g(z) ^{2} + R(g(z)) ,
\end{align*}
where $R(y) = b_1 y^3 + b_2 y^4 + \cdots $ is a deterministic polynomial whose coefficients are less than $C /q^2$. We set 
$$
f(z) = 	P(z, m(z)) -  R(g(z)), \quad b (z) = \partial_{2}P(z, m_{\star}(z)), \quad a(z) = \frac{1}{2}\partial^2_{2}P(z, m_{\star}(z))
$$
We get 
$$
a(z) g(z) = - b  (z) \pm \sqrt{ b(z)^2 + 4 f(z)a(z) }.
$$
By \cite[Proposition 2.6]{HLY20}, with overwhelming probability,  the following event holds: for some $C >0$, for all $z \in \C_+$, $|a(z) - 1| \leq C q^{-2}$ and $ |b(z)| \geq \sqrt{||\kappa|+\eta|}/C$. Moreover, by  Corollary \ref{cor:locallaw02}, for some $C >0$, with overwhelming probability, the following event: for all $z \in \cD_0$, $\Lambda(z) \leq 1/\log N$ and $|R(g(z))| \leq C \Lambda(z)^3/q^2 \leq \Lambda(z)^2 / q^2$ (for $N$ large enough). On the intersection of these two last events, say $\cE$, since $|f(z)a(z)|$ is bounded uniformly on $\bar \cD_0$ and $g(z)$ is analytic and vanishes as $\Im(z)$ goes to infinity, the only possibility for the $\pm$-sign is $+$. Arguing as in the proof of Lemma \ref{lem:locallaw02}, we deduce that, if $\cE$ holds, for some new $C>0$,
$$
\Lambda(z) = |g(z)| \leq C \sqrt{|f(z)|} .
$$
Since $|f(z)| \leq |P(z,m(z))| + \Lambda(z)^2 /q^{2}$, So finally, since $\Lambda(z) \leq 1/ \log N$ on $\cE$, if $N$ is large enough we get for some new $C >0$, 
\begin{equation}\label{eq:boundLbyP}
\Lambda(z)  \leq C \sqrt{|P(z,m(z))|}.
\end{equation}

The other way around, we now estimate $|P(z,m(z)|$ in terms of $\Lambda(z)$.  By \cite[Proposition 2.9]{HLY20}, we have
\begin{multline*}
	\E \left[ |P(z,m(z))|^{2r} \right] \\
	\prec \max_{1 \le s_{1}+s_{2}\le 2r}
	\E \left[ \left\{ |\partial_{2}P(z,m(z))|\left( \frac{1}{q^{3}} + \frac{1}{N\eta} \right) \frac{\Im (m(z))}{N\eta}  \right\}^{s_{1}/2} \left(\frac{\Im (m(z))}{N\eta}\right)^{s_{2}} |P(z,m(z))|^{2r-s_{1}-s_{2}} \right],
\end{multline*}
(in \cite{HLY20}, Proposition 2.9 is stated for $w \in \cD(\eps)$ but their proof holds in the larger domain $\cD_0$). From \eqref{eq:mstarasym}, it follows that
\begin{align*}
	\Im (m_{\star}(z)) \lesssim \sqrt{|\kappa|+\eta},
\end{align*}
which gives us 
\begin{align*}
	\Im (m(z)) \lesssim \sqrt{|\kappa|+\eta} + \Lambda.
\end{align*}
Also, from \cite[Proposition 2.6]{HLY20},
\begin{align*}
	|\partial_{2}P(z,m(z))| = |\partial_{2}P(z,m_{\star}(z))| 
	+ O(\Lambda) \lesssim  \sqrt{|\kappa|+\eta}+\Lambda .
\end{align*}
By Young's inequality, we obtain, for any $\eps >0$,
\begin{multline*}
	\E \left[ \left\{ |\partial_{2}P(z,m(z))|\left( \frac{1}{q^{3}} + \frac{1}{N\eta} \right) \frac{\Im (m(z))}{N\eta}  \right\}^{s_{1}/2} \left(\frac{\Im( m(z))}{N\eta}\right)^{s_{2}} |P(z,m(z))|^{2r-s_{1}-s_{2}} \right] \\
	\lesssim
	N^{\epsilon}\E\left[|\partial_{2}P(z,m(z))|^{r}\left( \frac{1}{q^{3}} + \frac{1}{N\eta} \right)^{r} \left(\frac{1}{N\eta}\right)^{r}\left\{(|\kappa|+\eta)^{r/2} + \Lambda^{r}\right\}\right]\\
	+ N^{\epsilon}\left(\frac{1}{N\eta}\right)^{2r}\left\{(|\kappa|+\eta)^{r} + \E\Lambda^{2r}\right\} + N^{-\epsilon/(2r -1)}\E|P(z,m(z))|^{2r}.
\end{multline*}
Thus, 
\begin{multline*}
	\E \left[ |P(z,m(z))|^{2r} \right] \\
	\prec 
	\left( \frac{1}{q^{3}} + \frac{1}{N\eta} \right)^{r} \left(\frac{1}{N\eta}\right)^{r}\left\{(|\kappa|+\eta)^{r} + \E\Lambda^{2r}\right\}
	+ \left(\frac{1}{N\eta}\right)^{2r}\left\{(|\kappa|+\eta)^{r} + \E\Lambda^{2r}\right\}.
\end{multline*}
If $\gamma = 1/q^3 + 1/ (N\eta)$, we find
\begin{align*}
	\E \left[ |P(z,m(z))|^{2r} \right]
	\prec   
	\left( \gamma \sqrt{|\kappa|+\eta} \right)^{2r}
	+ \gamma^{2r} \E\Lambda^{2r}.
\end{align*}

From \eqref{eq:boundLbyP}, we deduce that 
\begin{align*}
(\E\Lambda^{2r})^2 \leq 	\E \Lambda^{4r} & \lesssim  \E \left[ |P(z,m(z))|^{2r} \right] 	\prec 	\left( \gamma \sqrt{|\kappa|+\eta} \right)^{2r}
	+ \gamma^{2r} \E\Lambda^{2r}.
\end{align*}
Since $x^2 \leq a + b x$, $a,b,x\geq 0$ implies that $x \le \sqrt{2a} + b$, we have established
that 
\begin{align*}
	\E [\Lambda^{2r} | \prec \left( \gamma + \left(\gamma \sqrt{|\kappa|+\eta}  \right)^{1/2} \right)^{2r} .
\end{align*}
Since $r$ is arbitrary, from Markov inequality, the proof is complete. \end{proof}

\subsection{Proof of Lemma \ref{lem: lem10 in BLZ}}\label{sec: estimate on the imaginary part of local law}

It is enough to only show the second equality on $\Im(R_{ij})$, which is a consequence of the following lemma.
\begin{lem}\label{prop: estimate on the imaginary part of local law}
Assume $q \gtrsim N^{1/9}$ and let $0 \leq \delta < 1/3$. We have
	\begin{align*}
\sup_z \max_{1\leq i , j \leq N} |\Im R_{ij}(z)-\delta_{ij}\Im (m_\star(z))| \prec \frac{1}{N\eta},
	\end{align*}
	where the supremum is over all $z=E+\mathfrak{i}\eta$ with $|E-\cL|\le 2 N^{-2/3+\delta}$ and $\eta =  N^{-2/3-\delta}$.
\end{lem}

Under the assumptions of the above lemma,
we get for $i\neq j$,
\begin{align*}
\textnormal{Im}R_{ij}(z) \prec \frac{1}{N\eta},
\end{align*}
and, since by Equation \eqref{eq:mstarasym}, $\Im (m_\star(z)) \lesssim N^{-1/3 + \delta /2}$,
\begin{align*}
\Im R(z)_{ii}\le |\Im R(z)_{ii}-\Im (m_\star(z))| + \Im (m_\star(z)) \prec \frac{1}{N\eta}.
\end{align*}
The second statement of Lemma \ref{lem: lem10 in BLZ} follows.  The remainder of this subsection is dedicated to the proof of Lemma \ref{prop: estimate on the imaginary part of local law}. It relies on an iterative self-improving error bound on resolvent estimates.  The proof is an adaptation of \cite[Section 3]{EKYY13}, there are however some new difficulties coming from the randomness of $\cL$.

\paragraph{Step 1: net argument.} Arguing as in Step 1 of the proof of Lemma \ref{lem: lem13 in BLZ},  it is sufficient to prove that for any deterministic real $\kappa$ with $|\kappa| \leq N^{-2/3 + \delta}$, 
\begin{equation}\label{eq:tbpimll}
 \max_{1\leq i , j \leq N} |\Im R_{ij}(\tz)-\delta_{ij}\Im (m_\star(\tz))| \prec \frac{1}{N\eta},
 \end{equation}
with $\tz$ defined by \eqref{eq:deftz}.
In the remainder of the proof, we fix such  $\kappa$ and corresponding random $\tz$. 

\paragraph{Step 2: inductive events.}For $\eta = N^{-2/3 - \delta}$, we set $\cD_1 = \{ z = E + \mathfrak{i}\eta \in \C_+ : |E - \cL| \leq 2N^{-2/3+ \delta} \}$. We introduce the following variables
\begin{align*}
	&\Lambda_{e} :=\sup_{z \in \cD_1} \max_{i,  j}|R_{ij}(z) - \delta_{ij} m_\star(z)|,  \quad 	 \Lambda := \sup_{z \in \cD_1}|m(z)-m_\star(z)|. \\ 	
	& \Lambda_{o}^{\Im } :=\max_{i\neq j} \Im R_{ij}(\tz), \quad \Lambda_{d}^{\Im }:=\max_{i}|\Im R_{ii}(\tz)-\Im (m_\star(\tz))|. 
\end{align*}
Note that $\Lambda^{\Im }_o$ and $\Lambda^{\Im }_d$ depend implicitly on $\kappa$ (which is fixed).
For $\alpha >0$ such that $(N\eta)^{-1} \leq \alpha \leq 1/q$, we introduce the events
\begin{align*}
\Omega := \left\{\Lambda  \le N^{\eps}(N\eta)^{-1} \; ; \; \Lambda_{e} \le N^{\eps}q^{-1} \right\} \quad \hbox{and} \quad 	\Omega( \alpha) := \Omega \cap \left\{ \Lambda_{o}^{\Im }+\Lambda_{d}^{\Im }\le N^{\eps}\alpha \right\},
\end{align*}
where $\eps>0$ is an arbitrarily small constant to be chosen later. We note that by Lemma \ref{lem:locallaw2} and Corollary \ref{cor:locallaw02}, the event $\Omega(1/q)$ has overwhelming probability. By an inductive argument, we will prove  that $\Omega(N^\eps/(N\eta))$ has overwhelming probability (if $N\eta < q$, that is $q > N^{1/3-\delta}$, there is nothing more to prove and the proof of the lemma is complete). Then, if $\Omega(N^\eps/(N\eta))$ has overwhelming probability for all fixed $0 < \eps< 1/9$ then \eqref{eq:tbpimll} holds and the proof of the lemma is complete.

\paragraph{Step 3: resolvent of minors.}For ease of notation, in the sequel, we often omit $\tz$ and write $m$, $m_\star$ and $R$ in place of $m(\tz)$, $m_\star(\tz)$ and $R(\tz)$. We have $|m_\star| \asymp 1$, hence on $\Omega(\alpha)$, we find $|R_{ii}|\asymp 1$. Similarly, from Equation \eqref{eq:mstarasym}, on $\Omega(\alpha)$, 
\begin{align*}
	\Im R_{ii} \le \Im (m_\star) + |\Im R_{ii} - \Im (m_\star)| \lesssim N^{-1/3 + \delta/2} + N^{\eps} \alpha \lesssim N^\eps \alpha,
\end{align*}
where have used that $(N\eta)^{-1} = N^{-1/3 +\delta} \lesssim \alpha $. In summary, on $\Omega(\alpha)$, for all $i \ne j$, 
\begin{align}\label{eq:OmR}
|R_{ii}| \asymp  1 , \quad |R_{ij}| \leq \frac{N^\eps}{q} \quad \hbox{and}  \quad	\Im R_{ii} + \Im R_{ij}  \lesssim  N^\eps \alpha.
\end{align}
To be precise, the underlying constants in $\asymp$ and $\lesssim$ in the above expressions depend only on the measure $\rho_\star$ through Equation \eqref{eq:mstarasym}.  We will use this convention in the rest of the proof.

For $\T\subset\{1,\cdots,N\}$, let $H^{(\T)}$ be the $(N-|\T|)\times(N-|\T|)$ minor of $H$ obtained by removing all rows and columns of $H$ indexed by $i\in\T$. 
In addition, we set $R^{(\T)}(z)=(H^{(\T)}-zI)^{-1}$. Our first goal is extend the bounds in \eqref{eq:OmR} to $R^{(\T)} = R^{(\T)}(\tz)$ when $\Omega(\alpha)$ holds uniformly over sets $\T$ with $|\T| \leq 2$.

For $i,j  \ne k$, we have the following identity
\begin{align*}
	R_{ij}^{(k)} = R_{ij} -\frac{R_{ik}R_{kj}}{R_{kk}},
\end{align*}
(see e.g. \cite[Lemma 3.5]{MR3792624}). Thus, since $|\Im(ab)| \leq |\Im(a) b| + | \Im(b) a|$, we get 
$$\left|\Im\left(\frac{ab}{c} \right) \right|\leq \frac{| \Im(a) bc | + | \Im(b)ac | + | \Im(c)ab | + | \Im(a)\Im(b)\Im(c)| }{|c|^2}$$ and 
\begin{align*}
|	\Im R_{ij}^{(k)} - \Im R_{ij} | \leq  \frac{ \Im (R_{ik})|R_{kj} R_{kk}| +  \Im (R_{kj})|R_{ik} R_{kk}| + \Im (R_{kk})|R_{ik} R_{kj}| + \Im (R_{ik})\Im (R_{kj})\Im (R_{kk}) }{|R_{kk}|^{2}}.
\end{align*}
On $\Omega(\alpha)$, from \eqref{eq:OmR}, we obtain 
\begin{align*}
	\left| \Im R_{ij}^{(k)} - \Im R_{ij} \right| \lesssim \frac{N^{2\eps}  \alpha}{q} \ll N^\eps \alpha.
\end{align*}
Similarly, for all $z \in \cD_1$,
$$
\left| R_{ij}^{(k)} (z)- R_{ij}(z) \right| \le \frac{|R_{ik} (z) R_{k j}(z) |}{|R_{kk}(z) |} \lesssim \frac{N^{2\eps}}{q^2} \ll \frac{N^\eps}{q}. 
$$
We may repeat the above computation for $\T = \{k,l\}$ and $l \ne k$.  
It follows that, if $\Omega(\alpha)$ holds, for all $\T$ with $|\T| \leq 2$, for all $i,j \notin \T$ with $ i \ne j$,
\begin{align}\label{eq:OmRT}
\sup_{z \in \cD_1} |R^{(\T)}_{ii}(z)| \asymp  1 , \quad \sup_{z \in \cD_1} |R^{(\T)}_{ij}(z)| \lesssim \frac{N^\eps}{q}   \quad	\hbox{and} \quad \Im R^{(\T)}_{ii} + \Im R^{(\T)}_{ij} \lesssim  N^\eps \alpha.
\end{align}

\paragraph{Step 4: concentration inequality.} Next, if $\T \subset \{1,\ldots,N\}$, we use the notation
$$
\sum_{i}^{(\T)} = \sum_{i : i \notin (\T)}.
$$
Using classical resolvent identities, the following variables are used in the next step to control $R_{ii}(z)$ and $R_{ij}(z)$: 
\begin{align*}
	Z_{ij}(z)&:=\sum_{k,l}^{(ij)}h_{ik}R_{kl}^{(ij)}(z)h_{lj}, \\
	Z_{i}(z)&:= \sum_{k,l}^{(i)}\left( h_{ik}h_{li} - \frac{1}{N}\delta_{kl} \right) R_{kl}^{(i)}(z).
\end{align*}

For a fixed $z \in \C_+$, we note that $R^{(i)}(z)$ is independent of the vector $(h_{ik})_{k}$ and similarly for $R^{(ij)}(z)$ with $(h_{ik},h_{jl})_{k,l}$. We are however interested in $R^{(i)}= R^{(i)}(\tz)$ and $R^{(ij)}= R^{(ij)}(\tz)$ with $\tz$ defined by \eqref{eq:deftz}, this breaks the above independence property. To circumvent this difficulty, we define $\tz_i = \kappa + L + \cX_i + \mathfrak{i}\eta$ and $\tz_{ij} = \kappa + L + \cX_{ij} + \mathfrak{i}\eta$ with
$$
\cX_i = \frac 1 N \sum_{k,l}^{(i)} \left( h_{kl}^2 - \frac 1 N \right) \quad \hbox{ and } \quad \cX_{ij} = \frac 1 N \sum_{k,l}^{(ij)} \left( h_{kl}^2 - \frac 1 N \right).
$$
The independence of $R^{(i)}(\tz_{i})$ and $(h_{il})_{l \ne i}$  is now restored, and similarly for $R^{(ij)}(\tz_{ij})$. 
The next lemma relies on the concentration of the variables $Z_{ij}$ and $Z_i$.

\begin{lem}\label{lem: bound on off diagonal}
Assume $q\gg 1$ and $1/(N\eta) \le \alpha \le 1/q $. We have on $\Omega(\alpha)$,
	\begin{align*}
	& |Z_i| \prec N^{\eps}\left(\frac{1}{q}  + \sqrt{\frac{\alpha}{N\eta}} \right), \quad |Z_{ij} | \prec N^\eps \left( \frac{1}{q^{2}} + \sqrt{\frac{\alpha}{N\eta}}\right)\\
&	|\Im(Z_{i})| \prec N^\eps \left( \frac{\alpha}{q} + \sqrt{\frac{\alpha}{N\eta}}\right) ,
\quad |\Im(Z_{ij}) | \prec N^\eps \left( \frac{\alpha}{q} + \sqrt{\frac{\alpha}{N\eta}}\right).
	\end{align*}
\end{lem}
	\begin{proof}
We start by controlling $\tilde R^{(i)} = R^{(i)}(\tz_i)$ and $\tilde R^{(ij)} =  R^{(ij)}(\tz_{ij})$. From the resolvent identity, we have 
$$
R^{(i)} - \tilde R^{(i)}  = - (\tz - \tz_i) R^{(i)}  \tilde R^{(i)} .
$$
Moreover,  
$$
 |\tz - \tz_i| \prec \frac{1}{Nq},
$$
and on $\Omega(\alpha)$, from \eqref{eq:OmRT}, for any $k,l$
$$
\left| \left(R^{(i)} \tilde R^{(i)}\right)_{kl} \right| \leq \sum_{a}^{(i)} | R^{(i)}_{ka}    \tilde R^{(i)}_{al} | \leq \sum_{a}^{(i)} \left( (R^{(i)}_{ka})^{2} + (\tilde R^{(i)}_{al})^{2} \right)| \lesssim \frac{N^{\eps}\alpha}{\eta}.
$$
So finally, 
\begin{equation}\label{eq:diffRzzi}
\left| R_{kl}^{(i)} - \tilde R_{kl}^{(i)} \right| \prec \frac{N^{\eps}\alpha}{qN\eta}
\end{equation}
The same bound holds for $\left| R^{(ij)}_{kl}  - \tilde R^{(ij)}_{kl} \right|$.

Now, using the independence of $\tilde R^{(i)}$ and $(h_{il})_{l}$, the large deviation estimate \cite[Lemma 3.8 (ii)]{EKYY13} and \eqref{eq:OmRT}-\eqref{eq:diffRzzi}, we obtain on $\Omega(\alpha)$,
		\begin{align*}
			|Z_{i}|& \prec \left|\sum_{k}^{(i)}\left( |h_{ik}|^{2} - \frac{1}{N} \right)\tilde R_{kk}^{(i)}\right| + \left|\sum_{k\neq l}^{(i)}h_{ik}\tilde R_{kl}^{(i)} h_{li}\right| + \frac{N^{\eps}\alpha}{qN\eta}\\
			& \prec  \frac{\max_{k}|\tilde R_{kk}^{(i)}|}{q} + \frac{\max_{k\neq l}|\tilde R_{kl}^{(i)}|}{q} + \left(\frac{1}{N^{2}}\sum_{k,l}^{(ij)}|\tilde R_{kl}^{(i)}|^{2}\right)^{1/2} + \frac{N^{\eps}\alpha}{qN\eta} \\
			& \prec N^{\eps} \left(  \frac{1}{q} + \frac{1}{q^{2}} + \sqrt{\frac{\alpha}{N\eta}}\right),
		\end{align*}
where we have used Ward identity \eqref{eq:ward}.
The first claim follows.

Similarly, since for $i\neq j$, the random variables $\{h_{ik}\}_{k : k \ne j}$ are independent of $\{h_{lj}\}_{l: l \ne i}$, from \cite[Lemma 3.8 (iii)]{EKYY13} and \eqref{eq:OmRT}-\eqref{eq:diffRzzi},  on $\Omega(\alpha)$, we have
		\begin{align*}
			\left|\sum_{k,l}^{(ij)}Z_{ij} \right| &\prec  \frac{\max_{k}|\tilde R_{kk}^{(ij)}|}{q^{2}} + \frac{\max_{k\neq l}|\tilde R_{kl}^{(ij)}|}{q} + \left(\frac{1}{N^{2}}\sum_{k,l}^{(ij)}|\tilde R_{kl}^{(ij)}|^{2}\right)^{1/2} + \frac{N^{\eps}\alpha}{qN\eta} \nn\\
			&\prec  N^{\eps} \left( \frac{1}{q^{2}} + \sqrt{\frac{\alpha}{N\eta}}\right).
		\end{align*}
	
	The same argument gives, on $\Omega(\alpha)$, 
		\begin{align*}
			\left|\Im(Z_{ij}) \right| =  \left| \sum_{k,l}^{(ij)}h_{ik}\Im (\tilde R_{kl}^{(ij)})h_{lj}\right| \prec N^\eps \left(\frac{\alpha}{q} + \sqrt{\frac{\alpha}{N\eta}}\right).
		\end{align*}

		Finally, we obtain similarly, on $\Omega(\alpha)$
		\begin{align*}
			\Im Z_{i} & = \sum_{k}^{(i)}\left( |h_{ik}|^{2} - \frac{1}{N} \right) \Im \tilde R_{kk}^{(i)} + \sum_{k\neq l}^{(i)}h_{ik}\Im (\tilde R_{kl}^{(i)})h_{li}			\prec N^\eps \left( \frac{\alpha}{q} + \sqrt{\frac{\alpha}{N\eta}}\right).
		\end{align*}
		The proof is complete.
	\end{proof}

\paragraph{Step 5: Iteration of the error bounds.}

Our next lemma improves the bound for the off-diagonal entries of $\Im(R)$ when $\Omega(\alpha)$ holds.
\begin{lem}\label{lem: bound off digonal}Assume $q\gg N^{\eps}$ and $1/(N\eta)\le \alpha \leq 1/q$. We have on $\Omega(\alpha)$,
	\begin{align*}
		\Lambda_{o}^{\Im} \prec N^\eps \left( \frac{\alpha}{q} + \sqrt{\frac{\alpha}{N\eta}}\right).
	\end{align*}
\end{lem}
\begin{proof}
		Let $i\neq j$. Using 
		\begin{align}
			R_{ij} = -R_{ii}R_{jj}^{(i)}(h_{ij}-Z_{ij}),
		\end{align}
(see e.g. \cite[Lemma 3.5]{MR3792624}),
		it follows from \eqref{eq:OmRT} and Lemma \ref{lem: bound on off diagonal} that on $\Omega(\alpha)$,
		\begin{align*}
			\Im R_{ij} &\le  \Im R_{ii} |R^{(i)}_{jj}||h_{ij}-Z_{ij}| +  |R_{ii}| \Im R^{(i)}_{jj}|h_{ij}-Z_{ij}| + |R_{ii} | |R^{(i)}_{jj}|\Im Z_{ij}\\
			&\prec N^{2\eps} \alpha \left(\frac{1}{q}+\sqrt{\frac{\alpha}{N\eta}}\right) + N^\eps \left( \frac{\alpha}{q} + \sqrt{\frac{\alpha}{N\eta}}\right).
		\end{align*}
	Since $\alpha \leq 1/q \ll N^{-\eps}$,	the second term is dominant. Thus, by taking the maximum over $i\neq j$, the statement of the lemma follows.
\end{proof}

It remains to control the diagonal entries of $\Im(R)$ when $\Omega(\alpha)$ holds.
\begin{lem}\label{lem: bound on digonal}
Assume $q\gtrsim N^{\eps}$ and $1/(N\eta)\le \alpha \leq 1/q$. We have on $\Omega(\alpha)$,
	\begin{align*}
		\Lambda_{d}^{\Im} \prec N^\eps \left( \frac{\alpha}{q} + \sqrt{\frac{\alpha}{N\eta}} + \frac{1}{N\eta}\right).
	\end{align*}
	\end{lem}
	\begin{proof}
We will prove that if $\Omega(\alpha)$ holds then	\begin{align}\label{le:tbpfle}
		\max_{i}\big|\text{\normalfont Im}(R_{ii}) - \Im (m)\big| \prec  N^\eps \left(\frac{\alpha}{q}+\sqrt{\frac{\alpha}{N\eta}}\right).
	\end{align}
	Since, by assumption, on $\Omega(\alpha)$, $\Lambda = |m - m_\star| \leq N^{\eps} /(N\eta)$, it will conclude the proof.

From \cite[Lemma 3.10]{EKYY13}, the following identity holds: for all $z \in \C_+$,
		\begin{align*}
			- \frac{1}{R_{ii}(z)} =z + m(z) - \Upsilon_{i}(z) ,
		\end{align*}
 where 
		\begin{align*}
			\Upsilon_{i} (z)= h_{ii}-Z_{i} (z) +\mathcal{A}_{i}(z) \quad \hbox{and} \quad			\mathcal{A}_{i} (z) = \frac{1}{N}\sum_{j}\frac{R_{ij}(z)R_{ji}(z)}{R_{ii}(z)}.
		\end{align*}
				
In particular, $R_{ii}-R_{jj} = R_{ii}R_{jj}(\Upsilon_{j}  -\Upsilon_{i})$ and consequently, from \eqref{eq:OmRT}, on $\Omega(\alpha)$,
		\begin{align*}
			|\Im R_{ii}-\Im R_{jj}| &\lesssim |\Im  R_{ii}||R_{jj}||\Upsilon_{i}-\Upsilon_{j}| + |R_{ii}||\Im  R_{jj}||\Upsilon_{i}-\Upsilon_{j}| + |R_{ii}||R_{jj}||\Im (\Upsilon_{i}-\Upsilon_{j})|\\
			& \lesssim N^\eps \alpha |\Upsilon_{i}-\Upsilon_{j}| + |\Im (\Upsilon_{i}-\Upsilon_{j})|.
				\end{align*}

From Ward identity \eqref{eq:ward},
	\begin{align*}
			|\mathcal{A}_{i}| \le \frac{1}{N}\sum_{j}\frac{|R_{ij}(\tz_i)|^{2}}{|R_{ii}(\tz_i)|} \le \frac{C N^\eps \alpha}{N\eta}.
		\end{align*}
		Therefore, it follows from Lemma \ref{lem: bound on off diagonal} that on $\Omega(\alpha)$,
		\begin{align*}
			|\Upsilon_{i}-\Upsilon_{j}| \prec N^{\eps} \left( \frac{1}{q} + \sqrt{\frac{\alpha}{N\eta}}\right).
		\end{align*}
Similarly, we obtain from Lemma \ref{lem: bound on off diagonal} that on $\Omega(\alpha)$,
		\begin{align*}
			|\Im (\Upsilon_{i}-\Upsilon_{j})| \prec N^{\eps} \left( \frac{\alpha}{q} + \sqrt{\frac{\alpha}{N\eta}}\right).
		\end{align*}
		Therefore, it follows that
		\begin{align*}
			|\Im R_{ii}-\Im R_{jj}| \prec N^{\eps} \left( \frac{\alpha}{q}+\sqrt{\frac{\alpha}{N\eta}}\right).
		\end{align*}
		Since
		\begin{align*}
			\Im R_{ii} - \Im (m) = \frac{1}{N}\sum_{j}(\Im R_{ii}-\Im R_{jj}), 
		\end{align*}
	Equation \eqref{le:tbpfle} is established.
	\end{proof}

We are now ready to complete the proof of \eqref{eq:tbpimll} by proving that $\Omega(N^{\eps}/(N\eta))$ has overwhelming probability. Let $\alpha_1 = 1/q$. As already pointed, $\Omega(\alpha_1)$ has overwhelming probability. If $\alpha_1\ll N^{\eps} /(N\eta)$, we are done. Otherwise, by Lemma \ref{lem: bound off digonal} and Lemma \ref{lem: bound on digonal}, we have  
$$
	\Lambda_{o}^{\Im} + \Lambda_{d}^{\Im } \prec  N^\eps \left( \frac{\alpha_1}{q} + \sqrt{\frac{\alpha_1}{N\eta}} + \frac{1}{N\eta}\right).
$$
In particular, setting $$\alpha_2 =  \frac{\alpha_1}{q} + \sqrt{\frac{\alpha_1}{N\eta}} + \frac{1}{N\eta},$$ we have that $\Omega(\alpha_2)$ has overwhelming probability. If $\alpha_2 \ll N^{\eps} /(N\eta)$, we are done. Otherwise we continue. This process reach below $N^{\eps} /(N\eta)$ after a finite number of iterations because $q\gg 1$ and we have $\alpha_{k+1} \lesssim \alpha_{k}\max\left(1/q, \sqrt{1/(\alpha_{k}N\eta)}\right)$ as long as $\alpha_k \gg 1/(N\eta)$. 
\qed

\section{\ER graphs}\label{sec:ER}
Let $A$ be the normalized adjacency matrix of \ER graph with edge density $p=q^{2}/N$.
Each entry of the matrix $A=(a_{ij})_{1\le i,j\le N}$ is distributed as follows. Every diagonal entry $a_{ii}$ is zero. If $i < j$,
\begin{align*}
	a_{ij}=\begin{cases}
		\zeta/q &  \text{with probability} \; q^{2}/N, \\
		& \\
		0 & \text{with probability} \; 1-q^{2}/N,
	\end{cases}
\end{align*}
where
\begin{align*}
	\zeta:=(1-q^{2}/N)^{-1/2}.
\end{align*}
The resampling procedure is defined as in Definition \ref{def: resampling} with the random sets $S_k$ and an independent copy $A' = (a'_{ij})$ of $A$. For each integer $0 \leq k \leq N(N+1)/2$, we then obtain a matrix $A^{[k]}$ whose entries in $S_k$ above the diagonal are equal to the entries of $A'$ and whose entries in $S^c_k$ are equal to the corresponding entries in $A$.

Let $\nu_{1}\ge\cdots\ge\nu_{N}$ be the ordered eigenvalues of $A$. We denote an orthonormal basis of eigenvectors of $A$ by $\{\ww_{1},\cdots,\ww_{N}\}$, i.e., $A\ww_{i}=\nu_{i}\ww_{i}$ and $\lVert \ww_{i} \rVert$ = 1 for each $i$. Again from \cite{LV18}, with probability tending to one as $N$ goes to infinity, $\nu_{1}>\cdots>\nu_{N}$ and the eigenvectors are uniquely determined up to a sign. Similarly, we use the notation $\nu_{1}^{[k]}\ge\cdots\ge\nu_{N}^{[k]}$ and $\ww_{1}^{[k]},\cdots,\ww_{N}^{[k]}$ to denote the ordered eigenvalues and the associated unit eigenvectors of $A^{[k]}$.

In this section, we explain the proof of Theorem \ref{thm: main1A} and Theorem \ref{thm: main2A}. Let us fix $\ell=2$. The case $\ell=N$ can be handled in the almost similar argument. We define $N\times N$ matrix $\Ac=(\ac_{ij})$ by extracting the mean from the adjacency matrix $A$,
\begin{align*}
	\ac_{ij} := a_{ij}-\E a_{ij}.
\end{align*}
We find that
\begin{align*}
	A = \Ac + f \ee \ee^* - aI,
\end{align*}
where $f:=\zeta q$, $\ee:=N^{-1/2}(1,1,\cdots,1)^{T}\in\R^{N}$ and $a:=f/N$. The random correction term $\cX$ is again defined by setting 
\begin{align}\label{eq: random correction term A}
	\cX =\frac{1}{N}\sum_{1\le i,j \le N}\left(\ac^{2}_{ij}-\frac{1}{N}\right).
\end{align}
We note that $\Ac$ satisfies most properties of the sparse random matrix $H$ such as Lemma \ref{lem: delocalization}, Lemma \ref{lem:TW} and Lemma \ref{lem: local law}, see \cite{EKYY12, EKYY13, HLY20}. 

\subsection{Local laws and universality of \ER graphs}

The delocalization of eigenvectors is valid for $A$.  
\begin{lem}[Theorem 2.16 of {\cite{EKYY13}}]\label{lem: delocalization A}
	Assume $q\gg 1$. We have
	$$	\max_{1\le i \le N} \lVert \ww_{i} \rVert_{\infty} \prec \frac{1}{\sqrt N}.
	$$
\end{lem}

\noindent A non-asymptotic bound on the eigenvalue spacings of $A$ is given as follows.
\begin{lem}[Theorem 2.6 of {\cite{LL19}}]\label{lem: tail bound for gaps A}
	Assume $q\gg 1$. There exists a constant $c>0$ such that the following holds for any $\delta \geq N^{-c}$,
	$$
	\sup_{1\le i\le N-1} \prob\left( \nu_{i}-\nu_{i+1} \le \frac{\delta}{N}  \right)=O(\delta \log{N}).
	$$
\end{lem}

\noindent Let $\mathring\nu_{1}\ge\cdots\ge\mathring\nu_{N}$ be the ordered eigenvalues of $\Ac$. The following lemma explains the eigenvalue sticking between $\nu_{i+1}$ and $\mathring\nu_{i}$.
\begin{lem}[Eigenvalue sticking {\cite[Lemma 6.2]{EKYY12}}]\label{lem: eig val sticking}
	Assume $q\gg 1$. There exits $\delta>0$ such that we have for all $1\le i \le\delta N$
	\begin{align*}
		|\nu_{i+1}-(\mathring\nu_{i}-a)|\prec \frac{1}{N}.
	\end{align*}
    Similarly, if $N(1-\delta)\le i \le N$, it follows that
    \begin{align*}
    	|\nu_{i}-(\mathring\nu_{i}-a)|\prec \frac{1}{N}.
    \end{align*}
\end{lem}

\noindent Using the eigenvalue sticking lemma, we can show there is a gap of the order $N^{-2/3}$ between the extremal eigenvalues.
\begin{lem}[Tracy-Widom scaling for the gap]\label{lem:TW A}
	Assume $q\gtrsim N^{1/9}$. For any $\eps >0$, there exists a constant $c>0$ such that 
	$$
	\PP (  \nu_{2}-\nu_{3} \ge c N^{-2/3} ) \geq 1 - \eps. $$
	\begin{proof}
		According to \cite[Theorem 1.6]{HLY20}, we have for some constant
		\begin{align*}
			\PP (  \mathring\nu_{1}-\mathring\nu_{2} \ge (c/3) N^{-2/3} ) \geq 1 - \eps/3.
		\end{align*}
	(Setting diagonal entries to zeros does not harm the main argument of \cite{HLY20}.) 
		Thanks to Lemma \ref{lem: eig val sticking}, it follows that
		\begin{align*}
			\nu_{2}-\nu_{3} = \mathring\nu_{1} - \mathring\nu_{2} + O_{\prec}(N^{-1}),
		\end{align*}
	    and the lemma directly follows.
	\end{proof}
\end{lem}

\noindent We denote by $m_{A}(z)$ and $m_{\Ac}(z)$ the Cauchy-Stieltjes transforms of the empirical measures of eigenvalues of $A$ and $\Ac$ respectively:
$$
m_{A}(z) = \frac 1 N \sum_{i=1}^N \frac{1}{\nu_i - z},\quad  m_{\Ac}(z) = \frac 1 N \sum_{i=1}^N \frac{1}{\mathring\nu_i - z}.
$$

\noindent The local law estimates for \ER graphs holds. 
\begin{lem}\label{lem:locallaw A}
	Assume $q \gg 1$.  Let $ m_{\star}$ be as in Lemma \ref{lem: local law}. Uniformly on $w = \kappa + \mathfrak{i}\eta\in\mathcal{D}_0$, we have, with $z  = \mathcal{L}+w$, 
	\begin{align*}
		|m_{A}(z-a)-m_\star(z)| \prec \frac{1}{N\eta} + \frac 1 {q^3} + (\kappa + \eta)^{1/4} \left(\frac{1}{N\eta} + \frac 1 {q^3} \right)^{1/2}˚.
	\end{align*}
\end{lem}
\begin{proof}
Since $A + a I -  \Ac = f \ee^*$ has rank one, we have
$$
	|m_{A}(z-a)-m_{\Ac}(z)| = |m_{A+aI}(z)-m_{\Ac}(z)|\le\frac{\pi}{N\eta},
$$
the last inequality being a standard consequence of the interlacing inequality, see e.g. \cite[Lemma 7.1]{EKYY13}. The conclusion of the lemma then follows from Lemma \ref{lem:locallaw2} (local law) applied to $\Ac$.
\end{proof}

\noindent Next, by the standard argument using Helffer-Sj{\"{o}}strand formula, the following statements immediately follow as a consequence of Lemma \ref{lem:locallaw A}.

\begin{lem}[Eigenvalue rigidity]\label{lem: new rigidity estimate A}
	Assume $q\gg 1$. For all $2\le i\le N$, we have
	\begin{align*}
		|\nu_{i}- (\gamma_{i}-a)|\prec N^{-1/3}q ^{-3} + N^{-2/3}.
	\end{align*}
\end{lem}

\begin{cor}\label{cor: lower bound for spectral gap A}
	Let $\eps >0$ and assume $q \gtrsim N^{1/9}$. There exists $c>0$ such the following holds  for any $\delta>0$, for all $N$ large enough, with probability at least $1- \delta \log{N}$:
	\begin{align*}
		\nu_{2}-\nu_{i} \ge\begin{cases}
			c\delta N^{-1}  & \text{if}\;\; 3 \le i \le N^{\eps} \\
			c i^{2/3}N^{-2/3} & \text{if}\;\; N^{\eps} < i \le N.
		\end{cases}
	\end{align*}
	Moreover, on the event $\{\nu_{1}-\nu_{2}\ge c\delta N^{-1}\}$,
	the above inequalities holds with overwhelming probability.
\end{cor}

\subsection{Proof of Theorem \ref{thm: main1A}}
The outline of the proof is essentially same with that for the case of (centered) sparse random matrices. The adjacency matrix of \ER graph, $A$, can be regarded as a rank-one perturbation of a sparse random matrix so an adaptation is required. Since we now consider the second top eigenvector, not the top eigenvector, the argument of \eqref{eq: compare lambda_1 mu_1} and Lemma \ref{lem: eigen vec small perturbation} should be modified. See Lemma \ref{lem: small perturbation variation A} and Lemma \ref{lem: eigen vec small perturbation A} for detail.

Recall that $\nu_{1}\ge\cdots\ge\nu_{N}$ are the ordered eigenvalues of $A$ and $\ww_{1},\cdots,\ww_{N}$ are the associated unit eigenvectors of $A$.
For any $1\le i,j\le N$, denote by $A_{(ij)}$ the symmetric matrix obtained from $A$ by replacing the entry $a_{ij}$ and $a_{ji}$ with $a_{ij}''$ and $a_{ji}''$ respectively. We define $A^{[k]}_{(ij)}=(a^{[k]}_{(ij)})$ by
\begin{align*}
	a^{[k]}_{(ij)} = \begin{cases}
		a_{ij}'' & (i,j)\in S_{k}, \\
		a_{ij}'''& (i,j)\notin S_{k},
	\end{cases}
\end{align*}
where $a_{ij}'''$ is another independent copy of $a_{ij}$. Note that $\Ac'=(\ac_{ij}')$, $\Ac''=(\ac_{ij}'')$ and $\Ac'''=(\ac_{ij}''')$ are also independent copies. We denote the ordered eigenvalues of $A^{[k]}$ and their associated eigenvectors by $\nu_{1}\ge\cdots\ge\nu_{N}$ and $\ww_{1},\cdots,\ww_{N}$.
Denote by $(st)$ a random pair of indices chosen uniformly from $\{(i,j): 1\le i\le j\le N
\}$. Note that 
\begin{align*}
	|\{(i,j): 1\le i\le j\le N\}|=N(N+1)/2
\end{align*}
Let $\mu_{1}\ge\cdots\ge\mu_{N}$ be the ordered eigenvalues of $A_{(st)}$ and, let $\uu_{1},\cdots,\uu_{N}$ be the associated unit eigenvectors of $A_{(st)}$. Similarly, we define $\mu_{1}^{[k]}\ge\cdots\ge\mu_{N}^{[k]}$ and $\uu_{1}^{[k]},\cdots,\uu_{N}^{[k]}$ for $A^{[k]}_{(st)}$. We apply Lemma \ref{lem: superconcentration} with $Y=A$ and $f(A)=\nu_{2} - L - \cX$: 
\begin{align}\label{eq: application of superconcentration lem A}
	\E\left[ \big(\nu_{2}-\mu_{2}-Q_{st}\big)\big(\nu_{2}^{[k]}-\mu_{2}^{[k]}-Q_{st}^{[k]}\big) \right]\le\frac{2\text{Var}(\nu_{2}- L - \cX)}{k}\cdot\frac{N(N+1)+2}{N(N+1)},
\end{align}
where
\begin{align*}
	Q_{st}&:=\frac{2}{N}(\ac_{st}^{2}-(\ac_{st}'')^{2}) \quad \text{and}\nn\\
	Q_{st}^{[k]}&:=\begin{cases}
		\frac{2}{N}((\ac_{st}')^{2}-(\ac_{st}'')^{2}) &\text{if } (st)\in S_{k},\\
		\frac{2}{N}(\ac_{st}^{2}-(\ac_{st}''')^{2}) &\text{if } (st)\notin S_{k}.
	\end{cases}
\end{align*}

\begin{lem}\label{lem: small perturbation variation A}
	Let us write $\ww_{2}=(w_{1},\cdots,w_{N})$ and $\uu_{2}=(u_{1},\cdots,u_{N})$. There exists $\eps>0$ such that the following holds with overwhelming probability:
	\begin{align*}
		Z_{st}u_{s}u_{t} - \frac{N^{\eps}}{q^{3}N^{2}}
		\le \nu_{2}-\mu_{2} 
		\le Z_{st}w_{s}w_{t} + \frac{N^{\eps}}{q^{3}N^{2}}
	\end{align*}
	where 
	\begin{align*}
		Z_{st}:=2(\ac_{st}-\ac_{st}'').
	\end{align*}
	Similarly, with overwhelming probability, we have
	\begin{align*}
		Z_{st}^{[k]} u_{s}^{[k]}u_{t}^{[k]} - \frac{N^{\eps}}{q^{3}N^{2}}
		\le \nu_{2}^{[k]}-\mu_{2}^{[k]} 
		\le Z_{st}^{[k]} w_{s}^{[k]}w_{t}^{[k]} + \frac{N^{\eps}}{q^{3}N^{2}},
	\end{align*}
	where $\ww_{2}^{[k]}=(w_{1}^{[k]},\cdots,w_{N}^{[k]})$, $\uu_{2}^{[k]}=(u_{1}^{[k]},\cdots,u_{N}^{[k]})$ and
	\begin{align*}
		Z_{st}^{[k]}:=\begin{cases}
			2(\ac'-\ac'') &\text{if } (st)\in S_{k},\\
			2(\ac-\ac''') &\text{if } (st)\notin S_{k}.
		\end{cases}
	\end{align*}
	\begin{proof}
		By spectral theorem, we have
		\begin{align*}
			\langle \uu_{2},A\uu_{2}\rangle = \nu_{1}|\langle \uu_{2},\ww_{1}\rangle|^{2}+\sum_{i=2}^{N}\nu_{i}|\langle \uu_{2},\ww_{i}\rangle|^{2}
			\le (\nu_{1}-\nu_{2})|\langle \uu_{2},\ww_{1}\rangle|^{2}+\langle \ww_{2},A\ww_{2} \rangle.
		\end{align*}
		We write 
		\begin{align*}
			\ww_{1}=\alpha\uu_{2}+\beta\xx
		\end{align*}
		where $\alpha=\langle\uu_{2},\ww_{1}\rangle$, $\xx\in\text{span}(\uu_{1},\uu_{3},\cdots,\uu_{N})$ and $\lVert \xx\rVert^{2}=1-\alpha^{2}$. Since
		\begin{align*}
			A_{(st)}\ww_{1}=A\ww_{1}+(A_{(st)}-A)\ww_{1}=\nu_{1}\ww_{1}+(A_{(st)}-A)\ww_{1}
		\end{align*}
		and also
		\begin{align*}
			A_{(st)}\ww_{1}=\alpha \mu_{2} \uu_{2} + \beta  A_{(st)}\xx,
		\end{align*}
		it follows that
		\begin{align*}
			\nu_{1}\ww_{1} = \alpha \mu_{2} \uu_{2} + \beta A_{(st)}\xx + (A-A_{(st)})\ww_{1}.
		\end{align*}
		Then,
		\begin{align*}
			\nu_{1}\alpha=\nu_{1}\langle \uu_{2},\ww_{1}\rangle = \langle \uu_{2},\nu_{1}\ww_{1}\rangle
			=\mu_{2}\alpha+\langle \uu_{2},(A-A_{(st)})\ww_{1} \rangle.
		\end{align*}
		By the eigenvector delocalization,
		\begin{align*}
			|(\nu_{1}-\mu_{2})\alpha| = |\langle \uu_{2},(A-A_{(st)})\ww_{1} \rangle|\prec \frac{1}{qN}. 
		\end{align*}
		According to \cite[Theorem 6.2]{EKYY13}, we have $\nu_{1}\sim \zeta q+(\zeta q)^{-1}$ with overwhelming probability. Also, by the eigenvalue rigidity, we find $\mu_{2}\le C$ with overwhelming probability. Finally, we obtain
		\begin{align*}
			|\alpha|\prec\frac{1}{q^{2}N},
		\end{align*}
		which implies 
		\begin{align}\label{eq: compare lambda_2 mu_2}
			\langle \uu_{2},A\uu_{2}\rangle \le \frac{N^{\eps}}{q^{3}N^{2}}+\langle \ww_{2},A\ww_{2} \rangle.
		\end{align}
		Similarly, we have with overwhelming probability
		\begin{align*}
			\langle \ww_{2},A_{(st)}\ww_{2}\rangle \le \frac{N^{\eps}}{q^{3}N^{2}}+\langle \uu_{2},A_{(st)}\uu_{2} \rangle.
		\end{align*}
		As a result, it follows with overwhelming probability
		\begin{align*}
			\langle \uu_{2},(A-A_{(st)})\uu_{2} \rangle - \frac{N^{\eps}}{q^{3}N^{2}}
			\le \nu_{2}-\mu_{2} 
			\le \langle \ww_{2},(A-A_{(st)})\ww_{2}\rangle + \frac{N^{\eps}}{q^{3}N^{2}}.
		\end{align*}
		Using the same argument, we observe with overwhelming probability
		\begin{align*}
			\langle \uu_{2}^{[k]},(A^{[k]}-A_{(st)}^{[k]})\uu_{2}^{[k]} \rangle - \frac{N^{\eps}}{q^{3}N^{2}}
			\le \nu_{2}^{[k]}-\mu_{2}^{[k]}
			\le \langle \ww_{2}^{[k]},(A^{[k]}-A_{(st)}^{[k]})\ww_{2}^{[k]}\rangle + \frac{N^{\eps}}{q^{3}N^{2}}.
		\end{align*}
	\end{proof}
\end{lem}

We set $T_{1}=(Z_{st}w_{s}w_{t}+N^{\eps}/q^{3}N^{2}-Q_{st})(Z_{st}^{[k]}w_{s}^{[k]}w_{t}^{[k]}+N^{\eps}/q^{3}N^{2}-Q_{st}^{[k]})$, $T_{2}=(Z_{st}w_{s}w_{t}+N^{\eps}/q^{3}N^{2}-Q_{st})(Z_{st}^{[k]}u_{s}^{[k]}u_{t}^{[k]}-N^{\eps}/q^{3}N^{2}-Q_{st}^{[k]})$, $T_3 = (Z_{st}u_{s}u_{t}-N^{\eps}/q^{3}N^{2}-Q_{st})(Z_{st}^{[k]}w_{s}^{[k]}w_{t}^{[k]}+N^{\eps}/q^{3}N^{2}-Q_{st}^{[k]})$, $T_4 = (Z_{st}u_{s}u_{t}-N^{\eps}/q^{3}N^{2}-Q_{st})(Z_{st}^{[k]}u_{s}^{[k]}u_{t}^{[k]}-N^{\eps}/q^{3}N^{2}-Q_{st}^{[k]})$. We have
\begin{align}\label{eq: top eig val difference under sigle resample A}
	\min(T_{1},T_{2},T_{3},T_{4})\le  \big(\nu_{2}-\mu_{2}-Q_{st}\big)\big(\nu_{2}^{[k]}-\mu_{2}^{[k]}-Q_{st}^{[k]}\big) \le \max(T_{1},T_{2},T_{3},T_{4}).
\end{align}

\begin{lem}\label{lem: eigen vec small perturbation A}
	Assume $q \gtrsim N^{1/9}$ and let $c,\delta >0$ be such that $N^{c + \delta} \ll q $. For $1 \leq i \le j \leq N$, let $\uu_{2}^{(ij)}$ be a unit eigenvector of $A_{(ij)}$ associated with the second largest eigenvalue of $A_{(ij)}$. Then, on the event $\left\{  \nu_{2}-\nu_{3} \ge N^{-1-c}\right\} $,
	the event
	\begin{align*}
		\bigcap_{1 \leq i\leq j \leq N} \left\{   \inf_{s\in\{\pm 1\}}\lVert s\ww_{2}-\uu_{2}^{(ij)} \rVert_{\infty}\le N^{-1/2-\delta}\right\}
	\end{align*}
	holds with overwhelming probability. The analogous result for $H^{[k]}_{(ij)}$ also holds.
	\begin{proof}
		We shall modify the proof of Lemma \ref{lem: eigen vec small perturbation}. Let $\mu^{(ij)}_{1}\ge\cdots\ge\mu^{(ij)}_{N}$ be the ordered eigenvalues of $A_{(ij)}$ and, let $\uu_{1}^{(ij)},\cdots,\uu_{N}^{(ij)}$ be the associated unit eigenvectors of $A_{(ij)}$. According to \eqref{eq: compare lambda_2 mu_2}, we have with overwhelming probability
		\begin{align*}
			\nu_{2} & \ge \langle \uu_{2}^{(ij)},A\uu_{2}^{(ij)}\rangle - \frac{N^{\eps}}{q^{3}N^{2}} \\
			& = \mu^{(ij)}_{2}+\langle\uu_{2}^{(ij)},(A-A_{(ij)})\uu_{2}^{(ij)}\rangle - \frac{N^{\eps}}{q^{3}N^{2}} \\
			&\ge \mu^{(ij)}_{2}-2|\ac_{ij} - \ac_{ij}''|\lVert \uu_{2}^{(ij)} \rVert_{\infty}^{2} - \frac{N^{\eps}}{q^{3}N^{2}} \\
			&\ge \mu^{(ij)}_{2}-\frac{N^{\eps}}{qN}.
		\end{align*}
		Reversing the role of $A$ and $A_{(ij)}$, we also have with overwhelming probability
		\begin{align*}
			\mu^{(ij)}_{2} \ge \nu_{2}-\frac{N^{\eps}}{qN}.
		\end{align*}
		Thus, it follows that with overwhelming probability
		\begin{align}\label{eq: eigvenvalue perturbation A}
			\max_{1\le i\le j\le N} |\nu_{2}-\mu^{(ij)}_{2}| \le \frac{N^{\eps}}{qN}.
		\end{align}
		We write
		\begin{align*}
			\uu_{2}^{(ij)} = \sum_{\ell=1}^{N}\alpha_{\ell}\ww_{\ell},
		\end{align*}
		and get
		\begin{align*}
			\nu_{2} \uu_{2}^{(ij)} = \sum_{\ell=1}^{N}\nu_{\ell}\alpha_{\ell}\ww_{\ell} + (A_{(ij)}-A)\uu_{2}^{(ij)}+(\nu_{2} - \mu_{2}^{(ij)}) \uu_{2}^{(ij)}.
		\end{align*}
		Next, by taking an inner product with $\vv_{\ell}$ for $\ell\neq 2$, we obtain
		\begin{align*}
			\bigg( (\nu_{2}-\nu_{\ell})+ (\mu_{2}^{(ij)}-\nu_{2}) \bigg) \alpha_{\ell} = \langle \ww_{\ell},(A_{(ij)}-A)\uu_{2}^{(ij)} \rangle.
		\end{align*}
		According to \cite[Theorem 6.2]{EKYY13} and Corollary \ref{cor: lower bound for spectral gap A}, the following holds with overwhelming probability on the event $\left\{  \nu_{2}-\nu_{3} \ge N^{-1-c}\right\}$:
		\begin{align*}
			|\nu_{2}-\nu_{\ell}|\gtrsim\begin{cases}
				q & \ell=1, \\
				N^{-1-c} & 3 \le \ell \le N^{\eps}, \\
				\ell^{2/3}N^{-2/3}  & N^{\eps}< \ell \le N.
			\end{cases}
		\end{align*}
		Due to \eqref{eq: eigvenvalue perturbation A}, we have with overwhelming probability
		\begin{align*}
			|\nu_{2}-\nu_{\ell}|\gg |\mu_{2}^{(ij)}-\nu_{2}|,
		\end{align*}
		for every $\ell\in\{1,\cdots, N\}$. Since the eigenvector delocalization implies
		\begin{align*}
			\left| \langle \ww_{\ell},(A_{(ij)}-A)\uu_{2}^{(ij)} \rangle \right| 
			\prec \frac{1}{qN},
		\end{align*}
		we can observe
		\begin{align}\label{eq : bounds for coeffi A}
			|\alpha_{\ell}|\prec\begin{cases}
				q^{-2}N^{-1}& \ell=1, \\
				q^{-1}N^{c}& 3\le \ell \le N^{\eps}, \\
				q^{-1}\ell^{-2/3}N^{-1/3}& N^{\eps}< \ell \le N.
			\end{cases}
		\end{align}
		What remains can be done similarly as we did in Section \ref{pf: eigen vec small perturbation}.
	\end{proof}
\end{lem}
Next, let $0 < \delta < 1/9$  and  $0 < \eps < \delta /3$ to be defined later, we define the events 
\begin{align}\label{eq: event 1 conditions A}
	\mathcal{E}_{1} :=	\left\{ \max \left(\lVert \ww_2   \rVert_{\infty},\lVert \uu_2 \rVert_{\infty},\lVert  \ww_{2}^{[k]}\rVert_{\infty},\lVert \uu_2^{[k]} \rVert_{\infty}\right) \le  N^{\eps - 1/2} \right\},
\end{align}
\begin{align}\label{eq: event 2 conditions A}
	\mathcal{E}_{2} :=	  \left\{   \max \left(  \lVert \ww_{2}-\uu_{2} \rVert_{\infty} , \lVert \ww^{[k]}_{2}-\uu^{[k]}_{2} \rVert_{\infty} \right) \le N^{-1/2-\delta}\right\}.
\end{align}
Set the event $\mathcal{E}:=\mathcal{E}_{1}\cap\mathcal{E}_{2}$. Let $c >0$ such that $c + \delta < 1/9$. According to Lemma \ref{lem: delocalization}, \cite[Theorem 2.6]{LL19} and Lemma \ref{lem: eigen vec small perturbation A}, we have
$\prob(\mathcal{E}^{c})= O(N^{-c}\log{N})$ by choosing the $\pm$-phase properly for $\uu_{(ij)}$ and $\uu_{(ij)}^{[k]}$. On the event $\mathcal{E}$, 
it follows that from \eqref{eq: top eig val difference under sigle resample A}
\begin{multline}\label{eq:lmQk A}
	\big(\nu_{2}-\mu_{2}-Q_{st}\big)\big(\nu_{2}^{[k]}-\mu_{2}^{[k]}-Q_{st}^{[k]}\big)\ge Z_{st}Z_{st}^{[k]}w_{s}w_{t}w_{s}^{[k]}w_{t}^{[k]}-O\left( | Z_{st}Z_{st}^{[k]}| N^{3\eps - 2 - \delta} \right)\\
	-  |Q_{st}Z_{st}^{[k]}|N^{2\eps - 1}  -|Q_{st}^{[k]}Z_{st}|  N^{2\eps - 1} - |Q_{st}Q_{st}^{[k]}| - o(N^{-3}).
\end{multline}
The proof is done by the following two lemmas.
\begin{lem}
	If $4 \eps + \delta < 1/9$, we have
	\begin{align}\label{eq: expectation estimate 1 A}
		\E\left[Z_{st}Z_{st}^{[k]}w_{s}w_{t}w_{s}^{[k]}w_{t}^{[k]}\indic_{\mathcal{E}^c} \right] = o\left(\frac 1 {N^{3}}\right),
	\end{align}
	and
	\begin{align}\label{eq: expectation estimate 2 A}
		\E\left[ (\nu_{2}-\mu_{2}-Q_{st})(\nu_{2}^{[k]}-\mu_{2}^{[k]}-Q_{st}^{[k]})\indic_{\mathcal{E}^{c}} \right]=o\left(\frac 1 {N^{3}}\right).
	\end{align}
\end{lem}
\begin{lem}
	We have
	\begin{align*}
		\E\left[ Z_{st}Z_{st}^{[k]} w_{s}w_{t}w_{s}^{[k]}w_{t}^{[k]}\right] =\frac{2}{N^3} \E \left[\langle \ww_{2}, \ww_{2}^{[k]} \rangle^2 \right]+o\left(\frac 1 {N^{3}}\right).
	\end{align*}
\end{lem}
The above two lemma can be shown in the very similar way of Lemma \ref{lem: expectation estimate} and Lemma \ref{lem: expectation estimate2} so we omit the detail. Applying \eqref{eq: application of superconcentration lem}, we establish
\begin{align*}
	\E \left[\langle \ww_{2}, \ww_{2}^{[k]} \rangle^2 \right] \leq  \frac{N^3\text{Var}(\nu_{2}- L - \cX)}{k}\left( 1 + o(1) \right)  + o(1).
\end{align*}
Using \eqref{eq:riglambda1} and Cauchy interlacing, we have for any $\eps >0$,
$$
\text{Var}(\nu_{2}- L - \cX) = O ( N^{\eps-4/3} ),  
$$
which concludes the proof.
\qed

\subsection{Proof of Theorem \ref{thm: main2A}}
As in the previous subsection, we shall rely on the same strategy described in Section \ref{sec: proof strategy} and focus on explaining how to modify some details in regard to rank-one perturbation.

For $z=E+\mathfrak{i}\eta$ with $\eta>0$ and $E\in\R$, we define (with an abuse of notation)
\begin{align*}
	R(z):=(A-zI)^{-1},
\end{align*}
where $I$ denotes the identity matrix. We denote by $R^{[k]}(z)$ the resolvent of $A^{[k]}$. Then, as we showed in Section \ref{high level pf of main thm 2}, the desired result follows from the following two lemmas. 

\begin{lem}\label{lem: lem13 in BLZ A}
	Assume $q \gtrsim N^{1/9}$ and $k\ll N^{5/3}$. Let $R(z)$ be the resolvent of $A$. Then, there exists $\delta_0 >0$ such that for all $0  < \delta < \delta_0$, there exists $c>0$ such that, with overwhelming probability,
	\begin{align*}
		\sup_{z} \max_{1\le i,j\le N} N\eta|\Im R_{ij}^{[k]}(z-a)-\Im R_{ij}(z-a)|\leq  N^{-c},
	\end{align*}
	where the supremum is over all $z=E+\mathfrak{i}\eta$ with $|E-\cL |\le N^{-2/3+\delta}$ and $\eta = N^{-2/3-\delta}$, and the term $\cL$ is defined as in Lemma \ref{lem: local law} with setting $\cX$ as in \eqref{eq: random correction term A}.
\begin{proof}
	We notice that
	\begin{align*}
		R^{[k]}_{ij} -  R_{ij}  &= \sum_{t=1}^k  \left( R^{[t]}_{ij} -  R^{[t-1]}_{ij}\right)  \\
		&= \sum_{t=1}^k  (a_{i_tj_t} - a'_{i_tj_t}) (R^{[t]} E_{i_t j_t} R^{[t-1]} )_{ij} \\
		&= \sum_{t=1}^k  (\ac_{i_tj_t} - \ac'_{i_tj_t}) (R^{[t]} E_{i_t j_t} R^{[t-1]} )_{ij}.
	\end{align*}
    We also find that the resolvent estimates of $H$, Lemma \ref{lem: lem10 in BLZ}, still holds for 
    the resolvent of $A$.
    \begin{lem}\label{lem: lem10 in BLZ A}
    	Assume $q \gtrsim N^{1/9}$ and let $0 < \delta < 1/3$. Let $R(z)$ be the resolvent of $A$. We have
    	\begin{align*}
    		\sup_z \max_{1 \leq i,j \leq N} \Big|\big|R(z-a)_{ij}\big| - \delta_{ij} \Big| \prec \frac 1 q + \frac{1}{N\eta}, 	\end{align*}
    	and 
    	\begin{align*}
    		\sup_z  \max_{1 \leq i,j \leq N} \big|\Im R(z-a)_{ij}\big|\prec \frac{1}{N\eta} ,
    	\end{align*}
    	where the two suprema are over all $z = E + \mathfrak{i}\eta$ with $|E-\cL| \leq N^{-2/3 + \delta}$ and $\eta = N^{-2/3 -\delta}$, and the term $\cL$ is defined as in Lemma \ref{lem: local law} with setting $\cX$ as in \eqref{eq: random correction term A}. 
    \end{lem}
    \noindent The first statement of the lemma immediately follows from \cite[Theorem 2.9]{EKYY13}. We shall prove the second statement in Subsection \ref{sec: resolvent of ER}. Lemma \ref{lem: lem10 in BLZ A} is an essential ingredient. What remains would be a straightforward modification of the proof of Lemma \ref{lem: lem13 in BLZ}. Note that we used the trivial inequality $h_{ij}\prec q^{-1}$ in the proof of Lemma \ref{lem: lem13 in BLZ} and it still holds that $a_{ij}\prec q^{-1}$.
\end{proof}
\end{lem}
\begin{lem}\label{lem: lem14 in BLZ A}
	We write $\ww_{2}=(w_{1},\cdots,w_{N})$ and $\ww^{[k]}_{2}=(w^{[k]}_{1},\cdots, w^{[k]}_{N})$. Assume $q \gtrsim N^{1/9}$ and $k\ll N^{5/3}$. Let $0 < \delta <\delta_0$ be as in Lemma \ref{lem: lem13 in BLZ A}. There exists $c'>0$ such that with probability  $1-o(1)$ it holds that 
	\begin{align*}
		\max_{1\le i,j \le N}N|\eta\Im R_{ij}(z)-w_{i}w_{j}|\le N^{-c'}\quad\text{and}\quad
		\max_{1\le i,j \le N}N|\eta\Im R^{[k]}_{ij}(z)-w_{i}^{[k]}w_{j}^{[k]}|\le N^{-c'},
	\end{align*}
	with $z=\nu_{2}+\mathfrak{i}\eta$ and $\eta = N^{-2/3-\delta}$.
\end{lem}

\begin{proof}[Proof of Lemma \ref{lem: lem14 in BLZ A}]
	The next lemma is a modification of Lemma \ref{lem: lem12 in BLZ}. See the following lemma.
	\begin{lem}\label{lem: lem12 in BLZ A}
		Assume $q\gtrsim N^{1/9}$  and $k \ll N^{5/3}$. Then, if $0 < \delta < \delta_0$ with $\delta_0$ as in Lemma \ref{lem: lem13 in BLZ A}, we have
		\begin{align*}
			|\nu_{2}-\nu_{2}^{[k]}|\prec N^{-2/3-\delta}.
		\end{align*}
			\begin{proof}
				If $\nu_{2}=\nu_{2}^{[k]}$, we are done. Thus, suppose $\nu_{2}^{[k]}<\nu_{2}$. We set $\eta=N^{-2/3-\delta}$. According to Lemma \ref{lem: lem9 in BLZ}, we can find $1\le i\le N$ such that
				\begin{align*}
					\frac{1}{2\eta^{2}} \le N\eta^{-1}\Im R(\nu_{2}+\mathfrak{i}\eta)_{ii}.
				\end{align*}
				Since we have $|(\nu_{2}+a)-\mathcal{L}|\prec N^{-2/3}$, it follows from Lemma \ref{lem: lem9 in BLZ} that
				\begin{align*}
					N\eta^{-1}\Im R^{[k]}(\nu_{2}+\mathfrak{i}\eta)_{ii}\prec \left(\min_{1\le j\le N}\left|\nu_{2}-\nu_{j}^{[k]}\right|\right)^{-2}.
				\end{align*}
				With overwhelming probability, $\nu_{1}^{[k]}\gg\nu_{2}>\nu_{2}^{[k]}\ge\nu_{3}^{[k]}\ge\cdots\ge\nu_{N}^{[k]}$ so we have
				\begin{align*}
					\min_{1\le j\le N}\left|\nu_{2}-\nu_{j}^{[k]}\right| = \left|\nu_{2}-\nu_{2}^{[k]}\right|.
				\end{align*}
				Applying Lemma \ref{lem: lem13 in BLZ A}, we get the desired result by showing
				\begin{align*}
					N\eta^{-1}\Im R^{[k]}(\nu_{1}+\mathfrak{i}\eta)_{ii} \gtrsim \frac{1}{\eta^{2}}.
				\end{align*}
				The other case $\nu_{2}^{[k]}>\nu_{2}$ can be proven by reversing the role $A$ and $A^{[k]}$.
			\end{proof}
	\end{lem}
	We fix $0 < \delta < \delta_0$ and set $\eta = N^{-2/3-\delta}$. 
	We write $\ww_{m}=(\ww_{m}(1),\ldots,\ww_{m}(N))$ and $\ww^{[k]}_{m}=(\ww^{[k]}_{m}(1),\ldots, \ww^{[k]}_{m}(N))$ for $m\neq 2$. By the spectral theorem, we have
	\begin{align*}
		N\eta\Im R(z)_{ij} = \frac{N\eta^{2}w_{i}w_{j}}{(\nu_{2}-E)^{2}+\eta^{2}} + \sum_{m\neq 2}^N\frac{N\eta^{2}\ww_{m}(i)\ww_{m}(j)}{(\nu_{m}-E)^{2}+\eta^{2}}.
	\end{align*}
	Let $\eps >0$ and let $N':=\lfloor N^{2\eps} \rfloor$. We see that with overwhelming probability: for all $E$ satisfying $|E- (\mathcal{L}-a)|\le N^{-2/3+\eps}$, we have, for some $C >0$,
	\begin{align}\label{eq:djeijd A}
		\left|\sum_{m=N'+1}^{N}\frac{N\ww_{m}(i)\ww_{m}(j)}{(\nu_{m}-E)^{2}+\eta^{2}}\right| \leq C N^{\eps} (N')^{-1/3}N^{4/3}.
	\end{align}
	We can find $c_0>0$ such that 
	\begin{align*}
		\prob\left(\cE \right)\ge 1-\eps/2,
	\end{align*}
	where $\cE$ is the event that \eqref{eq:djeijd A} holds, $ \nu_{1}-\nu_{2} \ge c_{0}q$, $\nu_{2}-\nu_{3} > c_0 N^{-2/3}$ 
	and  $\max_m \|\ww_m\|_\infty^2  \leq N^{\eps-1}$. On the event $\cE$, we find for all $E$ with $|\nu_{2}-E|\le (c_{0}/2)N^{-2/3}$ that for some $C >0$,
	\begin{align*}
		\left|\sum_{m=3}^{N'}\frac{N\ww_{m}(i)\ww_{m}(j)}{(\nu_{m}-E)^{2}+\eta^{2}}\right| 
		\leq C N^{\eps} N'N^{4/3},
	\end{align*}
	and
	\begin{align*}
		\left|\frac{N\ww_{1}(i)\ww_{1}(j)}{(\nu_{1}-E)^{2}+\eta^{2}}\right| 
		\leq N^{\eps}q^{-2}.
	\end{align*}
	We fix $\delta' >0$ such that $\delta + \delta' < \delta_0$. On the event $\cE$, for any $E$ such that $|\nu_{2}-E|\le \eta N^{-\delta'}$, we have 
	\begin{align*}
		\left| \frac{N\eta^{2}w_{i}w_{j}}{(\nu_{2}-E)^{2}+\eta^{2}} - Nw_{i}w_{j}\right| \leq N^{\eps -2\delta'}.
	\end{align*}
	The proof is done by following the argument of the proof of Lemma \ref{lem: lem14 in BLZ} and applying Lemma \ref{lem: lem12 in BLZ A}.
\end{proof}

\subsection{Resolvent of \ER graph: Proof of Lemma \ref{lem: lem10 in BLZ A}}\label{sec: resolvent of ER}
The second statement of Lemma \ref{lem: lem10 in BLZ A} will be shown in this subsection. It is enough to prove the following lemma.
\begin{lem}
	Assume $q \gtrsim N^{1/9}$ and let $0 \leq \delta < 1/3$. $R(z)$ be the resolvent of $A$. We have
	\begin{align*}
		\sup_z \max_{1\leq i , j \leq N} |\Im R_{ij}(z-a)-\delta_{ij}\Im (m_\star(z-a))| \prec \frac{1}{N\eta},
	\end{align*}
	where the supremum is over all $z=E+\mathfrak{i}\eta$ with $|E-\cL|\le 2 N^{-2/3+\delta}$ and $\eta =  N^{-2/3-\delta}$, and the term $\cL$ is defined as in Lemma \ref{lem: local law} with setting $\cX$ as in \eqref{eq: random correction term A}.
	
\begin{proof}
		We can prove this lemma by using the same argument in Subsection \ref{sec: estimate on the imaginary part of local law} with some additional ingredients, Lemma \ref{lem: additional resolvent estimate A} and Lemma \ref{lem: bound on off diagonal A}. We already know it is sufficient to prove that for any deterministic real $\kappa$ with $|\kappa| \leq N^{-2/3 + \delta}$, 
	\begin{equation}\label{eq:tbpimll A}
		\max_{1\leq i , j \leq N} |\Im R_{ij}(\tz-a)-\delta_{ij}\Im (m_\star(\tz-a))| \prec \frac{1}{N\eta},
	\end{equation}
	with $\tz$ defined by $\tz = \kappa + L + \cX + \mathfrak{i}\eta$.
	
	For $\eta = N^{-2/3 - \delta}$, we set $\cD_1 = \{ z = E + \mathfrak{i}\eta \in \C_+ : |E - \cL| \leq 2N^{-2/3+ \delta} \}$. We introduce the following variables
	\begin{align*}
		&\Lambda_{e} :=\sup_{z \in \cD_1} \max_{i,  j}|R_{ij}(z-a) - \delta_{ij} m_\star(z)|,  \quad 	 \Lambda_A := \sup_{z \in \cD_1}|m_A(z-a)-m_\star(z)|. \\ 	
		& \Lambda_{o}^{\Im } :=\max_{i\neq j} \Im R_{ij}(\tz-a), \quad \Lambda_{d}^{\Im }:=\max_{i}|\Im R_{ii}(\tz-a)-\Im (m_\star(\tz))|.
	\end{align*}
	For $\alpha >0$ such that $(N\eta)^{-1} \leq \alpha \leq 1/q$, we introduce the events
	\begin{align*}
		\Omega := \left\{\Lambda_A  \le N^{\eps}(N\eta)^{-1} \; ; \; \Lambda_{e} \le N^{\eps}q^{-1} \right\} \quad \hbox{and} \quad 	\Omega( \alpha) := \Omega \cap \left\{ \Lambda_{o}^{\Im }+\Lambda_{d}^{\Im }\le N^{\eps}\alpha \right\},
	\end{align*}
	where $\eps>0$ is an arbitrarily small constant to be chosen later. By Lemma \ref{lem:locallaw A}, \cite[Theorem 2.9]{EKYY13} and Lemma \ref{lem:locallaw02}, the event $\Omega(1/q)$ holds with overwhelming probability. If $\Omega(N^\eps/(N\eta))$ has overwhelming probability for all fixed $0 < \eps< 1/9$ then \eqref{eq:tbpimll A} holds and the proof of the lemma is done.
	
	Note that, if $\Omega(\alpha)$ holds, for all $\T$ with $|\T| \leq 2$, for all $i,j \notin \T$ with $ i \ne j$,
	\begin{align}\label{eq:OmRT A}
		\sup_{z \in \cD_1} |R^{(\T)}_{ii}(z)| \asymp  1 , \quad \sup_{z \in \cD_1} |R^{(\T)}_{ij}(z)| \lesssim \frac{N^\eps}{q}   \quad	\hbox{and} \quad \Im R^{(\T)}_{ii} + \Im R^{(\T)}_{ij} \lesssim  N^\eps \alpha.
	\end{align}
	
	For ease of notation, in the sequel, we often omit $\tz$ and write $m$, $m_\star$ and $R$ in place of $m(\tz-a)$, $m_\star(\tz)$ and $R(\tz-a)$. We define $Z_{ij}(\tz-a)$ by setting
	\begin{align*}
		Z_{ij}:=\sum_{k,l}^{(ij)}a_{ik}R_{kl}^{(ij)}a_{lj}.
	\end{align*}
	We set
	\begin{align*}
		Z_{i}&:= Z_{ii} - \frac{1}{N}\sum_{k}^{(i)} R_{kk}^{(i)} \\
		&= \sum_{k}^{(i)}\left( |a_{ik}|^{2} - \frac{1}{N} \right)R_{kk}^{(i)} + \sum_{k\neq l}^{(i)}a_{ik} R_{kl}^{(i)} a_{li}.
	\end{align*}
	In addition, let us define $\tz_i = \kappa + L + \cX_i + \mathfrak{i}\eta$ and $\tz_{ij} = \kappa + L + \cX_{ij} + \mathfrak{i}\eta$ with
	$$
	\cX_i = \frac 1 N \sum_{k,l}^{(i)} \left( \mathring{a}_{ij}^2 - \frac 1 N \right) \quad \hbox{ and } \quad \cX_{ij} = \frac 1 N \sum_{k,l}^{(ij)} \left( \mathring{a}_{ij}^2 - \frac 1 N \right).
	$$
	Let us set $\tilde R^{(i)} := R^{(i)}(\tz_i-a)$ and $\tilde R^{(ij)} :=  R^{(ij)}(\tz_{ij}-a)$. The following lemmas are new inputs to show the desired results.
	
	\begin{lem}\label{lem: additional resolvent estimate A}
		Assume $q\gg 1$ and $1/(N\eta) \le \alpha \le 1/q $. We have on $\Omega(\alpha)$,
		\begin{align*}
			\sum_{k,l}^{(i)} \frac{f}{N} \Im \left(\tilde R_{kl}^{(i)}\right) \mathring{a}_{li}  \prec \frac{f^{2}}{N}\frac{1}{N\eta},
		\end{align*}
		and
		\begin{align*}
			\sum_{k,l}^{(i)} \frac{f^{2}}{N^{2}} \Im \left(\tilde R_{kl}^{(i)}\right)  \prec \frac{f^{2}}{N}\frac{1}{N\eta}.
		\end{align*}
		\begin{proof}
			Using the spectral decomposition of $R^{(i)}$, we have
			\begin{align*}
				\sum_{k,l}^{(i)} \frac{f}{N} \Im \left[\tilde R_{kl}^{(i)}\right] \mathring{a}_{li}
				= \frac{f\sqrt{N-1}}{N}\left( \sum_{\alpha}\frac{\eta\langle \ee_{N-1}, \ww^{(i)}_{\alpha} \rangle}{(\nu^{(i)}_{\alpha}-\kappa-L-\cX_i-a)^{2}+\eta^{2}} \sum_{l}^{(i)}\ww^{(i)}_{\alpha}(l)\mathring{a}_{li} \right).
			\end{align*}
			Thus, it is enough to estimate
			\begin{align}
				\frac{f}{\sqrt{N}}\sum_{\alpha}\frac{\eta\langle \ee_{N-1}, \ww^{(i)}_{\alpha} \rangle}{(\nu^{(i)}_{\alpha}-\kappa-L-\cX_i-a)^{2}+\eta^{2}} \sum_{l}^{(i)}\ww^{(i)}_{\alpha}(l)\mathring{a}_{li}.
			\end{align}
			We have $|\langle \ee_{N-1}, \ww^{(i)}_{\alpha} \rangle|\le1$, $|\ww^{(i)}_{\alpha}(l)|\prec N^{-1/2}$. Moreover, from the large deviation estimate \cite[Lemma 3.8 (ii)]{EKYY13},
			we obtain
			\begin{align*}
				\sum_{l}|\mathring{a}_{l}| = N\E|\mathring{a}| + \sum_{l}(|\mathring{a}_{l}|-\E|\mathring{a}_{l}|) \prec q \asymp f.
			\end{align*}
			Note also that
			$$
			\sum_{\alpha}\frac{\eta}{(\nu^{(i)}_{\alpha}-\kappa-L-\cX_i-a)^{2}+\eta^{2}} = \Im(m_{A^{(i)}}(\tz_{i}-a)).
			$$
			It follows from the local law that
			\begin{align}
				\frac{f}{\sqrt{N}}\sum_{\alpha}\frac{\eta\langle \ee_{N-1}, \ww^{(i)}_{\alpha} \rangle}{(\nu^{(i)}_{\alpha}-\kappa-L-\cX_i-a)^{2}+\eta^{2}} \sum_{l}^{(i)}\ww^{(i)}_{\alpha}(l)\mathring{a}_{li} \prec \frac{f^{2}}{N}\frac{1}{N\eta}.
			\end{align}
			Similarly, the second statement follows from
			\begin{align}
				\frac{f}{\sqrt{N}}\sum_{\alpha}\frac{\eta\langle \ee_{N-1}, \ww^{(i)}_{\alpha} \rangle}{(\nu^{(i)}_{\alpha}-\kappa-L-\cX_i-a)^{2}+\eta^{2}} \sum_{l}^{(i)}\ww^{(i)}_{\alpha}(l)\frac{f}{N} \prec \frac{f^{2}}{N}\frac{1}{N\eta},
			\end{align}
			as claimed.    \end{proof}
	\end{lem}
	
	\begin{lem}\label{lem: bound on off diagonal A}
		Assume $q\gg 1$ and $1/(N\eta) \le \alpha \le 1/q $. We have on $\Omega(\alpha)$,
		\begin{align*}
			& |Z_i| \prec N^{\eps}\left(\frac{1}{q}  + \sqrt{\frac{\alpha}{N\eta}} \right), \quad |Z_{ij} | \prec N^\eps \left( \frac{1}{q^{2}} + \sqrt{\frac{\alpha}{N\eta}}\right)\\
			&	|\Im(Z_{i})| \prec N^\eps \left( \frac{\alpha}{q} + \sqrt{\frac{\alpha}{N\eta}}\right) ,
			\quad |\Im(Z_{ij}) | \prec N^\eps \left( \frac{\alpha}{q} + \sqrt{\frac{\alpha}{N\eta}}\right).
		\end{align*}
	\end{lem}
	\begin{proof}
		From the resolvent identity, we have 
		$$
		R^{(i)} - \tilde R^{(i)}  = - (\tz - \tz_i) R^{(i)}  \tilde R^{(i)} .
		$$
		Moreover,  
		$$
		|\tz - \tz_i| \prec \frac{1}{Nq},
		$$
		and on $\Omega(\alpha)$, from \eqref{eq:OmRT A}, for any $k,l$
		$$
		\left| \left(R^{(i)} \tilde R^{(i)}\right)_{kl} \right| \leq \sum_{a}^{(i)} | R^{(i)}_{ka}    \tilde R^{(i)}_{al} | \leq \sum_{a}^{(i)} \left( (R^{(i)}_{ka})^{2} + (\tilde R^{(i)}_{al})^{2} \right)| \lesssim \frac{N^{\eps}\alpha}{\eta}.
		$$
		So finally, 
		\begin{equation}\label{eq:diffRzzi A}
			\left| R_{kl}^{(i)} - \tilde R_{kl}^{(i)} \right| \prec \frac{N^{\eps}\alpha}{qN\eta}.
		\end{equation}
		The same bound holds for $\left| R^{(ij)}_{kl}  - \tilde R^{(ij)}_{kl} \right|$.
		
		We write
		\begin{multline*}
			Z_{i}(\tz)= \sum_{k}^{(i)}\left( |a_{ik}|^{2} - \frac{1}{N} \right)\tilde R_{kk}^{(i)} + \sum_{k\neq l}^{(i)}a_{ik} \tilde R_{kl}^{(i)} a_{li} \\
			+ \sum_{k}^{(i)}\left( |a_{ik}|^{2} - \frac{1}{N} \right) \left(R_{kk}^{(i)} - \tilde R_{kk}^{(i)}\right) + \sum_{k\neq l}^{(i)}a_{ik} \left(R_{kl}^{(i)} - \tilde R_{kl}^{(i)}\right) a_{li}.
		\end{multline*}
		Note that
		\begin{align*}
			\left| \sum_{k}^{(i)}\left( |a_{ik}|^{2} - \frac{1}{N} \right) \left(R_{kk}^{(i)} - \tilde R_{kk}^{(i)}\right) \right| &\prec \frac{N^{\eps}\alpha}{qN\eta} \sum_{k}^{(i)}\left| |a_{ik}|^{2} - \frac{1}{N} \right| \\
			&\lesssim  \frac{N^{\eps}\alpha}{qN\eta} \sum_{k}^{(i)} \left(  \mathring a_{ik}^{2} + \frac{f}{N}|\mathring a_{ik}| + \frac{f^{2}}{N^{2}} + \frac{1}{N} \right) \\
			&\lesssim \frac{N^{\eps}\alpha}{N\eta},
		\end{align*}
		and
		\begin{align*}
			\left| \sum_{k\neq l}^{(i)}a_{ik} \left(R_{kl}^{(i)} - \tilde R_{kl}^{(i)}\right) a_{li} \right| &\prec \frac{N^{\eps}\alpha}{qN\eta} \sum_{k\neq l}^{(i)}\left| a_{ik}a_{li} \right| \\
			&\lesssim  \frac{N^{\eps}\alpha}{qN\eta} \sum_{k\neq l}^{(i)} \left(  |\mathring a_{ik} \mathring a_{li}| + \frac{f}{N}|\mathring a_{ik}| + \frac{f}{N}|\mathring a_{li}| + \frac{f^{2}}{N^{2}} \right) \\
			&\lesssim \frac{N^{\eps}q\alpha}{N\eta}.
		\end{align*}
		Then we have
		\begin{align*}
			|Z_{i}| \prec \left|\sum_{k}^{(i)}\left( |a_{ik}|^{2} - \frac{1}{N} \right)\tilde R_{kk}^{(i)} + \sum_{k\neq l}^{(i)}a_{ik}\tilde R_{kl}^{(i)} a_{li}\right| + \frac{N^{\eps}q\alpha}{N\eta}.
		\end{align*}
		Using the large deviation estimate \cite[Lemma 3.8 (ii)]{EKYY13}, it follows that
		\begin{align*}
			\left|\sum_{k}^{(i)}\left( \mathring a_{ik}^{2} - \frac{1}{N} \right)\tilde R_{kk}^{(i)} + \sum_{k\neq l}^{(i)}\mathring a_{ik}\tilde R_{kl}^{(i)} \mathring a_{li}\right| \prec  \frac{\max_{k}|\tilde R_{kk}^{(i)}|}{q} + \frac{\max_{k\neq l}|\tilde R_{kl}^{(i)}|}{q} + \left(\frac{1}{N^{2}}\sum_{k,l}^{(ij)}|\tilde R_{kl}^{(i)}|^{2}\right)^{1/2}.
		\end{align*}
		Applying \cite[Lemma 7.5]{EKYY13} and \cite[Inequality (7.18)]{EKYY13}, we find
		\begin{align*}
			\left| \sum_{k, l}^{(i)}\left( \frac{f}{N}\mathring a_{ik} + \frac{f}{N}\mathring a_{li} + \frac{f^{2}}{N^{2}} \right)\tilde R_{kl}^{(i)} \right| \lesssim \frac{1}{q} + \frac{1}{N\eta}.
		\end{align*}
		In sum, we establish on $\Omega(\alpha)$,
		\begin{align*}
			|Z_{i}| \prec N^{\eps} \left(  \frac{1}{q} + \sqrt{\frac{\alpha}{N\eta}}\right),
		\end{align*}
		where we have used Ward identity \eqref{eq:ward}.
		The first claim follows.

		Similarly, since for $i\neq j$, the random variables $\{h_{ik}\}_{k : k \ne j}$ are independent of $\{h_{lj}\}_{l: l \ne i}$, from \eqref{eq:OmRT A}-\eqref{eq:diffRzzi A}, \cite[Lemma 3.8 (iii)]{EKYY13},  \cite[Lemma 7.5]{EKYY13} and \cite[Inequality (7.18)]{EKYY13}, on $\Omega(\alpha)$, we have
		\begin{align*}
			\left|\sum_{k,l}^{(ij)}Z_{ij} \right| &\prec  \frac{\max_{k}|\tilde R_{kk}^{(ij)}|}{q^{2}} + \frac{\max_{k\neq l}|\tilde R_{kl}^{(ij)}|}{q} + \left(\frac{1}{N^{2}}\sum_{k,l}^{(ij)}|\tilde R_{kl}^{(ij)}|^{2}\right)^{1/2} 
			+ \frac{1}{q} + \frac{1}{N\eta}
			+ \frac{N^{\eps}q\alpha}{N\eta} \nn\\
			&\prec  N^{\eps} \left( \frac{1}{q^{2}} + \sqrt{\frac{\alpha}{N\eta}}\right).
		\end{align*}
		The same argument gives with aid of Lemma \ref{lem: additional resolvent estimate A}, on $\Omega(\alpha)$, 
		\begin{align*}
			\left|\Im(Z_{ij}) \right| =  \left| \sum_{k,l}^{(ij)}a_{ik}\Im ( R_{kl}^{(ij)})a_{lj}\right| \prec N^\eps \left(\frac{\alpha}{q} + \sqrt{\frac{\alpha}{N\eta}}\right).
		\end{align*}
		
		Finally, we obtain similarly, on $\Omega(\alpha)$
		\begin{align*}
			\Im Z_{i} & = \sum_{k}^{(i)}\left( |a_{ik}|^{2} - \frac{1}{N} \right) \Im  R_{kk}^{(i)} + \sum_{k\neq l}^{(i)}a_{ik}\Im ( R_{kl}^{(i)})a_{li}			\prec N^\eps \left( \frac{\alpha}{q} + \sqrt{\frac{\alpha}{N\eta}}\right),
		\end{align*}
		as claimed.
	\end{proof}
	
	Following Step 5 (iteration of the error bounds) of Subsection \ref{sec: estimate on the imaginary part of local law}, we can complete the proof with the above technical lemmas. We omit the details.
\end{proof}
	
\end{lem}

\end{document}